\documentclass{article}
\usepackage[a4paper]{geometry}
\usepackage{amssymb}
\usepackage{systeme}
\usepackage[english]{babel}
\usepackage{color}
\usepackage{amsmath}
\usepackage{amsthm}
\usepackage{graphicx}
\usepackage{geometry}
\usepackage{comment}
\usepackage[normalem]{ulem}

\usepackage[urlcolor=black]{hyperref}
\hypersetup{
    colorlinks=true, 
    linktoc=all,     
    linkcolor=black,  
    citecolor=black
}

\theoremstyle{definition}
\newtheorem{theorem}{Theorem}[section]
\newtheorem{definition}[theorem]{Definition}

\newtheorem{corollary}[theorem]{Corollary}
\newtheorem{lemma}[theorem]{Lemma}
\newtheorem{remark}[theorem]{Remark}
\newtheorem{proposition}[theorem]{Proposition}

\newtheorem{conjecture}[theorem]{Conjecture}

\title{S-duality and the universal isometries of q-map spaces}
\date{\small Department of Mathematics\\
University of Hamburg\\
Bundesstraße 55, D-20146 Hamburg, Germany}
\author{Vicente Cort\'es and Iv\'an Tulli}
\numberwithin{equation}{section}

\begin{document}

\maketitle
\begin{abstract}
    The tree-level q-map assigns to a projective special real (PSR)  manifold of dimension $n-1\geq 0$, a quaternionic K\"{a}hler (QK) manifold of dimension $4n+4$. It is known that the resulting QK manifold admits a $(3n+5)$-dimensional universal group of isometries (i.e.\ independently of the choice of PSR manifold). On the other hand, in the context of Calabi-Yau compactifications of type IIB string theory, the classical hypermultiplet moduli space metric is an instance of a tree-level q-map space, and it is known from the physics literature that such a metric has an  $\mathrm{SL}(2,\mathbb{R})$ group of isometries related to the $\mathrm{SL}(2,\mathbb{Z})$ S-duality symmetry of the full 10d theory. We present a purely mathematical proof that any tree-level q-map space admits such an $\mathrm{SL}(2,\mathbb{R})$ action by isometries, enlarging the previous universal group of isometries to a $(3n+6)$-dimensional group $G$.  As part of this analysis, we describe how the $(3n+5)$-dimensional subgroup interacts with the $\mathrm{SL}(2,\mathbb{R})$-action,
    and find a codimension one normal subgroup of $G$ that is unimodular. 
   By taking a quotient with respect to a lattice in the unimodular group, we obtain a quaternionic K\"ahler manifold
    fibering over a projective special real manifold with fibers of finite 
    volume, and compute the volume as a 
    function of the base.
    We furthermore provide a mathematical treatment of results from the physics literature concerning the twistor space of the tree-level q-map space and the holomorphic lift of the $(3n+6)$-dimensional group of universal isometries to the twistor space. 
\end{abstract}

\textit{Keywords: quaternionic K\"{a}hler metrics, isometry groups,  q-map spaces, 
S-duality.}\\

\textit{MSC class: 53C26}
\tableofcontents

\begin{section}{Introduction}
In the context of type IIB string theory, S-duality is an $\mathrm{SL}(2,\mathbb{Z})$-symmetry of the full $10$d-theory that relates strong coupling effects with weak coupling effects. When compactifying the $10$d-theory on a Calabi-Yau $3$-fold $X$, one obtains a $4$d effective theory with a moduli space $\mathcal{M}_{\text{4d}}^{\text{IIB}}(X)$. Furthermore, $\mathcal{M}_{\text{4d}}^{\text{IIB}}(X)$ is supposed to inherit an S-duality action by isometries when all or certain quantum corrections are included \cite[Section 2.2]{Sduality}. The moduli space $\mathcal{M}_{\text{4d}}^{\text{IIB}}(X)$ has the structure

\begin{equation}
    \mathcal{M}_{\text{4d}}^{\text{IIB}}(X)=\mathcal{M}_{\text{VM}}^{\text{IIB}}(X)\times \mathcal{M}_{\text{HM}}^{\text{IIB}}(X)\,,
\end{equation}
where $\mathcal{M}_{\text{VM}}^{\text{IIB}}(X)$ is known as the vector multiplet moduli space, and $\mathcal{M}_{\text{HM}}^{\text{IIB}}(X)$ as the hypermultiplet moduli space. $\mathcal{M}_{\text{VM}}^{\text{IIB}}(X)$ has the structure of a projective special K\"{a}hler (PSK) manifold, while $\mathcal{M}_{\text{HM}}^{\text{IIB}}(X)$ has the structure of a quaternionic K\"{a}hler (QK) manifold.  S-duality then descends to an action on $\mathcal{M}_{\text{4d}}^{\text{IIB}}(X)$, acting trivially on $\mathcal{M}_{\text{VM}}^{\text{IIB}}(X)$ and non-trivially on $\mathcal{M}_{\text{HM}}^{\text{IIB}}(X)$, so we focus on the later. When one considers the classical geometry of $\mathcal{M}_{\text{HM}}^{\text{IIB}}(X)$ (i.e.\ the geometry obtained by dropping all quantum corrections in the string coupling constant $g_s$, and also taking the large volume limit), it has been shown in the physics literature that S-duality acts by isometries \cite{BGHL}, and can be even enhanced to an action of $\mathrm{SL}(2,\mathbb{R})$ by isometries. An argument using the twistor space associated to $\mathcal{M}_{\text{HM}}^{\text{IIB}}(X)$ can be found in \cite{APSV,Sduality}, while a direct argument can be found for example in \cite[Section 2.2]{ABL}. Furthermore, a discussion of the isometries of q-map spaces from the supergravity literature is given in  \cite{sugra1,sugra2}. Our intention, however, is to do a purely differential geometric treatment, avoiding arguments from supergravity or string theory, and hence making it accessible to the mathematician with training in differential geometry.  \\

 Our first objective in this paper is to do a purely mathematical treatment of the S-duality action by isometries on any tree-level q-map space. More precisely, we give a proof that any tree-level q-map metric admits an $\mathrm{SL}(2,\mathbb{R})$-action by isometries, and show how it interacts with the previously known universal group of isometries from the mathematics literature \cite[Appendix A]{CDJL}. When the tree-level q-map metric is realized by the classical geometry of $\mathcal{M}_{\text{HM}}^{\text{IIB}}(X)$ in string theory, the $\mathrm{SL}(2,\mathbb{R})$-action corresponds to the action induced by S-duality.  Our proof that the $\mathrm{SL}(2,\mathbb{R})$-action acts by isometries is done by showing that the infinitesimal action is given by Killing fields; together with the simplifying Lemma \ref{simplifyinglemma}, which allows us to reduce the problem to showing that the Killing field property holds along a certain low dimensional submanifold of the tree-level q-map space.  The original argument of \cite{BGHL} relies on analysing the 4d low energy effective actions of type IIA and IIB string theory, which we want to avoid in our purely differential geometric argument. On the other hand, the direct argument of \cite[Section 2.2]{ABL} relies on two facts, namely that the S-duality action on the classical $\mathcal{M}_{\text{HM}}^{\text{IIB}}(X)$ is isometric, and that the pull-back
 of the classical IIB metric by the mirror map  coincides with the classical type IIA metric (which has the usual tree-level q-map metric form reviewed in Section \ref{review}). The calculation establishing the latter fact is however omitted, which is why we chose to present our alternative explicit arguments. \\

A q-map space is obtained by composing two constructions. First we start with a projective special real (PSR) manifold $(\mathcal{H},g_{\mathcal{H}})$ and construct a projective special K\"{a}hler (PSK) manifold $(\overline{M},g_{\overline{M}})$ via the supergravity r-map, see \cite{CHM,CDJL} and references therein. On the other hand, the (one-loop corrected) supergravity c-map takes a PSK manifold  $(\overline{M},g_{\overline{M}})$ and produces a 1-parameter family $(\overline{N}_c,g_{\text{FS}}^c)$ of QK metrics \cite{RLSV,QKPSK}, known as the 1-loop corrected Ferrara-Sabharval metric. The composition of the two previous constructions produces what is known as a (one-loop corrected) q-map space. In the following we take the 1-loop parameter $c=0$, and hence focus only on $(\overline{N},g_{\overline{N}}):=(\overline{N}_0,g_{\text{FS}}^0)$ obtained from PSR manifolds $(\mathcal{H},g_{\mathcal{H}})$ via the q-map. The case $c=0$ is the so-called tree-level case, and hence $(\overline{N},g_{\overline{N}})$ is refered to as a tree-level q-map metric. Our plan for the first part of the paper will then consist of the following:

\begin{itemize}
        \item In Section \ref{review} we review the universal isometries for tree-level q-map spaces that are known in the mathematics literature \cite[Appendix A]{CDJL}. If $\text{dim}(\mathcal{H})=n-1\geq 0$, then the tree-level q-map space $(\overline{N},g_{\overline{N}})$ has $\text{dim}(\overline{N})=4n+4$. Any such $(\overline{N},g_{\overline{N}})$ has a group acting by isometries of the form $L_2:=\text{Iwa}(\mathrm{SU}(1,n+2))$, where $\text{Iwa}(\mathrm{SU}(1,n+2))$ denotes the solvable Iwasawa subgroup of $\mathrm{SU}(1,n+2)$. Furthermore,  $(\overline{N},g_{\overline{N}})$ has a group acting by isometries of the from $L:=L_1\ltimes_{\varphi_h}L_2$. Here $L_1=\mathbb{R}_{>0}\ltimes \mathbb{R}^n$ comes from isometries of the PSK manifold in the image of the r-map, and $\varphi_h:L_1\to \text{Aut}(L_2)$ is a certain homomorphism depending on the cubic polynomial $h(t)$ defining the PSR manifold. The group $L$ has dimension $3n+5$.
    \item In Section \ref{sdualitysection} we define the S-duality $\mathrm{SL}(2,\mathbb{R})$-action on $\overline{N}$. To define the action, it will be convenient to perform a diffeomorphism of $\overline{N}$ that in the string theory setting corresponds to the classical mirror map \cite{BGHL}. We remark that even though we are not in the string theory setting (and hence there are no compactified theories on mirror Calabi-Yau's to relate), the formulas defining the classical mirror map still make sense as a diffeomorphism of $\overline{N}$.\\
    
     With the help of the $3n+5$-dimensional universal group $L$ of isometries,  we show that the S-duality action of $\mathrm{SL}(2,\mathbb{R})$ is generated by Killing vectors and is thus isometric. More precisely, let $e,f,h$ be the usual generators of the Lie algebra $\mathfrak{sl}(2,\mathbb{R})$ satisfying 
    
    \begin{equation}\label{sl2comm}
    [e,f]=h, \;\;\;\; [h,e]=2e, \;\;\;\; [h,f]=-2f\,,
    \end{equation}
    and let $X_e$, $X_f$ and $X_h$ be the corresponding vector fields on $\overline{N}$, representing the infinitesimal S-duality action. $X_e$ turns out to be a Killing vector field of any c-map space. On the other hand, the fact that $X_f$ and $X_h$ are Killing really uses the fact that we are a on a tree-level q-map space. For example, it will be easy to see that setting the $1$-loop parameter to $c\neq 0$ breaks the isometry generated by $X_h$ (and hence the one by $X_f$, by (\ref{sl2comm})).   
    \item The main results of Section \ref{sdualitysection} are the following:
    
\textbf{Theorems \ref{theorem} and \ref{theorem2}}: Let $(\overline{N},g_{\overline{N}})$ be a tree-level q-map space associated to the PSR manifold $(\mathcal{H},g_{\mathcal{H}})$ with $\text{dim}(\mathcal{H})=n-1\geq 0$. Then S-duality acts by isometries, and there is a $(3n+6)$-dimensional connected Lie-group $G$ of isometries  containing the isometries of the $L$-action and S-duality.  The Lie algebra of $G$ is isomorphic to

\begin{equation}
    \mathbb{R}\ltimes (\mathfrak{sl}_2(\mathbb{R})\ltimes (\mathbb{R}^n\ltimes \mathfrak{h}))
\end{equation}
with the brackets described in Section \ref{uniS}. The lie algebra $\mathfrak{h}$ denotes a certain real codimension $1$ sub-algebra of the Lie algebra $\mathfrak{heis}_{2n+3}(\mathbb{R})$ of the Heisenberg group.

Furthermore, in Corollary \ref{isoHcor} we consider how $G$ is enlarged by the isometries obtained from the automorphisms of the PSR manifold $\mathcal{H}$.

The other main result from this section is the following:

 \textbf{Theorem \ref{finvolfibers}}: Consider a tree-level q-map space $(\overline{N},g_{\overline{N}})$ associated to a PSR manifold $(\mathcal{H},h)$ with 

\begin{equation}
    h(t^a)=\frac{1}{6}k_{abc}t^at^bt^c, \quad k_{abc}\in \mathbb{Z}\,.
\end{equation}
Furthermore, recall the nilradical $\mathfrak{n}\subset \mathfrak{g}$ from Corollary \ref{nilradicalcor}. Then there is a lattice $\Gamma$ of the normal and unimodular codimension 1 subgroup $\mathrm{SL}(2,\mathbb{R})\ltimes \exp(\mathfrak{n})\subset G$, acting by isometries on  $(\overline{N},g_{\overline{N}})$. The quotient gives a fiber bundle $\overline{N}/\Gamma \to \mathcal{H}$ with fibers of finite volume.\\

In Section \ref{volumedensityappendix} we compute the fiber-wise volume density of $\overline{N}\to \mathcal{H}$, and explicitly study the volume growth of the fibers of $\overline{N}/\Gamma \to \mathcal{H}$ in the cases where $\mathcal{H}$ is a maximal PSR curve. This will allow us to conclude that $\overline{N}/\Gamma$ is of finite volume only when $\mathcal{H}$ is the incomplete maximal PSR curve. In the case of a complete
PSR curve with arclength parameter $s\in (-\infty , \infty)$ we show the following dichotomy holds. Either the sublevel and superlevel sets of 
$s : \overline{N}/\Gamma \to \mathbb{R}$ are both of infinite volume or, only the sublevel sets are (for the appropriate orientation of 
$\mathcal{H}$). The latter case occurs precisely when the PSR curve is 
homogeneous.
\end{itemize}

In the second part of the paper we revisit and do a mathematical treatment of results from the physics literature \cite{NPV, APSV, Sduality} concerning the twistor space $(\mathcal{Z},\mathcal{I},\lambda,\tau)$ of $(\overline{N},g_{\overline{N}})$ and the holomorphic lifts of the universal isometries to $(\mathcal{Z},\mathcal{I},\lambda,\tau)$ (here $\mathcal{I}$ denotes the complex structure of $\mathcal{Z}$, $\lambda$ the holomorphic contact distribution, and $\tau$ the real structure \cite{Salamon1982}). More specifically, in Section \ref{lifttwistorspacesec} we discuss the following:

\begin{itemize}
    \item Any tree-level q-map space $(\overline{N},g_{\overline{N}})$ lies in the image of the HK/QK correspondence, and hence has an associated (pseudo-)hyperk\"{a}hler (HK) manifold $(N,g_{N},I_1,I_2,I_3)$, together with a circle bundle $P\to N$ with a hyperholomorphic connection $\eta$. If $(\overline{N},g_{\overline{N}})$ is any QK manifold in the image of the HK/QK correspence, we will describe the twistor space $(\mathcal{Z},\mathcal{I},\lambda,\tau)$ of $(\overline{N},g_{\overline{N}})$ in terms of the ``HK-data" $(N,g_{N},I_1,I_2,I_3)$ and $(P \to N,\eta)$. This description has the advantage that the HK data is usually easier to describe and handle than the QK data. On the other hand, since the description works for any QK space obtained via HK/QK correspondence, it holds not only for the tree-level q-map case, but also any c-map space, or any c-map space with ``mutually local instanton corrections" \cite{CT}\footnote{In the physics literature this is also expected to extend to the case of non-mutually local D-instanton corrections \cite{WCHKQK}, but a mathematically rigorous construction of such metrics is not currently understood.}.
    \item We will then restrict the previous twistor description to the case where $(\overline{N},g_{\overline{N}})$ is a tree-level c-map space, and discuss certain holomorphic Darboux coordinates for the contact structure $\lambda$ found in \cite{NPV}. It is clear that the local contact distribution  defined by those coordinates coincides with $\ker \lambda$, but it is a bit more tricky to show that they are actually holomorphic with respect to the holomorphic structure of $\mathcal{Z}$. This difficulty originates from the fact that $\lambda \in \Omega^1(\mathcal{Z},\mathcal{L})$ is valued in a holomorphic line bundle $\mathcal{L}\to \mathcal{Z}$, and when discussing holomorphic Darboux coordinates for $\lambda$ one needs to make sure one is working in a holomorphic trivialization of $\mathcal{L}\to \mathcal{Z}$.
    \item On the other hand, any group action by isometries of a QK manifold lifts canonically to a holomorphic action on its twistor space, preserving the holomorphic contact distribution and commuting with the real structure \cite{auttwistor}.  If $(\overline{N},g_{\overline{N}})$ is a tree-level q-map metric, this means that the $(3n+6)$-dimensional group of isometries $G$ lifts to the twistor space. This in particular includes the lift of S-duality previously studied in \cite{APSV,Sduality}. We finish Section \ref{lifttwistorspacesec} by explicitly discussing the lift of the $G$-action to the twistor space.
\end{itemize}

 Finally, in Section \ref{outlook} we present a quick outlook on what is known in the physics literature regarding the S-duality action on the quantum corrected versions of $\mathcal{M}_{\text{HM}}^{\text{IIB}}(X)$ (see for example \cite{QMS, Sduality, Sduality2, D31, D32, D33}), and also state future directions regarding the mathematical treatment of such quantum corrected cases.\\

\textbf{Acknowledgements:} This work was supported by the Deutsche Forschungsgemeinschaft (German Research Foundation) under Germany’s Excellence Strategy -- EXC 2121 ``Quantum Universe'' -- 390833306.  The authors would like to thank Murad Alim, Alessio Marrani, Arpan Saha, J\"{o}rg Teschner, Danu Thung, Timo Weigand and Alexander Westphal for helpful  discussions. 
\end{section}

\begin{section}{Review of q-map metrics and their universal group of isometries}\label{review}

In this section we review the supergravity r-map, c-map and q-map. At each step we also discuss the universal groups of isometries that are known in the mathematics literature.
\begin{subsection}{The r-map}

The r-map produces a projective special K\"{a}hler (PSK) manifold $(\overline{M},g_{\overline{M}})$ from a projective special real (PSR) manifold $(\mathcal{H},g_{\mathcal{H}})$. We recall the corresponding definitions below. Our main references will be \cite{CHM,CDJL}.

\begin{definition}\label{PSR:def}
A projective special real (PSR) manifold is a Riemannian manifold $(\mathcal{H},g_{\mathcal{H}})$ such that $\mathcal{H}\subset \mathbb{R}^{n}$ is a hypersurface, and there is a homogeneous cubic polynomial $h:\mathbb{R}^{n}\to \mathbb{R}$ satisfying

\begin{itemize}
    \item $\mathcal{H}\subset \{t \in \mathbb{R}^{n}\;\; | \;\; h(t)=1\}$.
    \item $g_{\mathcal{H}}=-\partial^2h|_{T\mathcal{H}\times T\mathcal{H}}$.
\end{itemize}
We will denote the coordinates of $\mathbb{R}^n$ by $t^a$, where $a=1,2,...,n$. In particular, we will write $h$ as follows:

\begin{equation}
    h(t^a)=\frac{1}{6}k_{abc}t^at^bt^c
\end{equation}
where the coefficients $k_{abc} \in \mathbb{R}$ are symmetric in the indices. 
\end{definition}
\begin{definition} A conical affine special K\"{a}hler (CASK) manifold is a tuple $(M,g_M,\omega_M,\nabla,\xi)$ such that
\begin{itemize}
    \item $(M,g_M,\omega_M)$ is a pseudo-K\"{a}hler manifold. We denote its complex structure by $J$ (i.e. $g_M(J-,-)=\omega_M(-,-)$).
    \item $\nabla$ is a flat, torsion-free connection on $M$ satisfying $\nabla \omega_M=0$ and $d_{\nabla}J=0$, where $d_{\nabla}: \Omega^n_M(TM)\to \Omega^{n+1}_M(TM)$ is the natural extension of $\nabla$, and we think of $J$ as an element of $\Omega^1_M(TM)$.
    \item $\nabla \xi=D\xi=\text{Id}$, where $D$ is the Levi-Civita connection and $\text{Id}$ is the identity map on $TM$.
    \item $g_M$ is negative definite on $\mathcal{D}=\text{Span}\{\xi,J\xi\}$ and positive definite on $D^{\perp}$.
\end{itemize}

\end{definition}

Assume that $\xi$ generates a free $\mathbb{R}_{>0}$-action on $M$ and that $J\xi$ generates a free $S^1$-action on $M$. Then $\mu=\frac{1}{2}g_M(\xi,\xi)$ gives a moment map for the $S^1$-action, and by performing the K\"{a}hler quotient $\overline{M}:=\mu^{-1}(-1/2)/S^1$ (recall that $g_{M}(\xi,\xi)<0$), we obtain a K\"{a}hler manifold $(\overline{M},g_{\overline{M}})$.

\begin{definition}
The K\"{a}hler manifold $(\overline{M},g_{\overline{M}})$ obtained via the previous K\"{a}hler quotient is called a projective special K\"{a}hler (PSK) manifold. 
\end{definition}

An easy way to obtain PSK manifolds is via the notion of CASK domains \cite{ACD,CDS}:

\begin{definition}\label{defCASKdomain}
A CASK domain is a tuple $(M,\mathfrak{F})$ such that

\begin{itemize}
    \item $M\subset \mathbb{C}^{n+1}\backslash\{0\}$ is a $\mathbb{C}^{\times}$-invariant domain (with respect to the usual $\mathbb{C}^{\times}$-action on $\mathbb{C}^{n+1}\backslash\{0\}$ by multiplication).
    \item $\mathfrak{F}:M\to \mathbb{C}$ is a holomorphic function, homogeneous of degree $2$ with respect to the $\mathbb{C}^{\times}$-action.
    \item With respect to the natural holomorphic coordinates $(X^0,...,X^n)$ of $M$, the matrix
    
    \begin{equation}
        \tau_{ij}:=\Big(\frac{\partial^2\mathfrak{F}}{\partial X^i\partial X^j}\Big)
    \end{equation}
    satisfies that $\text{Im}(\tau_{ij})$ has signature $(n,1)$ and $\text{Im}(\tau_{ij})X^i\overline{X}^j<0$ for $X\in M$.
    
\end{itemize}
\end{definition}
From this data, one obtains a CASK manifold $(M,g_M,\omega_M,\nabla,\xi)$ by 

\begin{equation}
    g_M=\text{Im}(\tau_{ij})dX^id\overline{X}^j, \;\;\; \omega_M=\frac{i}{2}\text{Im}(\tau_{ij})dX^i\wedge d\overline{X}^j, \;\;\;\; \xi= X^i\partial_{X^i}+\overline{X}^i\partial_{\overline{X}^i}\,.
\end{equation}
The flat connection $\nabla$ is defined such that $dx^i:=\text{Re}(dX^i)$ and $dy_i:=-\text{Re}\Big(\frac{\partial \mathfrak{F}}{\partial X^i}\Big)$ is a flat frame of $T^*M$.\\

In the case of a CASK domain, $\xi$ and $J\xi$ generate a free $\mathbb{C}^{\times}$-action on $M$, so we can perform the K\"{a}hler quotient from before and obtain a PSK manifold $(\overline{M},g_{\overline{M}})$.\\

Now consider a PSR manifold $(\mathcal{H},g_{\mathcal{H}})$ defined by the real cubic homogeneous polynomial $h$, and let $U=\mathbb{R}_{>0}\cdot \mathcal{H}\subset \mathbb{R}^{n}\backslash \{0\}$. We define $\overline{M}:=\mathbb{R}^n+iU\subset \mathbb{C}^{n}$ with the canonical holomorphic structure, where global holomorphic coordinates are given by $z^a:=b^a+it^a \in \mathbb{R}^{n}+iU$. On $\overline{M}$ we consider the metric

\begin{equation}\label{PSKrmap}
    g_{\overline{M}}=\frac{\partial^2\mathcal{K}}{\partial z^a   \partial \overline{z}^b}dz^ad\overline{z}^b
\end{equation}
where 
\begin{equation}\label{PSKrmappot}
    \mathcal{K}:=-\log K(t), \;\;\;\; K(t):=8h(t)=\frac{4}{3}k_{abc}t^at^bt^c, \;\;\;\;t=(t^a)=(t^1,\ldots,t^n).
\end{equation}
Then it can be shown that $(\overline{M},g_{\overline{M}})$ is a PSK manifold  \cite{CHM}. Its corresponding CASK manifold is defined via the CASK domain $(M,\mathfrak{F})$, where $M\subset \mathbb{C}^{n+1}$ is given by 

\begin{equation}
    M:=\{ (X^0,...,X^n)  = X^0\cdot (1,z)\in \mathbb{C}^{n+1} \;\;\; | \;\;\; X^0 \in \mathbb{C}^{\times}, \;\; z\in \overline{M}\}
\end{equation}
and 
\begin{equation}\label{holprep}
    \mathfrak{F}(X)=-\frac{h(X^1,...,X^n)}{X^0}=-\frac{1}{6}k_{abc}\frac{X^aX^bX^c}{X^0}, \;\;\;\;X=(X^i)=(X^0,\ldots ,X^n).
\end{equation}

\begin{remark}
In the following, the indices labeled by $i$, $j$, $k$,... range from $0$ to $n$, while those labeled by $a$, $b$, $c$... range from $1$ to $n$.
\end{remark}

The relation between the coordinates $X^i$ on $M$ and the coordinates $z^a$ of $\overline{M}$ is given by

\begin{equation}
    \frac{X^i}{X^0}=z^i, \;\;\;\; z^0:=1\,.
\end{equation}

\begin{definition}
The construction described above, which associates the PSK manifold 
$(\overline{M},g_{\overline{M}})$ with the PSR manifold $(\mathcal{H},g_{\mathcal{H}})$,  
is called the r-map. The corresponding inclusion map $(\mathcal{H},g_{\mathcal{H}})\to (\overline{M},g_{\overline{M}})$ is also called the r-map. 
A similar double use of terminology applies to the c-map and q-map discussed below.
\end{definition}
\end{subsection}
\begin{subsection}{The universal group \texorpdfstring{$L_1=\mathbb{R}_{>0}\ltimes \mathbb{R}^n$}{TEXT} of isometries of r-map spaces}\label{uniisormap}

From (\ref{PSKrmap}) and (\ref{PSKrmappot}), one finds the following explicit expression for a PSK metric in the image of the r-map \begin{equation}\label{exprmapmetric}
    g_{\overline{M}}= -\frac14 \frac{\partial^2 \log h(t) }{\partial t^a \partial t^b} (db^adb^b + dt^adt^b)=\Big(-\frac{k_{abc}t^c}{4h(t)} +\frac{k_{acd}k_{bef}t^ct^dt^et^f}{(4h(t))^2}\Big)(db^adb^b + dt^adt^b)\,,
\end{equation}
where we recall that $z^a=b^a+it^a$. If $\text{dim}_{\mathbb{R}}(\mathcal{H})=n-1$, it then follows that $g_{\overline{M}}$ has a group $\mathbb{R}^n$ of isometries acting by translations in the $b^a$ coordinates. Furthermore, there is a group $\mathbb{R}_{>0}$ of isometries acting by scalings $z^a \to \lambda z^a$. Combining the two, we obtain the group $L_1:=\mathbb{R}_{>0}\ltimes \mathbb{R}^n$ of isometries. This group is independent of the original PSR manifold.\\

We remark that the map $\mathbb{R}_{>0}\times \mathcal{H}\to U$ given by $(r,p^a)\to rp^a$ is a diffeomorphism, allowing us to identify $\overline{M}=\mathbb{R}^{n}+iU\cong \mathcal{H}\times L_1$. With respect to this identification we can rewrite $(\ref{exprmapmetric})$ as 

\begin{equation}\label{rmapdecomp}
    \begin{split}
    g_{\overline{M}}&=\frac{1}{4} g_{\mathcal{H}} + \frac{3}{4}\frac{dr^2}{r^2}   +\frac{1}{r^2}\Big(-\frac{k_{abc}p^c}{4} +\frac{k_{acd}k_{bef}p^cp^dp^ep^f}{4^2}\Big)db^adb^b,\quad p \in \mathcal{H}\,\\
    &=\frac{1}{4}g_{\mathcal{H}}+g_{L_1}(p)
    \end{split}
\end{equation}
where $g_{L_1}(p)$ is a family of left-invariant metrics on $L_1$ parametrized by $p\in \mathcal{H}$.

\end{subsection}
\begin{subsection}{The c-map}

The supergravity c-map assigns a $1$-parameter family of QK metrics to a PSK manifold \cite{RLSV,QKPSK}. We describe the supergravity c-map in the case that the PSK manifold $(\overline{M},g_{\overline{M}})$ comes from a CASK domain $(M,\mathfrak{F})$, since this will be the case of interest to us.\\

Let $\overline{N}=\overline{M}\times \mathbb{R}_{>0}\times \mathbb{R}^{2n+3}$, where $n=\text{dim}_{\mathbb{C}}(\overline{M})$. We consider global coordinates on $\overline{N}$ given by $(z^a,\rho,\widetilde{\zeta}_i,\zeta^i,\sigma) \in \overline{M}\times \mathbb{R}_{>0}\times \mathbb{R}^{2n+2}\times \mathbb{R}$, where $i=0,...,n$. Furthermore, given $c\in \mathbb{R}$ we denote $\overline{N}_c:=\{\rho>-2c\}\subset \overline{N}$, and define the metric $g_{\text{FS}}^c$ on $\overline{N}_c$ by

\begin{equation}\label{1-loopmetric}
    g_{\text{FS}}^c:=\frac{\rho+c}{\rho}g_{\overline{M}}+\frac{\rho+2c}{4\rho^2(\rho+c)}d\rho^2 -\frac{1}{4\rho}\Big(N^{ij}-\frac{2(\rho+c)}{\rho K}z^i\overline{z}^j\Big)W_i\overline{W}_j +\frac{\rho+c}{64 \rho^2(\rho+2c)}(d\sigma + \widetilde{\zeta}_{i}d\zeta^{i}-\zeta^{i}d\widetilde{\zeta}_{i}-4cd^c\mathcal{K})^2\,,
\end{equation}
where $N_{ij}:=-2\text{Im}(\tau_{ij})$; $N^{ij}$ is the inverse matrix; $W_i:=d\widetilde{\zeta}_i-\tau_{ij}d\zeta^{j}$; $K=N_{ij}z^i\overline{z}^j$; $\mathcal{K}=-\log(K)$; and $d^c=i(\overline{\partial}-\partial)$. It is shown in \cite[Theorem 5]{QKPSK} that the metric $(\overline{N}_c,g_{\text{FS}}^c)$ is QK and positive definite.

\begin{remark}
We are using a slightly different convention from \cite{QKPSK} in the normalization of the variables. Our conventions are such that the scalings $\sigma \to -4\sigma$, $\widetilde{\zeta}_i\to 2\widetilde{\zeta}_i$, $\zeta^i \to 2\zeta^i$ takes the expression (\ref{1-loopmetric}) to the one in \cite[Equation 4.11]{QKPSK} (one also needs \cite[Equation 4.12]{QKPSK} for a direct comparison). Notice that there is a difference in the sign of $\tau_{ij}$ (and hence $N_{ij}$ and $K$), due to the opposite convention in the signature of the starting CASK domain, with respect to \cite{QKPSK}.
\end{remark}

In the following, we set $c=0$, and focus on the simplified tree-level metric $(\overline{N},g_{\overline{N}}):=(\overline{N}_0,g_{\text{FS}}^0)$. In this case (\ref{1-loopmetric}) reduces to

\begin{equation}\label{QKtree}
    g_{\overline{N}}=g_{\overline{M}}+\frac{d\rho^2}{4\rho^2}-\frac{1}{4\rho}(N^{ij}-\frac{2}{K}z^i\overline{z}^j)W_i\overline{W}_j + \frac{1}{64\rho^2}(d\sigma + \widetilde{\zeta}_{i}d\zeta^{i}-\zeta^{i}d\widetilde{\zeta}_{i})^2\,.
\end{equation}

\begin{definition}
Let $(M,\mathfrak{F})$ be a CASK domain and $(\overline{M}, g_{\overline{M}})$ its PSK manifold. Given $c\in \mathbb{R}$, the map $(\overline{M},g_{\overline{M}})\to (\overline{N}_c,g_{\text{FS}}^c)$ described above is called the (1-loop corrected) supergravity c-map. If we set $c=0$, we call the map $(\overline{M},g_{\overline{M}})\to (\overline{N},g_{\overline{N}})$ the tree-level supergravity c-map.
\end{definition}

\end{subsection}
\begin{subsection}{The universal group \texorpdfstring{$L_2=\text{Iwa}(SU(1,n+2))$}{TEXT} of isometries of tree-level c-map spaces}\label{isocmap}

Consider a c-map space $(\overline{N},g_{\text{FS}}^c)$  with $\text{dim}_{\mathbb{R}}(\overline{N})=4n+4$. Furthermore, let $L_2:=\text{Iwa}(SU(1,n+2))$ be the solvable Iwasawa subgroup of $SU(1,n+2)$. The group $L_2$ is diffeomorphic to $\mathbb{R}_{>0}\times \mathbb{R}^{2n+3}$ with the group structure given by \cite{CHM}:
\begin{equation}
    (r,\widetilde{\eta}_i,\eta^i,\kappa)\cdot (\rho,\widetilde{\zeta}_i,\zeta^i,\sigma)=(r\rho, \widetilde{\eta}_i + \sqrt{r}\widetilde{\zeta}_i, \eta^i+\sqrt{r}\zeta^i,\kappa +r \sigma + \sqrt{r}(\eta^i\widetilde{\zeta}_i-\widetilde{\eta}_i\zeta^i))\,.
\end{equation}
Furthermore, $L_2$ contains the Heisenberg group $\text{Heis}_{2n+3}(\mathbb{R})\cong \{1\}\times \mathbb{R}^{2n+3}\subset L_2$ as its nilradical.\\

Letting $L_2$ act on $(z^a,\rho,\widetilde{\zeta}_i,\zeta^i,\sigma) \in \overline{N}$ by acting trivially on $z^a\in \overline{M}$, it is easy to check that it acts by isometries on any tree level c-map metric $g_{\overline{N}}=g_{\text{FS}}^0$; while $\text{Heis}_{2n+3}(\mathbb{R})\subset L_2$ acts by isometries on any c-map metric $g_{\text{FS}}^c$.\\

We denote the generators of the infinitesimal action of $L_2$ on $\overline{N}$ by
\begin{equation}\label{infiwa}
    \begin{split}
        D=\rho \partial_{\rho}+\frac{1}{2}\widetilde{\zeta}_i\partial_{\widetilde{\zeta}_i}+\frac{1}{2}\zeta^i\partial_{\zeta^i}+\sigma \partial_{\sigma},\;\;\; P_i=\partial_{\zeta^i}+\widetilde{\zeta}_i\partial_{\sigma}, \;\;\;X^i=\partial_{\widetilde{\zeta}_i}-\zeta^i\partial_{\sigma}, \;\;\; Z=\partial_{\sigma}.
    \end{split}
\end{equation}
In particular, $P_i$, $X^i$ and $Z$ correspond to the infinitesimal action of $\text{Heis}_{2n+3}(\mathbb{R})$. The non-trivial commutation relations among $D$, $P_i$, $X^i$ and $Z$ are given by:
\begin{equation}
    [P_i,X^j]=-2\delta_i^jZ, \;\;\;\; [X^i,D]=\frac{X^i}{2}, \;\;\;\; [P_i,D]=\frac{P_i}{2}, \;\;\;\; [Z,D]=Z\,.
\end{equation}

As in the r-map case, we have $\overline{N}=\overline{M}\times L_2$, and we see from (\ref{QKtree}) that we can write

\begin{equation}\label{cmapdecomp}
    g_{\overline{N}}=g_{\overline{M}} +g_{L_2}(z),
\end{equation}
where $g_{L_2}(z)$ is a family of left-invariant metrics on $L_2$, parametrized by $z\in \overline{M}$.

\end{subsection}
\begin{subsection}{The q-map}

\begin{definition}
The q-map is the composition of the r-map and the supergravity c-map. That is, the composition of $(\mathcal{H},g_{\mathcal{H}})\to (\overline{M},g_{\overline{M}})\to (\overline{N}_c,g_{\text{FS}}^c)$ for some $c\in \mathbb{R}$. If the take $c=0$, we call the corresponding space $(\overline{N}_0,g_{\text{FS}}^0)$ a tree-level q-map space.
\end{definition}

We describe the special structure that some terms of $g_{\text{FS}}^c$ acquire when $g_{\text{FS}}^c$ is in the image of the q-map.  Besides the specific form of the PSK metric term $g_{\overline{M}}$ given in (\ref{exprmapmetric}), we have that the matrices $\tau_{ij}$ and $N_{ij}$ have the special form

\begin{equation}\label{qtau}
    (\tau_{ij})=\begin{pmatrix}
    -\frac{1}{3}k_{abc}z^az^bz^c & \frac{1}{2}k_{abc}z^bz^c\\
     \frac{1}{2}k_{abc}z^bz^c & -k_{abc}z^c\\
    \end{pmatrix}
\end{equation}
and 
\begin{equation}\label{qN}
    (N_{ij})=\begin{pmatrix}        -4h(t) +2k_{abc}b^ab^bt^c & -2k_{abc}b^bt^c\\
     -2k_{abc}b^bt^c   &2k_{abc}t^c\\
    \end{pmatrix}\,,
\end{equation}
where in the above expressions, we have used our previous conventions on the indices $i,j$ and $a,b,c$. In particular, the first entries are the $i=j=0$ components, the off-diagonal elements are the $i=0,j=a$ or $i=a,j=0$ components, and the lower block on the diagonal consists of the $i=a,j=b$ components. The identities (\ref{qtau}) and (\ref{qN}) follow from the specific form of the holomorphic prepotential $\mathfrak{F}$ given in (\ref{holprep}), together with the formula $N_{ij}=-2\text{Im}(\tau_{ij})$.\\

The inverse $N^{ij}$ can be easily computed, and it is given by 

\begin{equation}\label{qNinv}
    (N^{ij})=-\frac{1}{4h(t)}\begin{pmatrix}1 & b^a\\
    b^a & b^ab^b - (G^{-1})^{ab}\\
    \end{pmatrix}, \;\;\;\;\;\;\;\; G_{ab}=\frac{1}{2h(t)}k_{abc}t^c\,.
\end{equation}
The invertibility of $G_{ab}$ follows from the fact that $-\partial^2h$ defines a Lorentzian metric on $U=\mathbb{R}_{>0}\cdot \mathcal{H}$, together with the fact that $h(t)>0$ on $U$.\\

The special structure of the terms (\ref{qtau}), (\ref{qN}), (\ref{qNinv}), will be crucial for the results that we wish to prove.

\end{subsection}
\begin{subsection}{The universal group \texorpdfstring{$L=L_1\ltimes_{\varphi_h}L_2$}{TEXT} of isometries of tree-level q-map metrics}\label{isoqmap}

Let $(\overline{N},g_{\overline{N}})$ be a tree-level q-map space of dimension $4n+4$ (i.e. corresponding to a PSR manifold $\mathcal{H}$ with $\text{dim}_{\mathbb{R}}(\mathcal{H})=n-1\geq 0$). \\

Because $(\overline{N},g_{\overline{N}})$ is in particular a tree-level c-map space, it has the group $L_2=\text{Iwa}(SU(1,n+2))$ acting by isometries. Now let $\varphi_h:L_1 \to \mathrm{Sp}(\mathbb{R}^{2n+2})\subset \text{Aut}(\text{Iwa}(SU(1,n+2)))$ be the homomorphism defined as follows:

\begin{itemize}
    \item For $\lambda \in \mathbb{R}_{>0}\subset L_1$, we let $\varphi_h(\lambda)\cdot (\rho,\zeta^0,\zeta^a,\widetilde{\zeta}_0,\widetilde{\zeta}_a,\sigma)=(\rho, \lambda^{-3/2}\zeta^0, \lambda^{-1/2}\zeta^a, \lambda^{3/2}\widetilde{\zeta}_0,\lambda^{1/2}\widetilde{\zeta}_a,\sigma)$.
    \item For $v \in \mathbb{R}^n$, we have
    \begin{equation} \label{vact}
        \varphi_h(v)\cdot \begin{pmatrix}
        \rho \\
        \zeta^0 \\
        \zeta^a\\
        \widetilde{\zeta}_0\\
        \widetilde{\zeta}_a\\
        \sigma \\
        \end{pmatrix}=\begin{pmatrix}
        \rho \\
        \zeta^0 \\
        \zeta^a +\zeta^0v^a\\
        \widetilde{\zeta}_0+\frac{1}{6}k_{abc}v^av^bv^c\zeta^0 + \frac{1}{2}k_{abc}v^av^b\zeta^c -\widetilde{\zeta}_av^a\\
        \widetilde{\zeta}_a -\frac{1}{2}\zeta^0k_{abc}v^bv^c - k_{abc}v^b\zeta^c\\
        \sigma \\
        \end{pmatrix}
    \end{equation}
    In particular, it is easy to check that $\varphi_h(v)$ acts by an $\mathrm{Sp}(2n+2)$ transformation on $(\zeta^i,\widetilde{\zeta}_i)$, and that $\varphi_h(v+v')=\varphi_h(v)\circ \varphi_h(v')$. We remark that in the formula (\ref{vact}) above, the terms with $k_{abc}$ differ in sign with respect to \cite{CDJL}. This is again due to the opposite conventions in the signature of the CASK manifold, with respect to \cite{CDJL}.
    \item For $(\lambda,v)\in L_1$ one then defines $\varphi_h((\lambda,v))=\varphi_h((1,v)\cdot(\lambda,0)):=\varphi_h(v)\circ \varphi_h(\lambda)$. It is straightforward to check that $\varphi_h$ is a homomorphism. 
\end{itemize}
We then obtain the group $L=L_1\ltimes_{\varphi_h}\text{Iwa}(SU(1,n+2))$, and it acts by isometries on the tree-level q-map metric $(\overline{N},g_{\overline{N}})$, as shown in \cite{CDJL}.\\

The Killing vectors corresponding to the action of the $\mathbb{R}^n$ factor in $L_1$ are given by

\begin{equation}\label{vdef}
    V_a:=\partial_{b^a}+\zeta^0\partial_{\zeta^a}-\widetilde{\zeta}_a\partial_{\widetilde{\zeta}_0}-k_{abc}\zeta^c\partial_{\widetilde{\zeta}_b},
\end{equation}
while the Killing vector corresponding to the $\mathbb{R}_{>0}$ factor of $L_1$ is given by

\begin{equation}\label{infrdil}
    D_1=b^a\partial_{b^a}+t^a\partial_{t^a}-\frac{3}{2}\zeta^0\partial_{\zeta^0}-\frac{1}{2}\zeta^a \partial_{\zeta^a}+\frac{1}{2}\widetilde{\zeta}_a\partial_{\widetilde{\zeta}_a}+\frac{3}{2}\widetilde{\zeta}_0\partial_{\widetilde{\zeta}_0}.
\end{equation}

The non-trivial brackets among the infinitesimal generators $D_1$, $V_a$, $D$, $P_i$, $X^i$, $Z$ of the $L$-action on $\overline{N}$ are given by:

\begin{equation}
    \begin{split}
            [P_i,X^j]&=-2\delta_i^jZ, \;\;\;\; [X^i,D]=\frac{X^i}{2}, \;\;\;\; [P_i,D]=\frac{P_i}{2}, \;\;\;\; [Z,D]=Z,\\
            [P_0,V_a]&=P_a,\;\;\;\; [V_a,P_b]=k_{abc}X^c,\;\;\;\; [V_a,X^b]=\delta_{a}^bX^0,\;\;\;\; [V_a,D_1]=V_a,\\
          [P_0,D_1]&=-\frac{3}{2}P_0,\;\;\;\; [P_a,D_1]=-\frac{1}{2}P_a,\;\;\;\; [X^0,D_1]=\frac{3}{2}X^0,\;\;\;\; [X^a,D_1]=\frac{1}{2}X^a\,.
    \end{split}
\end{equation}

Combining previous remarks for the c-map and r-map, we find  that in the tree-level q-map case we have $\overline{N}=\overline{M}\times L_2\cong \mathcal{H}\times L$, and that we can write $g_{\overline{N}}$ using (\ref{rmapdecomp}) and (\ref{cmapdecomp}) as

\begin{equation}\label{qmapdecomp}
    g_{\overline{N}}=\frac{1}{4}g_{\mathcal{H}}+g_{L}(p)\,,
\end{equation}
where $g_L(p)$ is a family of left-invariant metrics on $L$, parametrized by $p\in \mathcal{H}$.

\end{subsection}
\end{section}
\begin{section}{S-duality for tree-level q-map spaces}\label{sdualitysection}

Let $(\overline{N},g_{\overline{N}})$ be a tree-level q-map space. We will first define two sets of global coordinates on $\overline{N}$, usually called in the physics literature the type IIA and type IIB coordinates. These coordinates are related by the so-called ``classical mirror map"\footnote{We use quotes because there is a priori no underlying Calabi-Yau manifold, and hence no relation between moduli of mirror Calabi-Yau's.}, which in our case will just be a certain diffeomorphism of $\overline{N}$. We then define the S-duality $\mathrm{SL}(2,\mathbb{R})$-action on $\overline{N}$, following the usual action of S-duality on the type IIB fields \cite{BGHL}.  We will describe how the resulting $\mathrm{SL}(2,\mathbb{R})$ action interacts with the universal group of isometries $L$, and use this to show that the S-duality $\mathrm{SL}(2,\mathbb{R})$-action acts by isometries on any tree-level q-map metric. We then finish the section showing that one can combine the S-duality isometries with the $L$-action isometries into a Lie group $G$ of universal isometries for tree-level q-map spaces. 

\begin{subsection}{The type IIA and type IIB variables}

We start by discussing two systems of global coordinates on $\overline{N}$, the so-called type IIA and type IIB coordinates. While the tree-level q-map metrics and the action of the universal group of isometries $L=L_1\ltimes_{\varphi_h}L_2$ are more naturally expressed in terms of type IIA coordinates, the S-duality action is more naturally described in terms of type IIB coordinates.

\begin{definition}\label{defcoords} We define two systems of global coordinates on $\overline{N}=\overline{M}\times \mathbb{R}_{>0}\times \mathbb{R}^{2n+2}\times \mathbb{R}$:
\begin{itemize}
    \item The global coordinates $(z^a,\rho,\widetilde{\zeta}_i,\zeta^i,\sigma) \in \overline{M}\times \mathbb{R}_{>0}\times \mathbb{R}^{2n+2}\times \mathbb{R}$ in which $g_{\overline{N}}$ is expressed in (\ref{1-loopmetric}) will be called type IIA coordinates.
    \item On the other hand, by isolating one of the $\mathbb{R}$-factors of $\mathbb{R}^{2n+2}\subset \overline{N}$ and combining it with the factor $\mathbb{R}_{>0}\subset \overline{N}$ as the upper half-plane $H:=\mathbb{R}+i\mathbb{R}_{>0}$ we can consider another set of global coordinates on $\overline{N}$ given by $(b^a+it^a,\tau_1+i\tau_2,c^a,c_a,c_0,\psi)\in \overline{M}\times H\times \mathbb{R}^{n}\times \mathbb{R}^n\times \mathbb{R}\times \mathbb{R}=\overline{N}$, related to the type IIA coordinates via the so-called classical mirror map \cite{BGHL,Sduality}:
    \begin{equation}\label{MM}
    \begin{split}
     \rho&=\frac{\tau_2^2}{2}h(t), \;\;\;\;\;\;\;  z^a=b^a+it^a,  \;\;\;\;\; \zeta^0=\tau_1, \;\;\;\;\; \zeta^a=-(c^a-\tau_1b^a),\\
            \widetilde{\zeta}_a&=c_a +\frac{k_{abc}}{2}b^b(c^c-\tau_1b^c), \;\;\;\;\;\;  \widetilde{\zeta}_0=c_0-\frac{k_{abc}}{6}b^ab^b(c^c-\tau_1b^c)\\
            \sigma&= -2(\psi + \frac{1}{2}\tau_1c_0) + c_a(c^a-\tau_1b^a) -\frac{k_{abc}}{6}b^ac^b(c^c-\tau_1b^c)\,,
    \end{split}
\end{equation}
where we recall that $h(t)$ is the cubic polynomial defining the PSR manifold:
\begin{equation}
    h(t)=\frac{1}{6}k_{abc}t^at^bt^c\,.
\end{equation}
The coordinates $(b^a+it^a,\tau_1+i\tau_2,c_0,c_a,c^a,\psi)$ related to $(z^a,\rho,\widetilde{\zeta}_i,\zeta^i,\sigma)$ via (\ref{MM}) will be called type IIB variables. 

\end{itemize}
\end{definition}

\begin{remark}
It is not hard to check that (\ref{MM}) is a diffeomorphism between the two global coordinates. Indeed, the inverse can be written explicitly as

\begin{equation}
    \begin{split}
            \tau_2=&\sqrt{2h^{-1}(t)\rho}, \;\;\;\;\;\;\;  b^a+it^a=z^a,  \;\;\;\;\; \tau_1=\zeta^0, \;\;\;\;\; c^a=-(\zeta^a-\zeta^0b^a)\\
            c_a&=\widetilde{\zeta}_a +\frac{k_{abc}}{2}b^b\zeta^c, \;\;\;\;\;\;  c_0=\widetilde{\zeta}_0-\frac{k_{abc}}{6}b^ab^b\zeta^c\\
            \psi&= -\frac{\sigma}{2}-\frac{\zeta^0}{2}(\widetilde{\zeta_0}-\frac{k_{abc}}{6}b^ab^b\zeta^c) -\frac{\zeta^a}{2}(\widetilde{\zeta}_a+\frac{k_{abc}}{2}b^b\zeta^c) -\frac{k_{abc}}{12}b^a\zeta^b(\zeta^c-\zeta^0b^c)\,.
    \end{split}
\end{equation}
In the first equation, we have used that $h(t)>0$ on $\overline{M}$, so that $\tau_2=\sqrt{2h^{-1}(t)\rho}$ is valued in $\mathbb{R}_{>0}$.
\end{remark}
\end{subsection}
\begin{subsection}{The S-duality action}
Consider $\overline{N}$ with the global type IIB variables $(b^a+it^a,\tau_1+i\tau_2,c^a,c_a,c_0,\psi)$ from Definition \ref{defcoords}. In these variables, we define the following $\mathrm{SL}(2,\mathbb{R})$-action \cite{BGHL,Sduality}:

\begin{equation}\label{sl2can}
    \begin{split}
    \tau&:=\tau_1+i\tau_2 \to \frac{a\tau +b}{c\tau +d},\;\;\;\;\; t^a \to |c\tau +d|t^a, \;\;\;\; c_a \to c_a, \\
    \begin{pmatrix}c^a\\ b^a\end{pmatrix}&\to \begin{pmatrix} a & b \\ c & d\\ \end{pmatrix}\begin{pmatrix}c^a \\ b^a\end{pmatrix}, \;\;\;\; \begin{pmatrix}c_0\\ \psi\end{pmatrix}\to \begin{pmatrix} d & -c \\ -b & a\\ \end{pmatrix}\begin{pmatrix}c_0 \\ \psi \end{pmatrix}, \;\;\;\; \begin{pmatrix} a & b \\ c & d\\ \end{pmatrix}\in \mathrm{SL}(2,\mathbb{R})
    \end{split}
\end{equation}

\begin{definition}\label{slcan}
 The $\mathrm{SL}(2,\mathbb{R})$-action defined by (\ref{sl2can}) will be called the S-duality action.
\end{definition}

\begin{proposition}\label{fiberwiseSduality} With respect to the decomposition $\overline{N}\cong \mathcal{H}\times L$, and the corresponding decomposition of the metric $g_{\overline{N}}=\frac{1}{4}g_{\mathcal{H}} + g_{L}(p)$ from (\ref{qmapdecomp}), the S-duality action is a fiber-wise action with respect to the projection into the first factor 
\begin{equation}
    \overline{N} \to \mathcal{H},\,\quad  (z^a=b^a+it^a,\rho,\widetilde{\zeta}_i,\zeta^i,\sigma)\mapsto  \left( \frac{t^a}{h(t)^{\frac13}}\right) ,
\end{equation}
leaving the $g_{\mathcal{H}}$ summand of $g_{\overline{N}}$ invariant.

\end{proposition}

\begin{proof}
To prove that S-duality acts fiber-wise with respect to $\overline{N}\to \mathcal{H}$, it is enough to 
observe that in type IIA variables it only acts on the $t^a$-coordinates by an overall rescaling.
This is obvious from (\ref{sl2can}) together with the fact that the classical mirror map (\ref{MM}) leaves $t^a$ invariant.\\

On the other hand, since S-duality acts trivially on the $\mathcal{H}$-factor of $\overline{N}\cong \mathcal{H}\times L$, it follows that it leaves the $\frac{1}{4}g_\mathcal{H}$ summand of $g_{\overline{N}}$ invariant. 
\end{proof}

We will now describe the infinitesimal S-duality action, both in type IIA and type IIB variables. We do this as follows: denote by $e,f,h\in \mathfrak{sl}(2,\mathbb{R})$ the generator of the Lie algebra $\mathfrak{sl}(2,\mathbb{R})$ given by
\begin{equation}
    e=\begin{pmatrix}
    0 & 1\\
    0 & 0\\
    \end{pmatrix}, \;\;\;  f=\begin{pmatrix}
    0 & 0\\
    1 & 0\\
    \end{pmatrix}, \;\;\; 
    h=\begin{pmatrix}
    1 & 0\\
    0 & -1\\
    \end{pmatrix} \;\;\; 
\end{equation}satisfying

\begin{equation}\label{sl2rliealg}
    [e,f]=h, \;\;\;\; [h,e]=2e, \;\;\;\; [h,f]=-2f\,.
\end{equation}
The infinitesimal action in type IIB variables is then described by the following lemma:
\begin{lemma}\label{infsdualityIIB}
Let $Y_e$, $Y_f$ and $Y_h$ be the generators of the infinitesimal S-duality action in type IIB variables. Then:

\begin{equation}\label{infsdualityiibexp}
    \begin{split}
        Y_e&=\partial_{\tau_1}+b^a\partial_{c^a}-c_0\partial_{\psi}\\
        Y_f&=(\tau_2^2-\tau_1^2)\partial_{\tau_1}-2\tau_1\tau_2\partial_{\tau_2}+\tau_1t^a\partial_{t^a}+c^a\partial_{b^a}-\psi \partial_{c_0}\\
        Y_h&=2\tau_1\partial_{\tau_1}+2\tau_2\partial_{\tau_2}-t^a\partial_{t^a}-b^a\partial_{b^a}+c^a\partial_{c^a}-c_0\partial_{c_0}+\psi\partial_{\psi}
    \end{split}
\end{equation}
\end{lemma}
\begin{proof}
This follows easily from a direct computation using the definition of the S-duality action in type IIB variables given in (\ref{sl2can}).
\end{proof}

Because we want to see how S-duality interacts with the universal isometries $L$, and the latter are more naturally described in terms of type IIA variables, it will be convenient to describe S-duality in type IIA variables (at least infinitesimally). This amounts to conjugating the action (\ref{sl2can}) in type IIB variables by the mirror map (\ref{MM}) relating type IIA and type IIB variables. Given the complexity of the mirror map, the S-duality action in type IIA variables turns out to be  quite complicated. We however can describe quite concisely the infinitesimal S-duality action in type IIA variables as follows:

\begin{proposition}\label{infS} Let $X_e$, $X_f$ and $X_h$ be the generators of the infinitesimal S-duality action on $\overline{N}$ expressed in type IIA variables. Then:

\begin{equation}\label{infSduality}
    \begin{split}
        X_e&=\partial_{\zeta^0}+\widetilde{\zeta_0}\partial_{\sigma}\\
     X_f&=(2\rho h(t)^{-1}-(\zeta^0)^2)\partial_{\zeta^0} + (2\rho h(t)^{-1}b^a-\zeta^0\zeta^a)\partial_{\zeta^a}+ \frac{k_{abc}}{2}(\zeta^b\zeta^c-2b^bb^c\rho h(t)^{-1})\partial_{\widetilde{\zeta}_a}\\
        &+\Big(\frac{1}{2}(\sigma + \zeta^i\widetilde{\zeta}_i)+\frac{k_{abc}}{3}\rho h(t)^{-1}b^ab^bb^c \Big)\partial_{\widetilde{\zeta}_0}+\zeta^0t^a\partial_{t^a} +(\zeta^0b^a-\zeta^a)\partial_{b^a} -\zeta^0\rho\partial_{\rho}\\
        &+\Big[2\rho h(t)^{-1}\Big(-\widetilde{\zeta}_0 -b^a\widetilde{\zeta}_a - k_{abc}\Big(\frac{b^ab^b\zeta^c}{2} - \frac{b^ab^bb^c\zeta^0}{6}\Big)\Big) -\frac{\zeta^0}{2}(\sigma-\widetilde{\zeta}_i\zeta^i) +\frac{k_{abc}}{6}\zeta^a\zeta^b\zeta^c\Big]\partial_{\sigma}\\
        X_h&=2\zeta^0\partial_{\zeta^0} +\zeta^a\partial_{\zeta^a}-\widetilde{\zeta}_0\partial_{\widetilde{\zeta}_0} + \sigma \partial_{\sigma}+\rho \partial_{\rho}-t^a\partial_{t^a}-b^a\partial_{b^a}\\
    \end{split}
\end{equation}

\end{proposition}

\begin{proof}
Let $\mathcal{M}:\overline{N}\to \overline{N}$ be the Mirror map from type IIB to type IIA variables described in (\ref{MM}). The expressions for $X_e$, $X_f$ and $X_h$ in terms of $Y_e$, $Y_f$ and $Y_h$ given in Lemma \ref{infsdualityIIB} can then be found via 

\begin{equation}
    d\mathcal{M}(Y_e|_{\mathcal{M}^{-1}(p)})=X_e|_p, \quad d\mathcal{M}(Y_f|_{\mathcal{M}^{-1}(p)})=X_f|_p, \quad d\mathcal{M}(Y_h|_{\mathcal{M}^{-1}(p)})=X_h|_p\,.
\end{equation}
In order to perform these computations, one then needs to compute the infinitesimal mirror map $d\mathcal{M}$. The relevant expressions are given by 

\begin{equation}\label{infMM}
        \begin{split}
        d\mathcal{M}(\partial_{\tau_2})&=\frac{\tau_2e^{-\mathcal{K}}}{8}\partial_{\rho}, \quad\quad          d\mathcal{M}(\partial_{t^a})=-\rho\partial_{t^a}\mathcal{K}\partial_{\rho}+\partial_{t^a}\\
        d\mathcal{M}(\partial_{\tau_1})&=\partial_{\zeta^0}+b^a\partial_{\zeta^a}-\frac{k_{abc}}{2}b^bb^c\partial_{\widetilde{\zeta}_a} +\frac{k_{abc}}{6}b^ab^bb^c\partial_{\widetilde{\zeta}_0} + \Big(-c_0 -c_ab^a+\frac{k_{abc}}{6}b^ab^bc^c\Big)\partial_{\sigma}\\
        d\mathcal{M}(\partial_{b^a})&=\partial_{b^a}+\tau_1\partial_{\zeta^a}+k_{abc}\Big(\frac{1}{2}c^c-\tau_1b^c\Big)\partial_{\widetilde{\zeta}_b}+k_{abc}\Big(-\frac{1}{3}b^bc^c+\frac{1}{2}b^bb^c\tau_1\Big)\partial_{\widetilde{\zeta}_0}\\
        &\;\;\;\;+\Big(-c_a\tau_1 + k_{abc}\Big(\frac{\tau_1b^bc^c}{3}-\frac{c^bc^c}{6}\Big)\Big)\partial_{\sigma}\\
        d\mathcal{M}(\partial_{c^a})&=-\partial_{\zeta^a}+\frac{k_{abc}}{2}b^c\partial_{\widetilde{\zeta}_b} -\frac{k_{abc}}{6}b^bb^c\partial_{\widetilde{\zeta}_0} +\Big(c_a -\frac{k_{abc}}{3}b^bc^c + \frac{k_{abc}}{6}b^bb^c\tau_1\Big)\partial_{\sigma}\\
        d\mathcal{M}(\partial_{c_a})&=\partial_{\widetilde{\zeta}_a}+c^a\partial_{\sigma}, \quad\quad 
        d\mathcal{M}(\partial_{c_0})=\partial_{\widetilde{\zeta}_0} -\tau_1\partial_{\sigma}, \quad\quad 
        d\mathcal{M}(\partial_{\psi})=-2\partial_{\sigma}\, .
    \end{split}
\end{equation}
The result then follows from a straightforward, but rather tedious, computation using (\ref{infMM}), (\ref{infsdualityiibexp}) and (\ref{MM}). More details on the computation are given in Appendix \ref{appA}.
\end{proof}

\begin{remark}
As can be seen from (\ref{infSduality}), while $X_e$ and $X_h$ are very simple, $X_f$ is highly non-trivial in the usual type IIA variables in which $g_{\overline{N}}$ is written. This complexity is expected, since in the string theory setting S-duality is non-manifest in type IIA variables. 
\end{remark}

\begin{proposition} \label{ehiso} $X_e$ and $X_h$ are Killing vector fields on $(\overline{N},g_{\overline{N}})$.

\end{proposition}

\begin{proof}
Notice that from Proposition \ref{infS} and equation (\ref{infiwa}) we find that $X_e=P_0$, where $P_0$ is one of the generators of the Heisenberg group action given in Section \ref{isocmap}. Furthermore, notice that
\begin{equation}
    X_h=D-D_1
\end{equation}
where $D$ is the generator of the dilation action of $\text{Iwa}(SU(1,n+2))$ given in (\ref{infiwa}), while $D_1$ is the generator of the dilation action coming from the PSR manifold isometries, given in (\ref{infrdil}). Since $P_0$, $D$ and $D_1$ are Killing fields, we conclude that $X_e$ and $X_h$ must also be Killing fields.
\end{proof}

\begin{remark}
Notice that $X_e$ is a Killing vector field of any c-map space (including the 1-loop deformations). However, the fact that $X_h$ is Killing really uses the fact that we are on the tree-level q-map space case. It is easy to check that $X_h$ (and hence $X_f$) is no longer a Killing field if we set the $1$-loop parameter to $c\neq 0$. 
\end{remark}

The fact that the remaining vector field $X_f$ is Killing in the tree-level case is harder to show, and will be proven in Section \ref{Siso}, using the results from the following section.  
\end{subsection}
\begin{subsection}{The universal group \texorpdfstring{$L=L_1\ltimes_{\varphi_h}L_2$}{TEXT} of isometries and S-duality}\label{uniS}
In this section we wish to show how S-duality interacts with the universal group of isometries $L=L_1\ltimes_{\varphi_h}L_2$ of $(\overline{N},g_{\overline{N}})$ given in Section \ref{isoqmap}. We will use this in the next section to show that $X_f$ is a Killing field, and hence that S-duality acts by isometries. \\

Recall the generators $D_1, D, V_a,  P_i, X^i, Z$ of the left action of $L$ on $\overline{N}$ given in Section \ref{isoqmap}, and the generators $X_e$, $X_f$, $X_h$ of the S-duality action given in (\ref{infSduality}). Furthermore, recall from Proposition \ref{ehiso} that
\begin{equation}\label{Xhdd1}
    X_e=P_0, \;\;\;\; X_h=D-D_1\,.
\end{equation}
We can therefore interchange the generator $D_1$ with $X_h$, and generate the infinitesimal $L$-action by $D, V_a, P_i, X^i,Z, X_h$.\\

\begin{proposition}\label{KValgebra}
    If $\text{dim}(\mathcal{H})=n-1\geq 0$, there is a $3n+6$ dimensional Lie algebra $\mathfrak{g}$ of vector fields on $\overline{N}$ generated by $D$, $X^i$, $P_a$, $P_0=X_e$, $X_f$, $X_h=D-D_1$, $V_a$, $Z$ where $a=1,...,n$ and $i=0,1,...,n$, satisfying the following commutation relations:
    
    \begin{itemize}
        \item $D$, $X^i$, $P_i$ and $Z$ have the non-trivial commutation relations:
        \begin{equation}\label{braketsiwasawa}
            [P_i,X^j]=-2\delta_i^jZ, \;\;\;\; [X^i,D]=\frac{X^i}{2}, \;\;\;\; [P_i,D]=\frac{P_i}{2}, \;\;\;\; [Z,D]=Z\,.
        \end{equation}
        \item $X_e$, $X_f$ and $X_h$ satisfy:
        \begin{equation}\label{sl2vectorbrakets}
        [X_e,X_f]=-X_h, \;\;\;\;         [X_h,X_e]=-2X_e, \;\;\;\; [X_h,X_f]=2X_f\,.
    \end{equation}
    In particular, they span a subalgebra   $\mathfrak{sl}(2,\mathbb{R})^{\text{op}}\subset  \mathfrak{g}$ 
    isomorphic to $\mathfrak{sl}(2,\mathbb{R})$. 
    Note that the linear map $\mathfrak{sl}(2,\mathbb{R})\rightarrow \mathfrak{sl}(2,\mathbb{R})^{\text{op}}$ which maps $(h,e,f)$ to $(X_h,X_e,X_f)$
    is an anti-isomorphism. Multiplying it with $-1$ yields an isomorphism.
    
    \item We have additionally the following non-trivial commutation relations:
    \begin{equation}
        \begin{split}
            &[P_0,V_a]=P_a, \;\;\; [V_a,P_b]=k_{abc}X^c, \;\;\; [V_a,X^b]=\delta_{a}^bX^0, \;\;\; [X_h,V_a]=V_a\\
            &[X_f,P_a]=V_a, \;\;\; [D,X_f]=\frac{X_f}{2}, \;\;\; [Z,X_f]=\frac{X^0}{2}\\
            & [P_a,X_h]=P_a, \;\;\; [X_h,X^0]=X^0, \;\;\; [Z,X_h]=Z\,.
        \end{split}
    \end{equation}
    In particular, the $V_a$ form an $n$-dimensional abelian subalgebra of Killing fields. It can be extended to an $n+1$ dimensional abelian subalgebra by adding either $X^0$, $D$, $X_f$ or $Z$; and to an $n+2$ dimensional abelian subalgebra by adding $X^0$ and $X_f$.

    \end{itemize}
\end{proposition}
\begin{proof}

Most of the brackets follow from the already known brackets from Sections \ref{isocmap} and \ref{isoqmap}, together with the identity (\ref{Xhdd1}). The brackets in (\ref{sl2vectorbrakets}) follow from the fact that $X_e$, $X_f$ and $X_h$ are the fundamental vector fields of the action by left-multiplication in $\mathrm{SL}(2, \mathbb{R})$, i.e.\ the right-invariant vector fields which evaluate to $e, f, h$ at the neutral element. The fact that they are right-invariant is responsible for the minus sign compared to the brakets of $e,f,h\in \mathfrak{sl}(2,\mathbb{R})$. \\

Since the only new part of the algebra is obtained by the addition of $X_f$, we include the computations of the brackets involving $X_f$ that are not obvious, namely: $[X_f,X^i]$, $[X_f,V^a]$, $[X_f,P_a]$, $[X_f,Z]$ and $[X_f,D]$.

\begin{itemize}
    \item $[X_f,X^i]$:
    \begin{equation}
        \begin{split}
        [X_f,X^0]&=-(2\rho h^{-1}-(\zeta^0)^2)\partial_{\sigma} - (-2\rho h^{-1}+(\zeta^0)^2)\partial_{\sigma}=0\\
        [X_f,X^a]&=-(2\rho h^{-1}b^a-\zeta^0\zeta^a)\partial_{\sigma} - (-2\rho h^{-1}b^a+\zeta^0\zeta^a)\partial_{\sigma}=0\,.
        \end{split}
    \end{equation}
    \item $[X_f,V^a]$:
    \begin{equation}
        \begin{split}
        X_f(V_a)&=(2\rho h^{-1}-(\zeta^0)^2)\partial_{\zeta^a}+\frac{k_{abc}}{2}(2b^bb^c\rho h^{-1}-\zeta^b\zeta^c)\partial_{\widetilde{\zeta}_0}+k_{abc}(\zeta^0\zeta^c-2\rho h^{-1}b^c)\partial_{\widetilde{\zeta}_b}\\
        V_a(X_f)&=X_f(V_a) \implies [X_f,V^a]=0\,.
        \end{split}
    \end{equation}
    \item $[X_f,P_a]$:
    \begin{equation}
        \begin{split}
        X_f(P_a)&=\frac{k_{abc}}{2}(\zeta^b\zeta^c-2b^bb^c\rho h^{-1})\partial_{\sigma}\\
        P_a(X_f)&= -V_a + \frac{k_{abc}}{2}(\zeta^b\zeta^c-2b^bb^c\rho h^{-1})\partial_{\sigma} \implies [X_f,P_a]=V_a\,.
        \end{split}
    \end{equation}
    \item $[Z,X_f]$:
    \begin{equation}
        [Z,X_f]= Z(X_f)-0=\frac{1}{2}(\partial_{\widetilde{\zeta}_0}-\zeta^0\partial_{\sigma})=\frac{X^0}{2}\,.
    \end{equation}
    \item $[D,X_f]$: from the fact that the functions $(z^a,\rho,\tilde{\zeta}_i,\zeta^i,\sigma)$ are eigenfunctions of $D$ with the respective eigenvalues $(0,1,\frac12, \frac12,1)$, we see that $X_f$ is an eigenvector of $\mathcal{L}_D$ for the eigenvalue $\frac12$,
    so that $[D,X_f]=X_f/2$. \qedhere
\end{itemize}
\end{proof}

Let $\mathfrak{g}$ be the Lie algebra of vector fields on $\overline{N}$ from Proposition \ref{KValgebra}. Let

\begin{equation}
    \mathfrak{h}:=\text{span}\{Z,P_a,X^i\}, \;\;\; \mathfrak{n}:=\text{span}\{V_a,Z,P_a,X^i\}, \quad \mathbb{R}^{n}:=\text{span}\{V_a\}\,.
\end{equation}
Then the commutation relations of Proposition \ref{KValgebra} show that $\mathfrak{n}$ is a subalgebra of $\mathfrak{g}$, and that $\mathfrak{h}\subset \mathfrak{n}$ is an ideal
with complementary abelian subalgebra $\mathbb{R}^n$, i.e.\ 

\begin{equation}
\mathfrak{n}=\mathbb{R}^n\ltimes \mathfrak{h}\,.
\end{equation}

Furthermore, we see that 
\begin{equation}
    \mathfrak{u}:=\text{span}\{X_e,X_f,X_h,V^a,Z,P_a,X^i\}
\end{equation}
is a subalgebra of $\mathfrak{g}$, which splits as
\begin{equation}
    \mathfrak{u}=\mathfrak{sl}(2,\mathbb{R})^{\text{op}}\ltimes \mathfrak{n}\,.
\end{equation}
Finally, it is easy to check that  $[D,\mathfrak{u}] \subset \mathfrak{u}$, such that 
\begin{equation}
    \mathfrak{g}=\mathbb{R}D\ltimes  \mathfrak{u}\,.
\end{equation}
We can summarize the results as follows.
\begin{corollary}\label{univliealgebra}
The Lie algebra $\mathfrak{g}$ of Proposition \ref{KValgebra} has the structure
\begin{equation}
    \mathfrak{g}=\mathbb{R}D\ltimes \mathfrak{u}=\mathbb{R}D\ltimes (\mathfrak{sl}_2(\mathbb{R})^{\text{op}}\ltimes \mathfrak{n})=\mathbb{R}D\ltimes (\mathfrak{sl}_2(\mathbb{R})^{\text{op}}\ltimes (\mathbb{R}^n\ltimes \mathfrak{h}))\,.
\end{equation}
\end{corollary}
We furthermore have:
\begin{corollary}\label{nilradicalcor} 
The subalgebra $\mathfrak{n}\subset \mathfrak{g}$ is the nilradical of $\mathfrak{g}$.
\end{corollary}
\begin{proof}
The fact that $\mathfrak{n}$ is an ideal follows immediately from the commutation relations of $\mathfrak{g}$ given in Proposition \ref{KValgebra}. On the other hand, the fact that it is nilpotent follows from the fact that $\mathfrak{n}_1=[\mathfrak{n},\mathfrak{n}]=\text{span}\{X^0,X^a,Z\}$, $\mathfrak{n}_2=[\mathfrak{n},\mathfrak{n}_1]=\text{span}\{Z,X^0\}$, $\mathfrak{n}_3=[\mathfrak{n},\mathfrak{n}_2]=\{0\}$. In the previous computation of $\mathfrak{n}_1$ we used that $\text{span}\{k_{abc}X^c\}=\text{span}\{X^c\}$, which follows from the non-degeneracy of the Hessian of the cubic polynomial, together with  \cite[Lemma 1]{BC}. Finally, to show that its maximal, assume that $\mathfrak{n'}\supset \mathfrak{n}$ is any ideal strictly bigger that $\mathfrak{n}$. It is not difficult to show using the commutation relations that either $\mathfrak{n}'$ contains $\mathfrak{sl}_2(\mathbb{R})^{\text{op}}$, in which case it cannot be nilpotent; or it does not contain $\mathfrak{sl}_2(\mathbb{R})^{\text{op}}$ and $\mathfrak{n}'=\mathfrak{n}\oplus \mathbb{R}(4D-X_{h})$,  which is neither nilpotent, since $[\mathfrak{n}',\mathfrak{n}']=\mathfrak{n}$ and $[\mathfrak{n}',\mathfrak{n}]=\mathfrak{n}$.

\end{proof}
\begin{corollary}
The Levi-Malcev decomposition of the Lie algebra $\mathfrak{g}$ of Proposition \ref{KValgebra}
is given by 
\[ \mathfrak{g} = \mathfrak{sl}_2(\mathbb{R})^{\text{op}} \ltimes 
(\mathbb{R}D' \ltimes \mathfrak{n}),\]
where 
\[ D' = 4D-X_h\in \mathbb{R}D\ltimes \mathfrak{sl}_2(\mathbb{R})^{\text{op}} = \mathbb{R}D'\oplus  \mathfrak{sl}_2(\mathbb{R})^{\text{op}}\cong \mathfrak{gl}_2(\mathbb{R}).\]
Here $\mathfrak{sl}_2(\mathbb{R})^{\text{op}}\cong \mathfrak{sl}_2(\mathbb{R})$ is the Levi subgroup and 
$\mathfrak{r} = \mathbb{R}D' \ltimes \mathfrak{n}$ is the radical, which happens to be 
non-unimodular and contains the nilradical $\mathfrak{n}$ as a codimension one ideal. 
\end{corollary}
\begin{proof} 
From the commutation relations of Proposition~\ref{KValgebra} we see that $D'$ commutes with $\mathfrak{sl}_2(\mathbb{R})^{\text{op}}$ and that the operator 
$\mathrm{ad}_{D'}|_{\mathfrak{n}}$ has the eigenspace decomposition 
$$\mathrm{span}\{ X^0 , Z\} \oplus  \mathrm{span}\{ X^a\} \oplus \mathrm{span}\{ V_a, P_a\} $$
with the respective eigenvalues $(-3,-2,-1)$. In particular, $\mathrm{tr}( \mathrm{ad}_{D'}) = -4n-6\neq 0$, which shows that 
the radical $\mathfrak{r}=\mathbb{R}D' \ltimes \mathfrak{n}$ is not unimodular.
\end{proof}

\end{subsection}
\begin{subsection}{The \texorpdfstring{$(3n+6)$}{TEXT}-dimensional group of universal isometries}\label{Siso}

In Proposition \ref{ehiso} we have shown that $X_e$ and $X_h$ are Killing vector fields. Here we wish to show that the remaining generator $X_f$ of S-duality is also a Killing vector field, which will imply that S-duality acts by isometries. We will later show that the $L$-action and S-duality combine into a $3n+6$-dimensional connected Lie group $G$ of universal isometries of $(\overline{N},g_{\overline{N}})$.\\

To show that $X_f$ is a Killing field, we will use some of the previous algebra structure found in Proposition \ref{KValgebra} to simplify the computation. In the following,  we consider the PSR manifold $\mathcal{H}$ as embedded in $\overline{N}=\mathbb{R}_{\rho>0}\times \overline{M}\times \mathbb{R}^{2n+3}$ via 

\begin{equation}
    p=(x^a) \in \mathcal{H} \to (1,0+ix^a,0)\in \mathbb{R}_{\rho>0}\times \overline{M}\times \mathbb{R}^{2n+3}=\overline{N}\,.
\end{equation}

\begin{lemma}\label{simplifyinglemma} Assume that we have shown that $\mathcal{L}_{X_f}g_{\overline{N}}(p)=0$ for $p\in \mathcal{H}\subset \overline{N}$. Then $\mathcal{L}_{X_f}g_{\overline{N}}=0$ on $\overline{N}$.
\end{lemma}
\begin{proof}

In the following, it will be convenient to consider the nested submanifolds $\mathcal{H}=\overline{N}_0\subset \overline{N}_1\subset \overline{N}_2\subset \overline{N}_3
=\overline{N}$ defined by

\begin{equation}
    \begin{split}
    \overline{N}_1&:=\{(\rho,b^a+it^a,\zeta^i,\widetilde{\zeta}_i,\sigma) \in \overline{N}\;\; | \;\; b^a=\zeta^i=\widetilde{\zeta}_i=\sigma=0, \;\;\; \rho=1\} \\
        \overline{N}_2&:=\{(\rho,b^a+it^a,\zeta^i,\widetilde{\zeta}_i,\sigma) \in \overline{N}\;\; | \;\; b^a=\zeta^i=\widetilde{\zeta}_i=\sigma=0\} 
    \end{split}
\end{equation}
so that by our assumptions $\mathcal{L}_{X_f}g_{\overline{N}}|
_{\overline{N}_0}=0$. We want to show that $\mathcal{L}_{X_f}g_{\overline{N}}|
_{\overline{N}_i}=0$ implies that $\mathcal{L}_{X_f}g_{\overline{N}}|
_{\overline{N}_{i+1}}=0$, with the final step then showing that $\mathcal{L}_{X_f}g_{\overline{N}}=0$. \\

With respect to the coordinates $(x^i)=(\rho,b^a+it^a,\zeta^0,\zeta^a,\widetilde{\zeta}_0,\widetilde{\zeta}_a,\sigma)$ we will denote $\mathcal{L}_{X_f}g_{\overline{N}}=T_{ij}dx^idx^j$. We first consider the Killing field $D_1=D-X_h$, satisfying

\begin{equation}
    [D_1,X_f]=-\frac{3}{2}X_f\,.
\end{equation}
From this and the fact that $D_1$ is Killing we find that

\begin{equation}
    \mathcal{L}_{D_1}\mathcal{L}_{X_f}g_{\overline{N}}=-\frac{3}{2}\mathcal{L}_{X_f}g_{\overline{N}}
\end{equation}
which we can write in the previous coordinates $(x^i)$ as

\begin{equation}
    D_1(T_{ij})+T_{kj}\partial_{x^i}D_1^k +T_{ik}\partial_{x^j}D_1^k=-\frac{3}{2}T_{ij}\,.
\end{equation}
Denoting the flow of $D_1$ as $\Phi_{D_1,t}$, we obtain the following system of first order, linear, homogeneous ODEs satisfied by the $T_{ij}(\Phi_{D_1,t}(p))$ for $p \in \overline{N}$:

\begin{equation}
\frac{d}{dt}T_{ij}(\Phi_{D_1,t}(p)) +    T_{kj}(\Phi_{D_1,t}(p))\partial_{x^i}D_1^k +T_{ik}(\Phi_{D_1,t}(p))\partial_{x^j}D_1^k+\frac{3}{2}T_{ij}(\Phi_{D_1,t}(p)) =0\,.
\end{equation}

In particular, we find that if $p \in \mathcal{H}$, then the $T_{ij}(\Phi_{D_1,t}(p))$ give a solution to the previous system satisfying the initial condition $T_{ij}(\Phi_{D_1,0}(p))=T_{ij}(p)=0$. Hence, we must have that for all $p \in \mathcal{H}$ and all $t$

\begin{equation}
    T_{ij}(\Phi_{D_1,t}(p))=0\,.
\end{equation}

Since $D_1|_{\overline{N}_1}=t^a\partial_{t^a}$, we find that for $(1,0+it^a,0)\in \mathcal{H}$

\begin{equation}
    \Phi_{D_1,t}(1,0+it^a,0)=(1,0+ie^tt^a,0)
\end{equation}
so that we conclude that $T_{ij}|_{\overline{N}_1}=0$ and hence $\mathcal{L}_{X_f}g_{\overline{N}}|_{\overline{N}_1}=0$.\\

Next we consider the Killing field $D$, which satisfies
$[D,X_f]=\frac{1}{2}X_f$. Since $3D+D_1=4D-X_h$ is a Killing field 
commuting with $X_f$, we see that $\mathcal{L}_{X_f}g_{\overline{N}}$ is invariant under the flow of $3D+D_1$. Note that applying the flow 
to $\overline{N}_1$ we obtain $\overline{N}_2$.  This follows from the fact that $3D+D_1=3\rho \partial_\rho + t^a\partial_{t^a}$ on $\overline{N}_2$. 
From $\mathcal{L}_{X_f}g_{\overline{N}}|_{\overline{N}_1}=0$
we hence conclude that 
$\mathcal{L}_{X_f}g_{\overline{N}}|_{\overline{N}_2}=0$. (Alternatively, we could have applied a similar argument as for $D_1$ to the flow of $D$.)\\

Finally, using the brackets:
\begin{equation}
    [X_f,P_0]=X_h, \;\; ,\;\; [X_f,P_a]=V_a, \;\; [X_f,X^i]=0, \;\; [X_f,Z]=-X^0/2
\end{equation}
and the fact that $X_h$, $P_i$, $X^i$, $Z$ and $V_a$ are Killing one obtains that $\mathcal{L}_{X_f}g_{\overline{N}}$ is invariant under the transformation group generated by the Lie algebra
\[ \mathbb{R}^n\ltimes \mathfrak{heis}_{2n+3}=\mathrm{span}\{V_a,P_i,X^i,Z\}.\]
Since the orbit of $\overline{N}_2\subset \overline{N}$ under that group is the whole manifold $\overline{N}$ and  $\mathcal{L}_{X_f}g_{\overline{N}_2}=0$, we conclude that $\mathcal{L}_{X_f}g_{\overline{N}}=0$.

\end{proof}

\begin{proposition}\label{fiso}
$X_f$ is a Killing vector field on $(\overline{N},g_{\overline{N}})$.
\end{proposition}
\begin{proof}
By the previous lemma, we only need to check that $\mathcal{L}_{X_f}g_{\overline{N}}=0$ on points of the form $(\rho,b^a+it^a,\zeta^i,\widetilde{\zeta}_i,\sigma)=(1,0+it^a,0,0,0)$ where the $t^a$ satisfy $h(t^a)=1$ (i.e. on points of $\mathcal{H}\subset \overline{N}$).\\

In the following all the Lie derivatives will be evaluated at points $(1,0+it^a,0,0,0)\in \mathcal{H}\subset \overline{N}$. We then have

\begin{equation}
    \mathcal{L}_{X_f}\Big(\frac{d\rho^2}{4\rho^2}\Big)\Big|_{\mathcal{H}}=-\frac{1}{2}d\zeta^0d\rho
\end{equation}
and using equation (\ref{exprmapmetric}) we find 
\begin{equation}
    \mathcal{L}_{X_f}g_{\overline{M}}\Big|_{\mathcal{H}}=\frac{k_{abc}t^bt^c}{4}d\zeta^0dt^a + \Big(\frac{k_{abc}t^c}{2} -\frac{k_{acd}k_{bef}t^ct^dt^et^f}{8}\Big) d\zeta^adb^b
\end{equation}
Furthermore, the Lie derivative of the line-bundle term gives:

\begin{equation}
    \mathcal{L}_{X_f}\Big(\frac{1}{64\rho^2}(d\sigma + \widetilde{\zeta}_id\zeta^i-\zeta^id\widetilde{\zeta}_i)^2\Big)\Big|_{\mathcal{H}}=-\frac{1}{8}d\sigma d\widetilde{\zeta}_0
\end{equation}
Computing the Lie derivative of the remaining term of $g_{\overline{N}}$
\begin{equation}
    \mathcal{L}_{X_f}\Big(-\frac{1}{4\rho}(N^{ij}-\frac{2}{K}z^i\overline{z}^j)W_i\overline{W}_j\Big)
\end{equation}
is a little bit more tricky. We do this in steps: we first have that  

\begin{equation}
    \mathcal{L}_{X_f}(N^{ij})\Big|_{\mathcal{H}}= \begin{pmatrix} -3\zeta^0N^{00} & -2\zeta^0N^{0a} +\frac{\zeta^a}{4h}\\
     -2\zeta^0N^{0a} +\frac{\zeta^a}{4h} &   -\zeta^0N^{ab} +\frac{\zeta^ab^b + b^a\zeta^b}{4h}\\
     \end{pmatrix}\Bigg|_{\mathcal{H}}=0
\end{equation}

\begin{equation}
    \mathcal{L}_{X_f}(2z^i\overline{z}^j/K)\Big|_{\mathcal{H}}= \begin{pmatrix} -6\frac{\zeta^0}{K} & -4\frac{\overline{z}^a\zeta^0}{K}-2\frac{\zeta^a}{K}\\
     -4\frac{z^a\zeta^0}{K}-2\frac{\zeta^a}{K}&-2\frac{z^a\overline{z}^b\zeta^0}{K}-2\frac{\zeta^a\overline{z}^b + z^a\zeta^b}{K}\\
     \end{pmatrix}\Bigg|_{\mathcal{H}}=0
\end{equation}
\begin{equation}\label{Ltau}
    \mathcal{L}_{X_f}(\tau_{ij}(z))\Big|_{\mathcal{H}}= \begin{pmatrix} -k_{abc}z^az^b(\zeta^0z^c-\zeta^c) & k_{abc}z^b(\zeta^0z^c-\zeta^c)\\
     k_{abc}z^b(\zeta^0z^c-\zeta^c)&-k_{abc}(\zeta^0z^c-\zeta^c)\\
     \end{pmatrix}\Bigg|_{\mathcal{H}}=0
\end{equation}
In particular, from (\ref{Ltau}) one obtains 

\begin{equation}
    \begin{split}
    \mathcal{L}_{X_f}(W_0)|_{\mathcal{H}}&=\frac{1}{2}d\sigma  +k_{abc}t^bt^c db^a  -\frac{2i}{3}k_{abc}t^at^bt^c d(\rho h^{-1})\\
    &=\frac{1}{2}d\sigma +k_{abc}t^bt^c db^a  -4id\rho + 2ik_{abc}t^at^bdt^c\\
    \mathcal{L}_{X_f}(W_a)|_{\mathcal{H}}&=2ik_{abc}t^cdb^a+ k_{abc}t^bt^c d(\rho h^{-1})\\
    &=2ik_{abc}t^cdb^a+k_{abc}t^bt^cd\rho -\frac{1}{2}k_{abc}t^bt^ck_{def}t^dt^edt^f
    \end{split}
\end{equation}

Denoting

\begin{equation}
    A^{ij}:=(N^{ij}-\frac{2}{K}z^i\overline{z}^j)
\end{equation}
we can then use the previous results to compute
\begin{equation}\label{hardterm}
    \begin{split}
        \mathcal{L}_{X_f}\Big(-\frac{1}{4\rho}A^{ij}W_i\overline{W}_j\Big)\Big|_{\mathcal{H}}=&\Big(-\frac{1}{4\rho}\Big[\mathcal{L}_{X_f}(A^{00}|W_0|^2)+\mathcal{L}_{X_f}(A^{0a}W_0\overline{W_a} + A^{a0}W_a\overline{W}_0)+\mathcal{L}_{X_f}(A^{ab}W_a\overline{W}_b)\Big]\Big)\Big|_{\mathcal{H}}\\
     =&\frac{1}{4}\Big[\text{Re}\Big(\mathcal{L}_{X_f}(W_0)\overline{W}_0\Big)\Big|_{\mathcal{H}}-\frac{t^a}{2 }\text{Re}(\mathcal{L}_{X_f}(iW_0)\overline{W}_a + \mathcal{L}_{X_f}(\overline{W}_a)iW_0)\Big|_{\mathcal{H}}\\
        &-\Big(-\frac{t^at^b}{2}+\frac{G^{ab}}{2}\Big)\text{Re}(\mathcal{L}_{X_f}(W_a)\overline{W}_b)\Big|_{\mathcal{H}}\Big]\\
    \end{split}
\end{equation}

For each of the terms above we obtain
\begin{equation}
    \begin{split}
    \frac{1}{ 4}\text{Re}\Big(\mathcal{L}_{X_f}(W_0)\overline{W}_0\Big)\Big|_{\mathcal{H}}=&\frac{1}{8}d\sigma d\widetilde{\zeta}_0 + \frac{1}{16}k_{abc}t^bt^cd\zeta^ad\sigma + \frac{1}{4}k_{abc}t^bt^cdb^ad\widetilde{\zeta}_0\\
    &+\frac{1}{8}k_{acd}t^ct^dk_{bef}t^et^fdb^ad\zeta^b +2d\rho d\zeta^0 - k_{abc}t^bt^cdt^a d\zeta^0\\
    \end{split}
\end{equation}
\begin{equation}
    \begin{split}
    -\frac{t^a}{ 8}\text{Re}\Big(\mathcal{L}_{X_f}(iW_0)\overline{W}_a\Big)\Big|_{\mathcal{H}}=&-\frac{t^a}{2}d\rho d\widetilde{\zeta}_a-\frac{3}{2}d\rho d\zeta^0+\frac{t^a}{4}k_{bcd}t^ct^ddt^bd\widetilde{\zeta}_a+\frac{3}{4}k_{bcd}t^ct^ddt^bd\zeta^0\\
    &-\frac{1}{16}k_{abc}t^at^cd\zeta^bd\sigma -\frac{1}{8}k_{dbc}t^bt^ck_{aef}t^at^edb^dd\zeta^f
    \end{split}
\end{equation}
\begin{equation}
    \begin{split}
    -\frac{t^a}{ 8}\text{Re}\Big(\mathcal{L}_{X_f}(\overline{W}_a)iW_0\Big)\Big|_{\mathcal{H}}=&-\frac{3}{2}d\rho d\zeta^0+\frac{3}{4 }k_{bcd}t^ct^ddt^bd\zeta^0-\frac{1}{4}k_{abc}t^at^cdb^bd\widetilde{\zeta}_0\\
    &-\frac{1}{8 }k_{abc}t^at^ck_{def}t^et^ddb^bd\zeta^f
    \end{split}
\end{equation}
\begin{equation}
    \begin{split}
        \frac{1}{4}\Big(\frac{t^at^b}{2}-\frac{G^{ab}}{2}\Big)\text{Re}(\mathcal{L}_{X_f}(W_a)\overline{W}_b)\Big|_{\mathcal{H}}=&\frac{t^b}{2}d\rho d\widetilde{\zeta}_b +\frac{3}{2}d\rho d\zeta^0-\frac{1}{4}t^bk_{efh}t^ft^hdt^ed\widetilde{\zeta}_b-\frac{3}{4}k_{efh}t^ft^hdt^ed\zeta^0\\
        &+\frac{1}{4}k_{adc}t^at^ck_{bef}t^bt^edb^dd\zeta^f  -\frac{1}{2}k_{bef}t^edb^bd\zeta^f\\\end{split}
\end{equation}

Summing up all contributions, we find that

\begin{equation}
    \begin{split}
     \mathcal{L}_{X_f}\Big(-\frac{1}{4\rho}A^{ij}W_i\overline{W}_j\Big)\Big|_{\mathcal{H}}&=\frac{1}{2}d\zeta^0d\rho -\frac{k_{abc}t^bt^c}{4}d\zeta^0dt^a - \Big(\frac{k_{abc}t^c}{2} -\frac{k_{acd}k_{bef}t^ct^dt^et^f}{8}\Big) d\zeta^adb^b+\frac{1}{8}d\sigma d\widetilde{\zeta}_0\\
     &=-\mathcal{L}_{X_f}\Big(\frac{d\rho}{4\rho^2}+g_{\overline{M}} +\frac{1}{64\rho^2}(d\sigma + \widetilde{\zeta}_id\zeta^i-\zeta^id\widetilde{\zeta}_i)^2\Big)\Big|_{\mathcal{H}}
     \end{split}
\end{equation}
so that $\mathcal{L}_{X_f}g_{\overline{N}}|_{\mathcal{H}}=0$. By Lemma \ref{simplifyinglemma} we then conclude that $\mathcal{L}_{X_f}g_{\overline{N}}=0$.
\end{proof}
From the previous results we obtain

\begin{theorem}\label{theorem}
S-duality acts by isometries on tree-level q-map spaces $(\overline{N},g_{\overline{N}})$.
\end{theorem}

\begin{proof}
We have shown in Propositions \ref{ehiso} and \ref{fiso} that the generators of the infinitesimal S-duality action $X_e$, $X_f$ and $X_h$ are Killing vectors. This implies that the S-duality action of 
$\mathrm{SL}(2,\mathbb{R})$ defined in equation (\ref{sl2can}) is isometric.  
\end{proof}

\begin{theorem}\label{theorem2} Let $(\overline{N},g_{\overline{N}})$ be a tree-level q-map space associated with a PSR manifold $(\mathcal{H},g_{\mathcal{H}})$ with $\text{dim}(\mathcal{H})=n-1\geq 0$. Then there is a $(3n+6)$-dimensional connected Lie group of isometries $G$ containing the isometries of the $L$-action and S-duality.  The Lie algebra $\mathfrak{g}$ of $G$ has the structure

\begin{equation}
    \mathfrak{g}=\mathbb{R}D\ltimes (\mathfrak{sl}_2(\mathbb{R})^{\text{op}}\ltimes (\mathbb{R}^n\ltimes \mathfrak{h}))
\end{equation}
with the brackets described in Section \ref{uniS}. 

\end{theorem}

\begin{proof}

Consider the Lie algebras of Killing fields

\begin{equation}
    \mathfrak{l}:=\text{span}\{D,X^i,P_a,X_e,X_h,V_a,Z\}, \quad \mathfrak{sl}(2,\mathbb{R})^{\text{op}}=\text{span}\{X_e,X_f,X_h\}\,,
\end{equation}
corresponding to the infinitesimal $L$-action and the infinitesimal S-duality action. Furthermore denote by $\text{Iso}_0(\overline{N},g_{\overline{N}})$ the identity component of the isometry group of $(\overline{N},g_{\overline{N}})$. Since $\mathfrak{l}\subset \text{Lie}(\text{Iso}_0(\overline{N},g_{\overline{N}}))$ and $\mathfrak{sl}(2,\mathbb{R})^{\text{op}}\subset \text{Lie}(\text{Iso}_0(\overline{N},g_{\overline{N}}))$ generate the Lie algebra $\mathfrak{g}$, we conclude that $\mathfrak{g}\subset \text{Lie}(\text{Iso}_0(\overline{N},g_{\overline{N}}))$. It follows 
that there exists a unique connected Lie subgroup $G\subset \text{Iso}_0(\overline{N},g_{\overline{N}})$ with $\text{Lie}(G)=\mathfrak{g}$. Obviously $G$ contains the isometries corresponding to the S-duality and $L$ actions, while the structure of the Lie algebra $\mathfrak{g}$ follows from the previous Corollary \ref{univliealgebra}.
\end{proof}

\begin{corollary} Recall the identification $\overline{N}\cong \mathcal{H}\times L$ together with the decomposition (\ref{qmapdecomp}) of $g_{\overline{N}}$ as

\begin{equation}\label{metricdecomp}
    g_{\overline{N}}=\frac{1}{4}g_{\mathcal{H}} + g_{L}(p)
\end{equation}
where $g_{L}(p)$ is a family of left-invariant metrics (with respect to the $L$-action) on $L$ parametrized by $p\in \mathcal{H}$. Then $G$ acts fiberwise by isometries with respect to the projection $\overline{N}\to \mathcal{H}$. In particular, when restricted to a fiber $L_p\cong L$, it acts by isometries on $(L,g_{L}(p))$.
\end{corollary}

\begin{proof}
We already know that $L$ and S-duality act fiberwise with respect to the projection $\overline{N}\to \mathcal{H}$ (Proposition \ref{fiberwiseSduality}). As a consequence, the same is true for the group $G$ generated by these groups.
\end{proof}

We can therefore conjecture the following: 

\begin{conjecture}
Let $\mathrm{Iso}_0(L,g_L(p))$ be the connected component of the identity of the isometry group of $(L,g_L(p))$. Then $G$ can be characterized as the intersection $\cap_{p\in \mathcal{H}}\mathrm{Iso}_0(L,g_L(p))$.
\end{conjecture}

\begin{corollary}\label{isoHcor} Let $\mathcal H \subset \mathbb{R}^n$ be a PSR manifold
and denote by 
\[ \mathrm{Aut}(\mathcal H) := \{ A\in \mathrm{GL}(n,\mathbb{R})\mid 
A \mathcal H = \mathcal H\}  \subset \mathrm{Aut}(h) := \{ A\in \mathrm{GL}(n,\mathbb{R})\mid 
A^*h = h\}
\]
its group of automorphisms. 
Then $\mathrm{Aut}(\mathcal H)$ acts by isometries on the corresponding 
tree-level q-map space $(\overline{N},g_{\overline{N}})$.  Moreover, 
this action normalizes the $(3n+6)$-dimensional connected Lie group of isometries $G$ of Theorem~\ref{theorem2} such that we have a semi-direct product
\[ \mathrm{Aut}(\mathcal H) \ltimes G \subset \mathrm{Iso}(\overline{N}, g_{\overline{N}}).\]
\end{corollary}

\begin{proof}The faithful isometric action of $\mathrm{Aut}(\mathcal H)$
on $\overline{N}$ is explicitly described in \cite[Appendix A]{CDJL}. 
An element $A\in \mathrm{Aut}(\mathcal H) \subset \mathrm{GL}(n,\mathbb{R})$
acts naturally on $\overline{M} = \mathbb{R}^n + i U$, $U=\mathbb{R}_{>0}\cdot \mathcal H$, by the identification $\mathbb{R}^n + i U \cong TU$. 
In formulas, this is simply $z=(z^a=b^a+it^a) \mapsto Az$. 
The action is extended to $\overline{N}$ as follows:
\[ (\rho , z,\zeta^0, \zeta,\tilde{\zeta}_0,\tilde{\zeta},\sigma) 
\mapsto (\rho , Az,\zeta^0, A\zeta,\tilde{\zeta}_0,(A^\top)^{-1}\tilde{\zeta},\sigma),\]
where $\zeta = (\zeta^a), \tilde{\zeta} = (\tilde{\zeta}_a)^\top \in \mathbb{R}^n$ and $(A^\top)^{-1}$ is the contragredient matrix. 
We already know from \cite[Appendix A]{CDJL} that this action normalizes
the isometric action of the group $L$, which is simply transitive on 
each fiber of $\overline{N} \rightarrow \mathcal{H}$. To see that 
it also normalizes the group $G\supset L$ it suffices to observe
that the vector field $X_f$ is manifestly invariant under the 
above action of $\mathrm{Aut}(\mathcal H) \subset \mathrm{Aut}(h)$.
\end{proof}

\begin{theorem}\label{finvolfibers} Consider a tree-level q-map space $(\overline{N},g_{\overline{N}})$ associated to a PSR manifold $(\mathcal{H},h)$ with 

\begin{equation}
    h(t^a)=\frac{1}{6}k_{abc}t^at^bt^c, \quad k_{abc}\in \mathbb{Z}\,.
\end{equation}
Furthermore, recall the nilradical $\mathfrak{n}\subset \mathfrak{g}$ from Corollary \ref{nilradicalcor}. Then there is a lattice $\Gamma$ of the normal and unimodular codimension 1 subgroup $\mathrm{SL}(2,\mathbb{R})\ltimes \exp(\mathfrak{n})\subset G$, acting by isometries on  $(\overline{N},g_{\overline{N}})$. The quotient gives a fiber bundle $\overline{N}/\Gamma \to \mathcal{H}$ with fibers of finite volume.

\end{theorem}

\begin{proof} Consider the homomorphism $\lambda : \mathfrak{g} \to \mathbb{R}$ given by  
$X \mapsto \mathrm{tr} (\mathrm{ad} X)$, and notice that $\mathrm{ker} (\lambda)= \mathfrak{sl}_2(\mathbb{R})^{\text{op}} \ltimes \mathfrak{n}$. It then follow that the corresponding codimension $1$ subgroup $\mathrm{SL}(2,\mathbb{R}) \ltimes \exp(\mathfrak{n})\subset G$ is unimodular and normal.\\

By the commutators in Proposition \ref{KValgebra} we see that $\mathfrak{n}$ is a nilpotent Lie algebra with integer structure constants, under the assumption $k_{abc}\in \mathbb{Z}$.
Then by Malcev’s theorem there is a lattice of the form $\exp (\Lambda )$ in the nilpotent group $\exp(\mathfrak{n})$, where $\Lambda$ consists of the integer span of a Malcev basis (a particular case of a Malcev basis is given by $(P_a, V_a, X^0, X^a , Z)$, corresponding to lower central series $\mathfrak{n}=\text{span}\{P_a, V_a, X^0, X^a , Z\}$, $\mathfrak{n}_1=[\mathfrak{n},\mathfrak{n}]=\text{span}\{X^0,X^a,Z\}$, $\mathfrak{n}_2=[\mathfrak{n},\mathfrak{n}_1]=\text{span}\{Z,X^0\}$, $\mathfrak{n}_3=[\mathfrak{n},\mathfrak{n}_2]=\{0\}$).
One can now check from the commutators in Proposition \ref{KValgebra} that under the S-duality action of $\mathrm{SL}(2,\mathbb{Z})$ by automorphisms of $\mathfrak{n}$
the lattice $\Lambda$ is preserved. This implies that the lattice $\exp(\mathfrak{n})$ is normalized by $\mathrm{SL}(2,\mathbb{Z})$, 
so $\mathrm{SL}(2,\mathbb{Z})\ltimes \Gamma$ is a lattice in $\mathrm{SL}(2,\mathbb{R}) \ltimes \exp(\mathfrak{n})\subset G$.\\

Now we can take the quotient of the manifold $\overline{N}$ by $\mathrm{SL}(2,\mathbb{Z})\ltimes \Gamma$,
which amounts to taking a quotient of every 
orbit of $\mathrm{SL}(2,\mathbb{R}) \ltimes \exp(\mathfrak{n})$. The space of orbits is $\mathcal{H}$, so we get a fiber bundle 
over $\mathcal{H}$ with orbits of finite volume. 
\qedhere
\end{proof}

In the following Section \ref{volumedensityappendix} we compute the fiber-wise volume density of $\overline{N}\to \mathcal{H}$ with respect to $g_{\overline{N}}$. We furthermore analyze the volume growth of the fibers of finite volume of $\overline{N}/\Gamma \to \mathcal{H}$ in the cases where $\mathcal{H}$ is a maximal PSR curve.
\end{subsection}

\begin{subsection}{Application: volume properties of q-map spaces associated to PSR curves}\label{volumedensityappendix}

In this section we consider $\overline{N}/\Gamma \to \mathcal{H}$ obtained from Theorem \ref{finvolfibers} and study the volume properties of $\overline{N}/\Gamma$ when $\mathcal{H}$ is a PSR curve. Since $\overline{N}/\Gamma \to \mathcal{H}$ has finite volume fibers, $\overline{N}/\Gamma$ has a chance of having finite volume. We will see that this is the case only when $\mathcal{H}$ is the incomplete PSR curve, giving rise to an incomplete QK manifold $\overline{N}/\Gamma$ of finite volume. \\  

Let $(\mathcal{H},h)$ be a PSR manifold of dimension $n-1\geq 0$ and $(\overline{N},g_{\overline{N}})$ the corresponding tree-level q-map space of dimension $4n+4$. The fibers of the canonical projection $\pi:\overline{N}\to \mathcal{H}$ are homogeneous submanifolds under the isometric action of the group $\mathrm{SL}(2,\mathbb{R})\ltimes \exp(\mathfrak{n})\subset G$ from Theorem \ref{finvolfibers}. We will consider $\mathcal{H}$ as a submanifold of $\overline{N}$ via the natural inclusion:

\begin{equation}
    \mathcal{H}\subset \mathcal{H}\times \{(r,\rho,b^a,\widetilde{\zeta}_i,\zeta^i,\sigma)\}=\mathcal{H}\times \mathbb{R}_{>0}^2\times \mathbb{R}^{3n+3}\cong \overline{N}\,,
\end{equation}
provided by the global coordinates $(t^a,\rho,b^a,\widetilde{\zeta}_i,\zeta^i,\sigma)$, where $t=(t^a)\in U=\mathbb{R}_{>0}\cdot \mathcal{H}\cong \mathcal{H}\times \mathbb{R}_{>0}$ is decomposed as $t=rp$ with $r=(h(t))^{1/3}$ and $h(p)=1$. Notice that the fibers $\pi^{-1}(p)$ are orthogonal to $\mathcal{H}$ with respect to $g_{\overline{N}}$ by (\ref{qmapdecomp}).\\

We start by computing the fiber-wise volume density corresponding to $(\overline{N},g_{\overline{N}})$ and $\pi:\overline{N} \to \mathcal{H}$. Namely, we compute the function $\delta \in C^{\infty}(\overline{N})$ defined by

\begin{equation}
    \mathrm{det}g_{\overline{N}}\Big|_{(T_p\mathcal{H})^{\perp}}=(\delta(\overline{p})\mathrm{dvol}_f)^{\otimes 2}\in \left(\Lambda^{3n+5}\left((T_p\mathcal{H})^{\perp}\right)^*\right)^{\otimes 2}\cong \mathbb{R}\,,
\end{equation}
where $\overline{p}\in \overline{N}$, $p=\pi(\overline{p})$, $\mathrm{dvol}_f$ denotes the product of the the differentials of the fiber coordinates $(r,\rho, b^a,\widetilde{\zeta}_i,\zeta^i,\sigma)$, and $(T_p\mathcal{H})^{\perp}$ denotes the orthogonal complement of $T_p\mathcal{H}\subset T_{\overline{p}}\overline{N}$ with respect to $g_{\overline{N}}$. We then study in Section \ref{completePSRcurvecase} the case when $\mathcal{H}$ is a complete PSR curve, and study the volume growth of the finite volume fibers of $\overline{N}/\Gamma \to \mathcal{H}$, where $\Gamma$ is as in Theorem \ref{finvolfibers}.

\begin{proposition}
The function $\delta(\overline{p})$ is a positive constant multiple of 
\begin{equation}
    \frac{\delta_{\overline{M}}}{\rho^{n+3}}\,,
\end{equation}
where $\delta_{\overline{M}}\in C^{\infty}(\overline{M})\subset C^{\infty}(\overline{N})$ is the positive function defined by 

\begin{equation}
    \mathrm{det}g_{\overline{M}}\Big|_{(T_p\mathcal{H})^{\perp_{\overline{M}}}}=(\delta_{\overline{M}}(\overline{p})\mathrm{dvol}_{\overline{M},f})^{\otimes 2}\in \left(\Lambda^{n+1}\left((T_p\mathcal{H})^{\perp_{\overline{M}}}\right)^*\right)^{\otimes 2}\cong \mathbb{R}\,,
\end{equation}
where $\mathrm{dvol}_{\overline{M},f}$ denotes the product of the differentials of the fiber coordinates $(r,b^a)$ of the fibration $\overline{M}\to \mathcal{H}$ of the PSK manifold associated to the PSR manifold $(\mathcal{H},h)$ via the r-map. Here $(T_p\mathcal{H})^{\perp_{\overline{M}}}$ denotes the orthogonal complement in $T_{\overline{p}}\overline{M}$.
\end{proposition}

\begin{proof}
Consider the CASK manifold $(M,g_{M},\omega_{M},\nabla,\xi)$ associated to $(\overline{M},g_{\overline{M}})$. Its metric volume form is parallel to the flat connection $\nabla$, since it is a constant multiple of a power of the $\nabla$-parallel K\"{a}hler form $\omega_{M}$. This implies that for $\nabla$-affine coordinates $q^I$, $I=1,2,...,2n+2$, the function $\mathrm{det}(g_M(\partial_{q^I},\partial_{q^J}))$ is constant. Changing the indefinite metric $g_M$, which in the present conventions is negative definite on $\mathcal{D}=\text{span}\{\xi,J\xi\}$, to a positive definite metric $\hat{g}_{M}$ by multiplying it by $-1$ on $\mathcal{D}$ yields a metric with the same volume form. Hence, $\mathrm{det}(\hat{g}_M(\partial_{q^I},\partial_{q^J}))$ is also constant. By \cite[Corollary 6]{CHM} we know that the metric coefficients of the metric $g_{\overline{N}}$ with respect to the coordinates $(\zeta_I)=(\widetilde{\zeta}_i,\zeta^i)$ are given by 

\begin{equation}
    g_{\overline{N}}(\partial_{\zeta_I},\partial_{\zeta_J})=\frac{\text{const}}{\rho}\hat{g}^{-1}_{M}(dq^I,dq^J)\,.
\end{equation}
where $\text{const}\in \mathbb{R}_{>0}$. We conclude that (for a new $\text{const}>0$)
\begin{equation}
    \mathrm{det} g_{\overline{N}}(\partial_{\zeta_I},\partial_{\zeta_J})=\frac{\text{const}}{\rho^{2n+2}}\,.
\end{equation}
Observing that the coefficients of $d\rho^2$ and $(d\sigma +\widetilde{\zeta}_id\zeta^i-\zeta^id\widetilde{\zeta}_i)^2$ in $g_{\overline{N}}$ have coefficients $1/\rho^2$ (up to constant factors), it is then easy to check that the proposition follows.
\end{proof}

\begin{proposition}
The function $\delta_{\overline{M}}\in C^{\infty}(\overline{M})\subset C^{\infty}(\overline{N})$ is given by 

\begin{equation}
    \delta_{\overline{M}}(\overline{p})=\frac{\sqrt{3}}{2^{n+1}r^{n+1}}\sqrt{\mathrm{det}(\gamma_{ab}(p))},
\end{equation}
where $\overline{p}\in \overline{N}$, $\pi(\overline{p})=p\in \mathcal{H}$, $p=(p^a)$, and

\begin{equation}
    \gamma_{ab}(p)=-k_{abc}p^c + \frac{1}{4}(k_{acd}p^cp^d)(k_{bef}p^ep^f)\,.
\end{equation}
\end{proposition}
\begin{proof}
This follows immediately from (\ref{rmapdecomp}).
\end{proof}
\begin{proposition}
Consider the Riemannian metric $g_U=-\partial^2\log h$ on $U=\mathbb{R}_{>0}\cdot \mathcal{H}$, the evaluation of which along $\mathcal{H}$ is $\gamma_{ab}dt^adt^b$. Furthermore, with respect to the natural coordinates $t^a$, $a=1,2,...,n$, on $U\subset \mathbb{R}^n$, let $\xi=t_a\partial_{t^a}$. Then $g_U$ is related to the affine special real (ASR) metric $g_{\text{ASR}}=-\partial^2 h$ on $U$ as follows:

\begin{equation}
    g_U(\xi_p,\xi_p)=-\frac{1}{2}g_{\text{ASR}}(\xi_p,\xi_p)=3, \quad g_U|_{\xi_p^{\perp}}=g_{\text{ASR}}|_{\xi_p^{\perp}}
\end{equation}
for all $p\in \mathcal{H}$. Here $\xi_p^{\perp}=T_p\mathcal{H}\subset T_pU=\mathbb{R}^n$ denotes the orthogonal complement of the line $\mathbb{R}\xi_p$ with respect to $g_U$ (or equivalently $g_{\text{ASR}}$).
\end{proposition}
\begin{proof}
Using the homogeneity of $h(t)=\frac{1}{6}k_{abc}t^at^bt^c$ we obtain 
\begin{equation}
    g_{\text{ASR}}(\xi,\xi)=-6h, \quad g_U(\xi,\xi)=-6+9=3\,.
\end{equation}
Restriction to $\mathcal{H}$ then yields the first statement. The other equation follows from the fact that the one form $k_{abc}t^at^bdt^c$ vanishes on $\xi_p^{\perp}=T_p\mathcal{H}$. 
\end{proof}

From the previous proposition, we then obtain the following

\begin{corollary} For $p\in \mathcal{H}\subset U$ we have
 \begin{equation}
     \mathrm{det}(\gamma_{ab}(p))= \frac{1}{2}(-1)^{n-1}\mathrm{det}\left(\frac{\partial^2h(p)}{\partial t^a\partial t^b}\right)
 \end{equation}
\end{corollary}

Summarizing, we obtain the following theorem:

\begin{theorem}\label{densitytheorem}
The fiber volume density function $\delta \in C^{\infty}(\overline{N})$ of any tree-level q-map space $\pi:\overline{N}\to \mathcal{H}$ is given by

\begin{equation}
    \delta(\overline{p})=\frac{\text{const}}{\rho^{n+3}r^{n+1}}\left| \mathrm{det}\left(\frac{\partial^2h(\pi(\overline{p}))}{\partial t^a \partial t^b}\right)\right|^{1/2},
\end{equation}
where $\text{const}$ is a positive constant depending only on $n=\text{dim}(\mathcal{H})+1$.
\end{theorem}
In the following subsections we specialize to PSR curves, that is to the case $\dim \mathcal H = 1$. It is sufficient to consider \emph{maximal} PSR curves.
\begin{definition}A PSR curve $\mathcal H \subset \{x\in \mathbb{R}^2 \mid h(x) =1\}$
(see Definition~\ref{PSR:def}) 
is called maximal if it is a connected component of the curve 
$\{x\in \mathbb{R}^2 \mid h(x) =1, \det \partial^2 h (x) < 0\}$.
\end{definition}
Note that complete curves are always maximal but the converse is not true. 
The maximal PSR curves are listed in \cite[Theorem 8]{CHM}. Up to isomorphism,
there are precisely three maximal PSR curves: two are complete and 
one incomplete.

\begin{subsubsection}{The case of a complete PSR curve}\label{completePSRcurvecase}

As an illustration, we will now compute the density function $\delta$ for the two complete PSR manifolds of dimension $1$. By Theorem \ref{densitytheorem} this amounts to computing the function 

\begin{equation}
    \Delta_h:=\left| \mathrm{det}\left(\frac{\partial^2h}{\partial t^a \partial t^b}\right)\right|^{1/2}
\end{equation}
along $\mathcal{H}$. Recall that there are two complete PSR curves \cite[Theorem 8]{CHM}, which are distinguished by the fact that one is homogeneous, that is has a transitive automorphism group, whereas the other has a finite automorphism group. To minimize the number of parentheses (in expressions like $(t^1)^2$), we will denote the coordinates of $\mathbb{R}^2$ by $x_1,x_2$ rather than $t^1,t^2$, since powers of the coordinates will be involved in the explicit expressions.\\

Let $\Gamma$ be a lattice in $\mathrm{SL}(2,\mathbb{R})\ltimes \exp(\mathfrak{n})\subset \text{Iso}(\overline{N},g_{\overline{N}})$ and denote by $\overline{N}/\Gamma$ the corresponding quotient manifold. The main qualitative result of this section is that the volume of the fibers of $\overline{N}/\Gamma$ along the two ends of the base curve is either asymptotically exponentially increasing or 
asymptotically decreasing to zero. The homogeneous PSR curve has ends of both types while the inhomogeneous complete PSR curve has divergent volume along both ends. \\

We first consider the PSR curve

\begin{equation}
    \mathcal{H}:=\{(x_1,x_2)\;\; |\;\;  x_1^2x_2=1,\; x_1>0\}.
\end{equation}
It is the unique homogeneous PSR curve, up to isomorphism. Its associated tree-level q-map space turns out to be the symmetric space $\overline{N}=\frac{\mathrm{SO}_0(3,4)}{\mathrm{SO}(3)\times \mathrm{SO}(4)}$ \cite[Corollary 4]{CHM}. We denote by $h(x)=x_1^2x_2$ the corresponding cubic polynomial. A straightforward calculation shows that 

\begin{equation}
    \Delta_h=2x_1\,.
\end{equation}
Next we compute the PSR metric in the global coordinate $x_1$ in order to determine the volume density of $\overline{N}$ along $\mathcal{H}$.

\begin{proposition}
The metric $g_{\mathcal{H}}=-\iota^*\partial^2h$ of the homogeneous PSR curve $\iota: \mathcal{H} \to \mathbb{R}^2$ is given by

\begin{equation}
    g_{\mathcal{H}}=6\frac{dx_1^2}{x_1^2}\,.
\end{equation}
In particular, its arc-length parameter $s$ measured from $(x_1,x_2)=(1,1)\in \mathcal{H}$ is given by $s=\sqrt{6}\log(x_1)$.
\end{proposition}
\begin{proof}
Eliminating $x_2=1/x_1^2$ we compute

\begin{equation}
    g_{\mathcal{H}}=-2x_2dx_1^2 -4x_1dx_1dx_2=6\frac{dx_1^2}{x_1^2}\,.
\end{equation}
\end{proof}
\begin{proposition}
The fiber-wise volume density of the tree-level q-map space $\overline{N}=\frac{\mathrm{SO}_0(3,4)}{\mathrm{SO}(3)\times \mathrm{SO}(4)}$ associated to the homogeneous PSR curve $(\mathcal{H},h)$ has the following form with respect to the coordinates $(s,r,\rho,\sigma,\zeta^i,\widetilde{\zeta}_i)$.

\begin{equation}
    \delta=\frac{\text{const}}{\rho^5r^3}\exp\left(\frac{s}{\sqrt{6}}\right)\,.
\end{equation}
\end{proposition}
\begin{proof}
This follows from Theorem \ref{densitytheorem} together with the results of this subsection.
\end{proof}

\begin{corollary}\label{homcase}
After taking a quotient of $\overline{N}\to \mathcal{H}$ by a lattice in the fiber-wise preserving group $\mathrm{SL}(2,\mathbb{R})\ltimes \exp(\mathfrak{n})\subset \text{Iso}(\overline{N})$, the fiber-wise volume increases exponentially along the homogeneous PSR curve $\mathcal{H}$ for $s\to \infty$, and decreases exponentially for $s\to -\infty$. In particular, the total volume is infinite. 
\end{corollary}

Next we consider the curve 

\begin{equation}
    \mathcal{H}:=\{(x_1,x_2)\quad | \quad x_1(x_1^2-x_2^2)=1,\; x_1>0\}
\end{equation}
with cubic polynomial $h(x)=x_1^3-x_1x_2^2$. It is the unique inhomogeneous complete PSR curve, up to isomorphism. A direct computation shows that

\begin{equation}
    \Delta_h^2=16x_1^2 -\frac{4}{x_1}\,.    
\end{equation}
In particular, we see that we have the asymptotics $\Delta_h \sim 4x_1$ as $x_1\to \infty$.\\

To interpret the result, we need to compare $x_1$ with the arc-length parameter of $\mathcal{H}$.

\begin{proposition}
The metric $g_{\mathcal{H}}=-\iota^*\partial^2h$ of the inhomogeneous complete PSR curve $\iota: \mathcal{H}\to \mathbb{R}^2$ is given with respect to the (local) coordinate $x_1>1$ by:

\begin{equation}
    g_{\mathcal{H}}=\frac{6}{x_1^2}\left(\frac{1-1/(4x_1^3)}{1-1/x_1^3}\right)dx_1^2\,.
\end{equation}
\end{proposition}
In particular, the arc length parameter $s$ measured from $(x_1,x_2)=(1,0)$ satisfies $s\sim \sqrt{6}\log(x_1)$ as $x_1 \to \infty$. (Hence the asymptotics coincides with the exact value for the homogeneous curve).
\begin{proof}
Eliminating $x_2=\pm \sqrt{x_1^2-1/x_1}$ we obtain

\begin{equation}
    \begin{split}
    g_{\mathcal{H}}&=-6x_1dx_1^2 +4x_2dx_1dx_2+2x_1dx_2^2\\
    &=\left(-6x_1 +2(2x_1+1/x_1^2)+\frac{(2x_1+1/x_1^2)^2}{2(x_1-1/x_1^2)}\right)dx_1^2\\
    &=\frac{6}{x_1^2}\left(\frac{1-1/(4x_1^3)}{1-1/x_1^3}\right)dx_1^2\,.
    \end{split}
\end{equation}
\end{proof}
\begin{corollary}
Consider the inhomogeneous complete PSR curve $\mathcal{H}$ and the corresponding complete \cite[Theorem 4 and 5]{CHM} (compare \cite{CDJL}) q-map space $(\overline{N},g_{\overline{N}})$. Then the fiber-wise volume density of the QK manifold in the coordinates $(s,r,\rho,\sigma,\widetilde{\zeta}_i,\zeta^i)$ satisfies

\begin{equation}
    \delta \sim \frac{\text{const}}{\rho^5r^3}\exp\left(\frac{s}{\sqrt{6}}\right)\,, \quad  s\to \infty\,,
\end{equation}
where $s\sim \sqrt{6}\log(x_1)$ is the arc length parameter on $\mathcal{H}$. In other words, the asymptotics coincide with the exact value obtained for the q-map space associated to the homogeneous PSR curve. 
\end{corollary}

\begin{corollary}
After taking a quotient of $\overline{N}\to \mathcal{H}$ by a lattice in $\mathrm{SL}(2,\mathbb{R})\ltimes \exp(\mathfrak{n})\subset \text{Iso}(\overline{N},g_{\overline{N}})$, the fiber-wise volume increases exponentially as a function of the arc-length parameter $s$ along both ends of $\mathcal{H}$. In particular, the total volume is infinite. 

\begin{proof}
Note that contrary to the homogeneous case of Corollary \ref{homcase}, in this case we have $s\to \infty$ along both ends of the curve, since $1\leq x_1 < \infty$.  
\end{proof}
\end{corollary}

\begin{remark}\label{PSRautrem}
This symmetric behavior is consistent with the fact that $\theta: (x_1,x_2)\to (x_1,-x_2)$ is an automorphism exchanging the two ends of the curve $\mathcal{H}$ and extending to an isometry of $(\overline{N},g_{\overline{N}})$ normalizing $\mathrm{SL}(2,\mathbb{R})\ltimes \exp(\mathfrak{n})\subset \text{Iso}(\overline{N},g_{\overline{N}})$, by Corollary \ref{isoHcor}. Moreover, the lattice can be chosen to be invariant under the above involution of $\overline{N}$, such that it induces a fiber-preserving isometric involution of the quotient space, inducing the map $\theta$ on $\mathcal{H}$.  
\end{remark}

\end{subsubsection}
\begin{subsubsection}{The case of an incomplete PSR curve}\label{incompletePSRcurvecase}

Finally, we consider the curve

\begin{equation}
    \begin{split}
    \mathcal{H}:&=\{(x_1,x_2)\quad | \quad x_1(x_1^2+x_2^2)=1,\; 3x_1^2-x_2^2<0, \quad x_2>0\}\\
    &=\{(x_1,x_2)\quad | \quad x_1(x_1^2+x_2^2)=1,\; 0<x_1<4^{-1/3},\quad  x_2>0\},\\
    \end{split}
\end{equation}
associated to the cubic polynomial $h(x)=x_1(x_1^2+x_2^2)$. It gives the unique incomplete maximal PSR curve \cite{CHM}, up to isomorphism.  \\

In this case, we have for $x\in \mathcal{H}$ that

\begin{equation}
    \Delta_h^2=\frac{4}{x_1}-16x_1^2\,,
\end{equation}
so that $\Delta_h^2 \to 0$ as $x_1 \to 4^{-1/3}$, and $\Delta_h^2 \sim 4/x_1$ as $x_1 \to 0$.\\

\begin{proposition}
The metric $g_{\mathcal{H}}=-\iota^*\partial^2h$ of the incomplete PSR curve $\iota: \mathcal{H}\to \mathbb{R}^2$ is given with respect to the the global coordinate $0<x_1<4^{-1/3}$ by:

\begin{equation}
    g_{\mathcal{H}}=\frac{6}{x_1^2}\left(\frac{1-1/(4x_1^3)}{1-1/x_1^3}\right)dx_1^2\,.
\end{equation}
In particular, if $s$ denotes the arc length parameter  we have $s \sim -\sqrt{\frac{3}{2}}\log(x_1)$ as $x_1 \to 0$, while $s \sim \text{const}$ as $x_1 \to 4^{-1/3}$.
\end{proposition}

\begin{proof} Eliminating $x_2=\pm \sqrt{1/x_1-x_1^2}$ we obtain
\begin{equation}
    \begin{split}
    g_{\mathcal{H}}&=-6x_1dx_1^2 -4x_2dx_1dx_2-2x_1dx_2^2\\
    &=\left(-6x_1 +2(2x_1+1/x_1^2)+\frac{(2x_1+1/x_1^2)^2}{2(x_1-1/x_1^2)}\right)dx_1^2\\
    &=\frac{6}{x_1^2}\left(\frac{1-1/(4x_1^3)}{1-1/x_1^3}\right)dx_1^2\,.
    \end{split}
\end{equation}
\end{proof}
\begin{corollary}
Consider the incomplete PSR curve $\mathcal{H}$ and the corresponding q-map space $(\overline{N},g_{\overline{N}})$. Then the fiber-wise volume density of the QK manifold in the coordinates $(s,r,\rho,\sigma,\widetilde{\zeta}_i,\zeta^i)$ satisfies

\begin{equation}
    \delta \sim \frac{\text{const}}{\rho^5r^3}\exp\left(-\frac{\sqrt{2}s}{\sqrt{3}}\right)\,, \quad  s\to \infty\,,
\end{equation}
where $s \sim -\sqrt{\frac{3}{2}}\log(x_1)$  is the arc length parameter on $\mathcal{H}$. 
\end{corollary}
\begin{corollary}
After taking a quotient of $\overline{N}\to \mathcal{H}$ by a lattice in $\mathrm{SL}(2,\mathbb{R})\ltimes \exp(\mathfrak{n})\subset \text{Iso}(\overline{N},g_{\overline{N}})$, the fiber-wise volume decreases exponentially as a function of the arc-length parameter $s$ along the end of $\mathcal H$ corresponding to $s\to \infty$, and remains of finite volume on the other end corresponding to $(x_1,x_2)=(4^{-1/3}, 3^{1/2}\cdot 4^{-1/3})$. In particular, we obtain that the (incomplete) QK manifold $\overline{N}/\Gamma$ is of finite volume. 

\end{corollary}

\begin{remark}
In \cite{CT}, certain instanton corrections to c-map spaces were studied, where the instanton corrections are encoded in the notion of mutually local variation of BPS structures. When including such instanton corrections, several of the previous universal continuous isometries of the tree-level q-map space are either broken, or expected to be broken down to a discrete subgroup. For example, the $\mathrm{SL}(2,\mathbb{R})$ isometries are expected to break down to a discrete subgroup $\mathrm{SL}(2,\mathbb{Z})$ \cite{Sduality}; while the isometries corresponding to the Killing fields $P_i$ are 
broken down to a discrete group \cite[Corollary 5.7 and Remark 5.8]{CT}. Provided the lattice $\Gamma \subset \mathrm{SL}(2,\mathbb{R})\ltimes \exp(\mathfrak{n})$ thus acts by isometries of the 
instanton corrected quaternionic K\"ahler metric, one can then ask whether the instanton corrections to $\overline{N}/\Gamma$ can be chosen such as to remove the incompleteness of the previous example while keeping the volume finite. This would give examples of complete QK manifolds with finite volume. 
\end{remark}
\end{subsubsection}

\end{subsection}

\end{section}
\begin{section}{Lift of the universal isometries to the twistor space}\label{lifttwistorspacesec}

Let $(\mathcal{Z},\mathcal{I},\lambda,\tau)$ be the twistor space of the tree-level q-map space $(\overline{N},g_{\overline{N}})$, where $\mathcal{I}$ denotes the complex structure, $\lambda$ the holomorphic contact structure, and $\tau$ the real structure (see \cite{Salamon1982,SwannHKQK}). In this section we wish to do a mathematical treatment of previous results from the physics literature \cite{NPV, APSV,Sduality} discussing holomorphic Darboux coordinates for $\lambda$, and the lift of S-duality to the twistor space. More precisely, our aim is the following:

\begin{itemize}
    \item We will give an explicit description of the twistor space of any QK manifold obtained via HK/QK correspondence. The  description will be in terms of the HK data, which are usually the one that admits a more explicit and simple description. For this, we will make use of certain results and notations from \cite{Conification}.
    \item We will then restrict to the case where the QK manifold is a tree-level q-map space $(\overline{N},g_{\overline{N}})$ and discuss a distinguished set of holomorphic Darboux coordinates for $\lambda$, originally found in \cite{NPV} via projective superspace arguments. In particular, we wish to discuss why the distinguished Darboux coordinates for $\lambda$ are actually holomorphic in the complex structure of $\mathcal{Z}$. The reason this is not immediate is because $\lambda$ is a holomorphic $1$-form valued in a holomophic line bundle $\mathcal{L}\to \mathcal{Z}$, and when discussing holomorphic Darboux coordinates for $\lambda$ one must also make sure one is working in a holomorphic trivialization of $\mathcal{L} \to \mathcal{Z}$.
    \item Finally, since the $(3n+6)$-dimensional group $G$ from Section \ref{Siso} acts by isometries on $(\overline{N},g_{\overline{N}})$, it must lift to an action on $(\mathcal{Z},\mathcal{I},\lambda,\tau)$ preserving the twistor space structure \cite{auttwistor}. In other words, the lift must act holomorphically, preserve the contact distribution, and commute with $\tau$. We end the section by discussing the lift of $G$. In particular, this includes the explicit lift of S-duality that was described in \cite{APSV,Sduality}, together with a lift of the $L$-action.
\end{itemize}

\begin{subsection}{Review of the twistor space of a QK manifold}\label{reviewtwistor}

Let $(\overline{N},g_{\overline{N}})$ by a QK manifold of dimension $4n$ with $n\geq 2$, and let $Q\to \overline{N}$ denote the associated quaternionic structure, a parallel subbundle $Q\subset \text{End}(T\overline{N})$  admitting local trivializations $(J_1,J_2,J_3)$  by skew-hermitian endomorphisms satisfying the quaternion relations. The twistor space of $(\overline{N},g_{\overline{N}})$ is then defined as  the sphere bundle $q:\mathcal{Z}\to \overline{N}$ given by

\begin{equation}
    \mathcal{Z}_p:=\{J \in Q_p \; | \; J^{2}=-1\}\,, \quad p \in \overline{N}\,.
\end{equation}

It is well known (see \cite{Salamon1982}) that $\mathcal{Z} \to \overline{N}$ carries a canonical complex structure $\mathcal{I}$, holomorphic contact structure  $\lambda$, and real structure $\tau$ (i.e.\ an anti-holomorphic involution of $(\mathcal{Z},\mathcal{I})$) with certain compatibility properties. Furthermore, the fibers $q^{-1}(p)=\mathcal{Z}_p\cong \mathbb{C}P^1$ are holomorphic submanifolds of $(\mathcal{Z},\mathcal{I})$, transverse to the contact distribution, and having normal bundle $\mathcal{O}(1)^{\oplus 2n}$. Conversely, by a theorem due to LeBrun \cite{LeBrun1989}, from the previous data of $(\mathcal{Z},\mathcal{I},\lambda, \tau)$ and $q:\mathcal{Z} \to \overline{N}$ one can invert the twistor construction and recover $g_{\overline{N}}$.\\

On the other hand, the bundle of frames of $Q\to \overline{N}$ determines a principal $\mathrm{SO}(3)$-bundle $S\to \overline{N}$. It is known that $\hat{N}:=S\times \mathbb{R}_{>0}\to \overline{N}$ carries a canonical pseudo-HK structure $(\hat{N},\hat{g},\hat{I}_{1},\hat{I}_{2},\hat{I}_{3})$ known as the Swann bundle or HK cone (see \cite{SwannHKQK} and \cite[Section 1]{QKPSK} for a review on the construction). In particular, the $\mathrm{SO}(3)$-action acts by isometries on $(\hat{N},\hat{g})$, and rotates the complex structures $\hat{I}_{\alpha}$; while the $\mathbb{R}_{>0}$-action due to the $\mathbb{R}_{>0}$-factor of $\hat{N}$ acts by homotheties on $\hat{g}$. The twistor space of $(\mathcal{Z},\mathcal{I},\lambda,\tau)$ of $(\overline{N},g_{\overline{N}})$ can be described in terms of $(\hat{N},\hat{g},\hat{I}_{1},\hat{I}_{2},\hat{I}_{3})$ as follows \cite{SwannHKQK,HitHKQK}:

\begin{itemize}
    \item By fixing one of the complex structures, say $\hat{I}_3$, one obtains an $\mathrm{SO}(2)\subset \mathrm{SO}(3)$-action fixing $\hat{I}_3$, and $\mathbb{C}^{\times}=\mathrm{SO}(2)\times \mathbb{R}_{>0}$ acts holomorphically on $(\hat{N},\hat{I}_3)$. The quotient $(\hat{N},\hat{I}_3)/\mathbb{C}^{\times}$ then gives $(\mathcal{Z},\mathcal{I})$, which is independent of the chosen $\hat{I}_3$.
    \item If $\hat{\omega}_{\alpha}:=\hat{g}(\hat{I}_{\alpha}-,-)$ denotes the K\"{a}hler forms of the pseudo-HK manifold $(\hat{N},\hat{g},\hat{I}_{1},\hat{I}_{2},\hat{I}_{3})$, then $\hat{\varpi}:=\hat{\omega}_1+i\hat{\omega}_2$ gives a holomorphic symplectic form for $(\hat{N},\hat{I}_3)$. If $K^{1,0}$ denotes the generator of the previous holomorphic $\mathbb{C}^{\times}$-action, then the complex rank $2n$ distribution $\text{Ker}(\iota_{K^{1,0}}\hat{\varpi})/K^{1,0}\to \hat{N}$ descends to a holomorphic contact distribution $\mathcal D$ on $(\mathcal{Z},\mathcal{I})$ defining the contact structure $\lambda$ 
    as a holomorphic one-form with values in a holomorphic line bundle $\mathcal{L}\to \mathcal{Z}$. 
    It is given by the canonical projection $\lambda : T\mathcal{Z} \rightarrow \mathcal{L}=T\mathcal{Z}/\mathcal D$. 
    \item The real structure $\tau$ descends from the anti-holomorphic action of $\hat{I}_1$ on $(\hat{N},\hat{I}_3)$.
\end{itemize}

\end{subsection}

\begin{subsection}{The HK cone associated to an HK manifold with rotating circle action}\label{HKcone}

Elaborating on work of Haydys \cite{HKQK},  it was shown in \cite{Conification} how to construct a 
conical pseudo-HK manifold $\hat{N}$ of dimension $4n+4$ from a pseudo-HK manifold $N$ of dimension $4n$ with a rotating circle action. The manifold 
$\hat{N}$ was called the \emph{conification} of $N$. In the following we will mostly drop the prefix `pseudo' for simplicity.\\

We start by describing the HK cone $(\hat{N},g_{\hat{N}},\hat{\omega}_{\alpha})$ associated to an HK manifold $(N,g_N,\omega_{\alpha})$ with rotating circle action as explicitly as possible, following \cite{Conification}. If   $(\overline{N},g_{\overline{N}},Q)$ denotes the QK manifold obtained from $(N,g_N,\omega_{\alpha})$ via HK/QK correspondence, then the HK cone associated to the HK manifold coincides with the HK cone associated to the QK manifold, described in the previous section \cite{QKPSK}. We will then use this to describe the twistor space of the QK manifold in terms of the HK manifold with rotating action. \\

Let us recall what the initial HK data for the HK/QK correspondence is:

\begin{itemize}
    \item We require a pseudo-HK manifold $(N,g_N,\omega_{\alpha})$ together with a vector field $V$ on $N$ satisfying:
\begin{equation}\label{rotaction}
    \mathcal{L}_V\omega_1=-2\omega_2, \;\;\; \mathcal{L}_V\omega_2=2\omega_1, \;\;\; \mathcal{L}_V\omega_3=0\,.
\end{equation}
We assume that $V$ is time-like or space-like, and that $f:N\to \mathbb{R}$ is a non-vanishing Hamiltonian for $V$, with respect to $\omega_3$:

\begin{equation}
    df=-\iota_V\omega_3\ .
\end{equation}
We furthermore define
\begin{equation}
    f_3:= f-\frac{g_N(V,V)}{2}\,,
\end{equation}
and assume $f_3$ is non-vanishing on $N$ \footnote{The subscript for $f_3$ is only to remember that it is built from a Hamiltonian of $V$ with respect to $\omega_3$.}.
\item We also need an $S^1$-principal bundle $\pi: P\to N$ with connection $\eta$, such that 
\begin{equation}
    d\eta=\pi^*(\omega_3-\frac{1}{2}d(\iota_Vg_N)).
\end{equation}
Such a bundle turns out to be hyperholomorphic, in the sence that the curvature $\pi^*F=d\eta$ satisfies that $F$ is of type $(1,1)$ with respect to $I_{\alpha}$ for $\alpha=1,2,3$ (see for example \cite{HitHKQK}).
\end{itemize}

With this data, we construct the HK cone $(\hat{N},g_{\hat{N}},\hat{\omega}_{\alpha})$ as follows. We will use the notations and construction of \cite{Conification}.\\

Letting $\widetilde{V}$ denote the horizontal lift of $V$ to $P$ with respect to $\eta$, and $X_P$ the fundamental vector field of $P$, we define

\begin{equation}
    V_3:=\widetilde{V}+f_3X_P\,,
\end{equation}
and the one forms on $P$ (below we identify tensor fields on $N$ with their pullbacks to $P$)

\begin{equation}\label{thetadef}
    \theta_3^P:=\eta +\frac{1}{2}\iota_Vg_N, \quad 
    \theta_1^P:=\frac{1}{2}\iota_V\omega_2,\quad 
    \theta_2^P:=-\frac{1}{2}\iota_V\omega_1\,. 
\end{equation}
Notice that 

\begin{equation}\label{L_V3:eq}
    d\theta_{\alpha}^{P}=\pi^*\omega_{\alpha}, \quad \mathcal{L}_{V_3}\eta=0\,, \quad  \mathcal{L}_{V_3}\theta_3^P=0, \quad \mathcal{L}_{V_3}\theta_1^P=-2\theta_2^P, \quad \mathcal{L}_{V_3}\theta_2^P=2\theta_1^P\,.
\end{equation}

Letting $\theta_{\alpha}:=f^{-1}\theta_{\alpha}^P$ and $(i_{\alpha})=(i,j,k)$ for $\alpha=1,2,3$ denote the canonical generators of the imaginary quaternions $\mathrm{Im}(\mathbb{H})\subset \mathbb{H}$; we consider the $1$-form $\theta \in \Omega^1(P,\mathrm{Im}(\mathbb{H}))$ given by

\begin{equation}
    \theta:=\sum_{\alpha=1}^3 \theta_{\alpha}i_{\alpha}\,.
\end{equation}

Now let $\varphi$ be the right-invariant Maurer-Cartan form on $\mathbb{H}^{\times}:=\mathbb{H}-\{0\}$, so that for $q\in \mathbb{H}^{\times}$ we can write $\varphi=dq\cdot q^{-1}=\varphi_0 + \sum_{\alpha=1}^3\varphi_{\alpha}i_{\alpha}$. If $e_0,e_1,e_2,e_3$ denote the right invariant vector fields of $\mathbb{H}^{\times}$ coinciding with the canonical basis at $q=1$, we then have

\begin{equation}
    \varphi_a(e_b)=\delta_{ab}, \quad a,b=0,1,2,3.
\end{equation}
We furthermore denote
\begin{equation}
    \mathrm{Ad}_q: \mathbb{H}\to \mathbb{H}, \;\;\;\; \mathrm{Ad}_q(x)=qxq^{-1}\,.
\end{equation}

Letting $\widetilde{N}:=\mathbb{H}^{\times}\times P$, we can extend $\theta$ to $\widetilde{\theta}\in \Omega^1(\widetilde{N},\mathbb{H})$ by setting

\begin{equation}\label{tildetheta:eq}
    \widetilde{\theta}=\widetilde{\theta}_0 + \sum_{\alpha=1}^{3}\widetilde{\theta}_{\alpha}i_{\alpha}:=\varphi + \mathrm{Ad}(\theta)\,.
\end{equation}

Furthermore, let $\widetilde{\omega}_{\alpha}\in \Omega^2(\widetilde{N})$ be given by

\begin{equation}\label{presymp}
    \widetilde{\omega}_{\alpha}:=d(|q|^2f\widetilde{\theta}_{\alpha})\,,
\end{equation}
and let $e_3^L$ denote the left-invariant vector field on $\mathbb{H}^{\times}$ which coincides with the canonical basis vector $e_3$ at $q=1$. Let $V_3^L:=e_3^L-V_3\in \Gamma(\widetilde{N}, T\widetilde{N})$ and consider $\hat{N}$ the space of integral curves of $V_3^L$ together with the projection $\widetilde{\pi}:\widetilde{N}\to \hat{N}$. We assume that $\hat{N}$ is a Hausdorff manifold. We then have:

\begin{theorem}{\cite[Theorem 2]{Conification}}
There is a pseudo-HK structure $(\hat{g},\hat{I}_{\alpha})$ on $\hat{N}$ with exact K\"{a}hler forms  $\hat{\omega}_{\alpha}\in \Omega^2(\hat{N})$ such that

\begin{equation}\label{pullsymp}
    \widetilde{\pi}^*\hat{\omega}_{\alpha}=\widetilde{\omega}_{\alpha}\,.
\end{equation}
Furthermore, the position vector field 
$\xi_{\mathbb{H}}$ on $\mathbb{H}^\times \subset \mathbb{H}$, $\xi_{\mathbb{H}}(q)=q$, projects to a vector field $\xi$ such that $(\hat{N},\hat{g},\hat{\omega}_{\alpha},\xi)$ is a conical HK-manifold. In particular, the right-invariant vector fields $e_{\alpha}$ for $\alpha=1,2,3$ descend to Killing vector fields $\hat{J}_{\alpha}\xi$, generating a 
faithful $\mathrm{SO}(3)$-action on the three-dimensional space spanned by the two-forms $\hat{\omega}_{\alpha}$.
\end{theorem}

An explicit form for $\widetilde{\omega}_{\alpha}$ in terms of the HK data is given by \cite[Lemma 3]{Conification}:

\begin{equation}\label{HKexp}
    \widetilde{\omega}_{\alpha}=2|q|^2 f\Big(\varphi_0\wedge \varphi_{\alpha} +\varphi_{\beta}\wedge \varphi_{\gamma} +\varphi_0\wedge(\text{Ad}_q\theta)_{\alpha} -\theta_0\wedge \varphi_{\alpha} +\varphi_{\beta}\wedge (\text{Ad}_q\theta)_{\gamma}-\varphi_{\gamma}\wedge (\text{Ad}_q\theta)_{\beta}\Big)+|q|^2(\text{Ad}_q\omega)_{\alpha}\,,
\end{equation}
where $(\alpha,\beta,\gamma)$ are in positive cyclic order.

\subsubsection{The c-map case}\label{c-mapcase}

It is well known that QK manifolds in the image of the 1-loop corrected c-map can be obtained by applying the HK/QK correspondence to a HK manifold obtained via the rigid c-map \cite{QKPSK}. More precisely, if $M$ is a CASK domain and $\overline{M}$ is the associated PSK manifold, then the QK manifold obtained by applying the 1-loop corrected c-map to $\overline{M}$ can also be obtained by applying the rigid c-map to $M$, and then the HK/QK correspondence, with the $1$-loop parameter corresponding to a choice of Hamiltonian $f$ for the rotating vector field \cite{QKPSK}. \\

In this subsection we wish to recall the description of the HK-manifold associated to a 1-loop corrected c-map space via HK/QK correspondence, and furthermore describe the HK cone data associated to it. This will be used later when we describe the Darboux coordinates for the contact structure of the QK twistor space. We will mostly use the formulas from \cite{CT}, with the corresponding instanton corrections set to $0$ (i.e. setting all BPS indices $\Omega(\gamma)=0$ for all charges $\gamma$). \\

Let $(M,\mathfrak{F})$ be a CASK domain of signature $(2,2n)$ (notice that this is, as in \cite{CT}, the opposite\footnote{To obtain the 
signature convention of Definition~\ref{defCASKdomain} it suffices to replace the holomorphic prepotential $\mathfrak{F}$ by $-\mathfrak{F}$.} signature from Definition~\ref{defCASKdomain}) and $(N,g_N,\omega_{\alpha})$ the associated HK manifold (of signature $(4,4n)$) via the rigid c-map.  Let $Z^i$ for $i=0,...,n$ be the natural holomorphic coordinates on the CASK domain $M\subset \mathbb{C}^{n+1}$, and let $Z_i:=\partial \mathfrak{F}/\partial Z^i$ and $\tau_{ij}:=\partial^2\mathfrak{F}/\partial Z^i \partial Z^j$. In the following, if $v,w \in \mathbb{R}^{2n+2}$ where $v=(v_i,v^i)$, $w=(w_i,w^i)$ and $i=0,1...,n$; then we denote by $\langle - , - \rangle$ the symplectic pairing

\begin{equation}\label{symppairing}
    \langle v,w \rangle= v_iw^i-v^iw_i\,.
\end{equation}
By combining (\ref{symppairing}) with the wedge product, we can describe the HK manifold $(N, g_N, \omega_{\alpha})$ explicitly by \cite[Section 2.4]{CT}:

 \begin{equation}
        \omega_1+i\omega_2=-\langle dZ\wedge d\zeta \rangle, \quad \omega_3=\frac{\pi}{2}\langle dZ\wedge d\overline{Z} \rangle - \frac{1}{4\pi }\langle d\zeta \wedge d\zeta \rangle\,,
    \end{equation}
where $\zeta=(\widetilde{\zeta}_i,\zeta^i)$ and $Z=(Z_i,Z^i)$; while the metric $g_{N}$ is given by

\begin{equation}
    g_{N}=2\pi (\text{Im}(\tau))_{ij}dZ^id\overline{Z}^j + \frac{1}{2\pi}(\text{Im}(\tau))^{ij}\widetilde{W}_i\overline{\widetilde{W}_j}\,, \quad \widetilde{W}_i:=d\widetilde{\zeta}_i - \tau_{ij}d\zeta^j\,,
\end{equation}
where $g_M = (\text{Im}(\tau))_{ij}dZ^id\overline{Z}^j$ is the conical affine special K\"ahler metric.  We remark that $\widetilde{W}_i$ differs from $W_i$ defined below \eqref{1-loopmetric} just in the convention of the signature of the CASK metric defined by $\tau_{ij}$ (i.e. with respect to $(M,\mathfrak{F})$ of signature $(2,2n)$ we have $W_i=d\widetilde{\zeta}_i - (-\tau_{ij})d\zeta^j$).
The rotating vector field $V$ is given by 

\begin{equation}
     V=2\Big(iZ^i\partial_{Z^i}-i\overline{Z}^i\partial_{\overline{Z}^i}\Big)\,,
\end{equation}
while if we denote by $\xi$ the Euler vector field of the CASK manifold, and $r:=\sqrt{g_{M}(\xi,\xi)}$, then

\begin{equation}
    f=2\pi r^2 -c, \;\;\; f_3= -2\pi r^2 -c, \quad \text{for} \quad c\in \mathbb{R}.
\end{equation}

On the other hand, the hyperholomorphic circle bundle $(P,\eta)$ associated to $(N,g_N,\omega_{\alpha},V,f)$ is given by $P=N\times S^1$ with connection \cite[Section 4]{CT}
    
\begin{equation}
    \begin{split}
    \eta&=d\sigma -\frac{1}{4\pi}\langle \zeta \wedge d\zeta \rangle +  \frac{i\pi }{2}(\overline{\partial} r^2 - \partial r^2) -\frac{1}{2}\iota_Vg_{N}\\
    &=d\sigma -\frac{1}{4\pi}\langle \zeta \wedge d\zeta \rangle   -\pi r^2\widetilde{\eta}
    \end{split}
\end{equation}
where we have set $\widetilde{\eta}:=\frac{1}{r^2}\iota_{J\xi}g_{M}$ (note that in the present conventions $V$ denotes the $\nabla$-horizontal lift of  $2J\xi$, not the horizontal lift of $J\xi$ as in \cite{CT}).\\

The QK metric we obtain by doing HK/QK correspodence to the HK data $(N,g_N,\omega_{\alpha},V,f)$ with the hyperholomorphic bundle $(P,\eta)$ and taking

\begin{equation}
    \overline{N}:=\{ (Z^i,\zeta^i,\widetilde{\zeta}_i,\sigma) \in P \quad | \quad \text{Arg}(Z^0)=0\} 
\end{equation}
is given by \cite[Equation 5.8]{CT}:
\begin{equation}\label{HKQK}
    g_{\overline{N}}=\frac{\rho+c}{\rho}g_{\overline{M}}+\frac{\rho +2c }{4(\rho+c)\rho^2}d\rho^2+\frac{4(\rho+c)}{\rho^2(\rho+2c)}\Big(d\sigma -\frac{1}{4\pi}\langle \zeta,d\zeta \rangle -\frac{c}{4}d^c\widetilde{\mathcal{K}}\Big)^2
    -\frac{1}{2\pi\rho}\widetilde{W}_i\Big(\widetilde{N}^{ij}-\frac{2(\rho+c)}{\rho\widetilde{K}}z^i\overline{z}^j\Big)\overline{\widetilde{W}}_j
\end{equation}
where $\rho=f$; $z^i:=Z^i/Z^0$\; and $\widetilde{N}_{ij}=\text{Im}(\tau_{ij})$, $\widetilde{K}=\widetilde{N}_{ij}z^i\overline{z}^j$, $\widetilde{\mathcal{K}}=-\log(\widetilde{K})$.  In particular, if we perform the rescaling

\begin{equation}\label{rescalings}
    \rho \to \frac{4}{\pi}\rho, \quad \sigma \to \frac{\sigma}{4\pi}, \quad c \to \frac{4}{\pi}c\,, \quad \zeta^i \to -\zeta^i
\end{equation}
we can write

\begin{equation}
    g_{\overline{N}}=\frac{\rho+c}{\rho}g_{\overline{M}}+\frac{\rho +2c }{4(\rho+c)\rho^2}d\rho^2+\frac{(\rho+c)}{64\rho^2(\rho+2c)}\Big(d\sigma +\langle \zeta,d\zeta \rangle -4cd^c\mathcal{K}\Big)^2
    -\frac{1}{4\rho}W_i\Big(N^{ij}-\frac{2(\rho+c)}{\rho K}z^i\overline{z}^j\Big)\overline{W}_j\,,
\end{equation}
matching (\ref{1-loopmetric}). Here we have used that $\widetilde{N}_{ij} = \mathrm{Im} (\tau_{ij}) = \mathrm{Im}(\partial^2_{ij}\mathfrak{F}) =  \frac{1}{2}(-2\mathrm{Im}(\partial^2_{ij}(-\mathfrak{F})))=\frac{1}{2}N_{ij}$, taking into account that $-\mathfrak{F}$ is the 
prepotential for which the signature of $(\mathrm{Im}(\partial^2_{ij}(-\mathfrak{F})))$ is $(2n,2)$. Furthermore, setting $c=0$ we obtain

\begin{equation}
    g_{\overline{N}}=g_{\overline{M}}+\frac{d\rho^2}{4\rho^2}+\frac{1}{64\rho^2}\Big(d\sigma +\langle \zeta,d\zeta \rangle\Big)^2 -\frac{1}{4 \rho}W_i\Big(N^{ij}-\frac{2}{K}z^i\overline{z}^j\Big)\overline{W}_j=g_{\text{FS}}^0\,,
\end{equation}
matching (\ref{QKtree}).\\

\subsection{The QK twistor space in terms of the HK cone}

Let $(\overline{N},g_{\overline{N}},Q)$ be the QK space obtained by applying the HK/QK correspondence to the data  $(N,g_N,$ $\omega_{\alpha},V,f)$ and the hyperholomorphic line bundle $(\pi:P \to N, \eta)$ (see \cite[Theorem 2]{QKPSK}). In this section, we use the description of the twistor space of $(\overline{N},g_{\overline{N}},Q)$ in terms of the HK cone given in Section \ref{reviewtwistor}, together with the description of the HK cone in terms of the HK data $(N,g_N,\omega_{\alpha},V,f)$ and $(\pi:P \to N, \eta)$ from Section \ref{HKcone}, in order to obtain a description of the twistor space in terms of the HK data.\\

\begin{itemize}
    \item \textbf{The complex manifold $(\mathcal{Z},\mathcal{I})$}: Let $(\hat{N},g_{\hat{N}},\hat{\omega}_{\alpha})$ be the HK cone associated to the data $(N,g_N,$ $\omega_{\alpha},V,f)$, and $(\pi:P\to N,\eta)$. Furthermore, consider the previous auxiliary space $\widetilde{N}=\mathbb{H}^{\times}\times P$ from Section \ref{HKcone}, together with the projection $\widetilde{\pi}:\widetilde{N}\to \hat{N}$. Let $q_0+q_1i+q_2j+q_3k\in \mathbb{H}^{\times}$, and consider the complex coordinates $z=q_0+iq_3$ and $w=q_1+iq_2$ on $\mathbb{H}^{\times} \cong    \mathbb{C}^{2}\setminus \{ 0\}$. Recall that on  $\widetilde{N}$ we have the commuting actions generated by the vector fields $V_3^L$ and $e_3$. With respect to the complex coordinates $(z,w)$ we can write:
    \begin{equation}\label{rotvectfields}
        V_3^L=e_3^L-V_3=iz\partial_z -i\overline{z}\partial_{\overline{z}} -iw\partial_w +i\overline{w}\partial_{\overline{w}}-V_3, \quad e_3=iz\partial_z -i\overline{z}\partial_{\overline{z}} +iw\partial_w -i\overline{w}\partial_{\overline{w}}\,.
    \end{equation}
    
    For future reference, we note that $e_3^L$ generates an $S^1$-action on $\mathbb{H}^{\times}$  given by
    
    \begin{equation}
        \lambda\cdot (z,w)\to (\lambda z, \overline{\lambda } w)\, ,
    \end{equation}
    while $e_3$ generates the $S^1$-action given by
    \begin{equation}
        \lambda \cdot (z,w) \to (\lambda z,\lambda w)\,.
    \end{equation}
    
    Since $[V_3^L,e_3]=0$, we have that $e_3$ descends to a vector field $K$ on $\hat{N}$, and by \cite[Theorem 2]{Conification} we have $K=\hat{I}_3\xi$ and 
    
    \begin{equation}
        \mathcal{L}_K\hat{I}_3=0, \quad \mathcal{L}_K\hat{I}_1=-2\hat{I}_2, \quad \mathcal{L}_K\hat{I}_2=2\hat{I}_1\,.
    \end{equation}
    Furthermore, the $S^1$-action on $\widetilde{N}$ admits a natural extension to a $\mathbb{C}^{\times}$-action, and descends to the quotient $\hat{N}$. It is easy to check that the $\mathbb{C}^{\times}$-action on $\widetilde{N}$ is generated by $\frac{1}{2}(e_3+ie_0)$.
     From the fact that a vector field of type $(1,0)$ with respect to a complex structure $J$ is holomorphic if any only if its real part preserves $J$ we obtain the following lemma.
    \begin{lemma} \label{Klemma}
    The vector field 
      \begin{equation}
     K^{1,0}:=\frac{1}{2}(K-i\hat{I}_3K)
    \end{equation}
    on $\hat{N}$ is $\hat{I}_3$-holomorphic. Furthermore, it generates an $\hat{I}_3$-holomorphic $\mathbb{C}^{\times}$-action on $\hat{N}$, descending from the $\mathbb{C}^{\times}$-action generated by $\frac{1}{2}(e_3+ie_0)$ on $\widetilde{N}$.
    \end{lemma}
    
    The twistor space $\mathcal{Z}$ of $\overline{N}$, together with its holomorphic structure is then obtained by the holomorphic quotient (see for example \cite[Section 4]{HitHKQK}):
    
    \begin{equation}
        (\mathcal{Z},\mathcal{I}):=(\hat{N},\hat{I}_3)/\mathbb{C}^{\times}_{K^{1,0}}\,.
    \end{equation}
    
    Even more, the function $\mu=\hat{g}(\xi,\xi)/2$ gives a moment map for the $\mathrm{U}(1)$-action generated by $K$, so $(\mathcal{Z},\mathcal{I})$ has a (pseudo)-K\"{a}hler structure $(g_{\mathcal{Z}},\mathcal{I})$, induced from the K\"{a}hler quotient
    
    \begin{equation}
        (\mathcal{Z},\mathcal{I},g_{\mathcal{Z}})=(\hat{N},\hat{I}_3,\hat{g})//
    \mathrm{U}(1)_K \,.
    \end{equation}
    
    \item \textbf{The holomorphic contact structure $\lambda$}: we consider the holomorphic symplectic form on $\hat{N}$ in complex structure $\hat{I}_3$ given by
    \begin{equation}
        \hat{\varpi}:=\hat{\omega}_1+i\hat{\omega}_2\,.
    \end{equation}
    
    To describe this holomorphic symplectic form as explicitly as possible, we consider as before the complex coordinates $z$ and $w$ on the $\mathbb{H}^{\times}$ factor of $\widetilde{N}$. We then have by (\ref{presymp}) and (\ref{pullsymp}) that $\widetilde{\lambda}\in \Omega^1(\widetilde{N},\mathbb{C})$ defined by
    
    \begin{equation}
        \widetilde{\lambda}:=|q|^2f(\widetilde{\theta}_1+i\widetilde{\theta}_2)
    \end{equation} 
    satisfies
    \begin{equation}\label{tildeomega:eq}
        d\widetilde{\lambda}=
        \widetilde{\varpi} =\widetilde{\omega}_1 + i\widetilde{\omega}_2=\widetilde{\pi}^*\hat{\varpi}\,.
    \end{equation}
    In terms of the coordinates $z$ and $w$ on the $\mathbb{H}^{\times}$-factor of $\widetilde{N}$, it follows from a straightforward computation using (\ref{tildetheta:eq}) that $\widetilde{\lambda}$ can be written as 
    
\begin{equation}\label{preliouville}
    \begin{split}
        \widetilde{\lambda}&=f(zdw-wdz +(z^2+w^2)\theta_1 +i(z^2-w^2)\theta_2-2izw\theta_3)\\
        &=f(zdw-wdz +z^2\theta_{+} +w^2\theta_{-}-2izw\theta_3)
        \end{split}
    \end{equation}
    where $\theta_{\pm}:=\theta_1\pm i\theta_2$.\\ 
    
    Using that 
    \begin{equation}
        V_3^L=iz\partial_z -i\overline{z}\partial_{\overline{z}} -iw\partial_w+i\overline{w}\partial_{\overline{w}} -\widetilde{V} - f_3X_P\,,
    \end{equation}
    one can easily check that 
    \begin{equation}
        \begin{split}
         \iota_{V_3^L}\widetilde{\lambda}&=f(-2izw -2izw\theta_3(-V_3))=0\\
         \mathcal{L}_{V_3^L}\widetilde{\lambda}&=f\Big(\mathcal{L}_{V_3^L}(zdw-wdz)+\mathcal{L}_{V_3^L}(z^2\theta_+ )+\mathcal{L}_{V_3^L}(w^2\theta_-)+\mathcal{L}_{V_3^L}(-2izw\theta_3)\Big)=0
        \end{split}
    \end{equation}
    where we used that $\theta_3(V_3)=1$ and (compare with equation (\ref{L_V3:eq}))
    
    \begin{equation}
        \mathcal{L}_{V_3^L}z^2=2iz^2, \quad \mathcal{L}_{V_3^L}w^2=-2iw^2, \quad \mathcal{L}_{V_3^L}\theta_+=-2i\theta_+, \quad \mathcal{L}_{V_3^L}\theta_-=2i\theta_-, \quad \mathcal{L}_{V_3^L}\theta_3=0\,.
    \end{equation}
    We conclude that $\widetilde{\lambda}$ descends to a $1$-form $\hat{\lambda}\in \Omega^1(\hat{N},\mathbb{C})$ such that
    
    \begin{equation}
        \widetilde{\pi}^*\hat{\lambda}=\widetilde{\lambda},\quad d\hat{\lambda}=\hat{\varpi}\,.
    \end{equation}
    \begin{lemma} \label{liouvillehol}
    The complex $1$-form $\hat{\lambda}$ is $\hat{I}_3$-holomorphic. Furthermore, if $R_x$ denotes the action by $x\in \mathbb{C}^{\times}$ under the holomorphic $\mathbb{C}^{\times}$-action on $\hat{N}$ generated by $K^{1,0}$, we also have 
    \begin{equation}
        R^*_x\hat{\lambda}=x^2\hat{\lambda}\,.
    \end{equation}
    \end{lemma}
    
    \begin{proof}
    To show this, we recall from Lemma \ref{Klemma} that the vector field $K^{1,0}$ is $\hat{I}_3$-holomorphic, and that $e_3$ projects to $K$.  In particular, we conclude that $\iota_{K^{1,0}}\hat{\varpi}$ is a holomorphic $1$-form in complex structure $\hat{I}_3$ and
    
    \begin{equation}
        \widetilde{\pi}^*\iota_{K^{1,0}}\hat{\varpi}=\widetilde{\pi}^*\iota_K\hat{\varpi}=\iota_{e_3}\widetilde{\varpi}.
    \end{equation}
    On the other hand, using (\ref{tildeomega:eq}),
    (\ref{preliouville}) and (\ref{rotvectfields}), we find that
    
    \begin{equation}
        \iota_{e_3}\widetilde{\varpi}=\mathcal{L}_{e_3}\widetilde{\lambda} -d(\iota_{e_3}\widetilde{\lambda})= 2i\widetilde{\lambda}=2i\widetilde{\pi}^*\hat{\lambda}\,.
    \end{equation}
    It then follows that
    \begin{equation}\label{omegalambdarel}
        \iota_{K^{1,0}}\hat{\varpi}=2i\widehat{\lambda}
    \end{equation}
    and hence that $\hat{\lambda}$ is holomorphic in complex structure $\hat{I}_3$.\\
    
    On the other hand, recall that $\frac{1}{2}(e_3+ie_0)$ generates a $\mathbb{C}^{\times}$ -action on the $\mathbb{H}^{\times}$-factor of $\widetilde{N}$ acting by
    
    \begin{equation}
        (z,w)\to (xz,xw)\,.
    \end{equation}
    Using (\ref{preliouville}), we then find
    
    \begin{equation}
        R_x^*\widetilde{\lambda}=x^2\widetilde{\lambda}\,.
    \end{equation}
    Finally, since $R_x$ descends to $\hat{N}$, $R_x$ and $\widetilde{\pi}$ commute, so that 
    
    \begin{equation}
        \widetilde{\pi}^*R^*_x\hat{\lambda}=R^*_x\widetilde{\lambda}=\widetilde{\pi}^*(x^2\hat{\lambda}) \implies R^*_x\hat{\lambda}=x^2\hat{\lambda}\,.
    \end{equation}
    \end{proof}

\begin{proposition} \label{holcontstructure} Consider the holomorphic line bundle

\begin{equation}
    \mathcal{L}:=\hat{N}\times_{\mathbb{C}^{\times}}\mathbb{C} \to \mathcal{Z} 
\end{equation} associated to the action $x\cdot z=x^2z$ of $\mathbb{C}^{\times}$ on $\mathbb{C}$. Then $\hat{\lambda}$ descends to a holomorphic section of $T^*\mathcal{Z}\otimes \mathcal{L}\to \mathcal{Z}$ (i.e a holomorphic $1$-form on $\mathcal{Z}$ with values in $\mathcal{L}$), and $\lambda$ defines a holomorphic contact distribution on $\mathcal{Z}$ matching the canonical contact structure of the twistor space.
\end{proposition}
\begin{proof}
Notice that

\begin{equation}
    \widetilde{\pi}^*\iota_{K^{1,0}}\hat{\lambda}=\widetilde{\pi}^*\iota_K\hat{\lambda}=\iota_{e_3}\widetilde{\lambda}=\iota_{e_3}\iota_{e_3}(\widetilde{\varpi}/2i)=0 \implies \iota_{K^{1,0}}\hat{\lambda}=0\,.
\end{equation}
so that $\hat{\lambda}$ is vertical with respect to $\hat{N} \to \mathcal{Z}$. By the transformation rule of $\hat{\lambda}$ under the $\mathbb{C}^{\times}$-action from Lemma \ref{liouvillehol}, we conclude that $\hat{\lambda}$ descends to a holomorphic $1$-form on $\mathcal{Z}$ with values in $\mathcal{L}$.\\

To see that $\lambda$ defines a holomorphic contact structure on $\mathcal{Z}$, we note that if $2n=\mathrm{dim}_{\mathbb{C}}(N)$, then 
    
    \begin{equation}
        \hat{\lambda} \wedge d\hat{\lambda}^n=\frac{1}{(2i)^{n+1}}\iota_{K^{1,0}}\hat{\varpi} \wedge (\mathcal{L}_{K^{1,0}}\hat{\varpi})^n=\frac{i^{n}}{(2i)^{n+1}}\iota_{K^{1,0}}\hat{\varpi}^{n+1}\,.
    \end{equation}
    Since $K^{1,0}$ generates the $\mathbb{C}^{\times}$-action from above, and $(\hat{\varpi})^{n+1}$ gives a holomorphic volume form on $\hat{N}$, it follows that $\hat{\lambda}\wedge d\hat{\lambda}^n$ descends to a non-degenerate holomorphic form on  $\mathcal{Z}$, valued in $\mathcal{L}^{\otimes (n+1)}$. It then follows that $\lambda$ defines a holomorphic contact structure on $\mathcal{Z}$. The fact that it matches the canonical contact structure on the twistor space follows from the remarks at the end of Section \ref{reviewtwistor}.\\
    
\end{proof}
    
    \begin{remark}
    From the previous proof it follows that if $\mathcal{K} \to \mathcal{Z}$ is the canonical bundle of $\mathcal{Z}$, and $\mathcal{K}^*$ the dual bundle, then $\mathcal{L}\cong (\mathcal{K}^*)^{1/(n+1)}$ (i.e. $\mathcal{L}$ is an $(n+1)$-root of $\mathcal{K}^*$ with respect to the tensor product of bundles).
    \end{remark}

    For future reference, we remark that smooth sections of $\mathcal{L}\to \mathcal{Z}$ can be identified with smooth functions $\hat{N}\to \mathbb{C}$ homogeneous of degree $2$ with respect to the action generated by $K$. In turn, such functions can be identified with functions $\widetilde{N}\to \mathbb{C}$ which are homogeneous of degree $2$ with respect to the action of $e_3$, and invariant under the action of $V_3^L$. A particular example of this is given by the function $zw: \widetilde{N} \to \mathbb{C}$, which then defines a smooth section of $\mathcal{L}\to \mathcal{Z}$. Latter we will see that this section is holomorphic.
    
    \item The real structure $\tau$: via the identification     
    \begin{equation}
        (\mathcal{Z},\mathcal{I}):=(\hat{N},\hat{I}_3)/\mathbb{C}^{\times}_{K^{1,0}}\,.
    \end{equation} the real structure $\tau$ on $(\mathcal{Z},\mathcal{I})$ descends from the anti-holomorphic map  $\hat{I}_1:(\hat{N},\hat{I}_3)\to (\hat{N},\hat{I}_3)$. The fact that it descends to an involution comes from the fact that $\hat{I}_1^2=-1$ acts as the identity on the quotient by the $\mathbb{C}^{\times}$-action.
    \end{itemize}
    \subsubsection{Local expressions of the contact structure}
    
    In the following, we set $t:=w/z$. From (\ref{preliouville}), we can then write the following expression for $\widetilde{\lambda}\in \Omega^1(\widetilde{N},\mathbb{C})$:
    \begin{equation}
       \widetilde{\lambda}=f\Big(\frac{dt}{t}+t^{-1}\theta_+ -2i\theta_3 +t\theta_-\Big)\cdot zw\,.
    \end{equation}
    Each of the two factors in the above expression is invariant under the $V_3^L$-action. In particular, $zw$ descends to a function $[zw]$ on $\hat{N}$, homogeneous of degree two with respect to the $K^{1,0}$-action, while 
    \begin{equation}
       f\Big(\frac{dt}{t}+t^{-1}\theta_+ -2i\theta_3 +t\theta_-\Big)
    \end{equation}
    descends to a $1$-form $\hat{\lambda}_0$ on $\hat{N}$, invariant under the $K^{1,0}$-action. We can then locally write
    
    \begin{equation}
       \hat{\lambda}=\hat{\lambda}_0\cdot [zw]\,.
    \end{equation}
    
    Furthermore, since $\hat{\lambda}_0$ is invariant under the $K^{1,0}$-action, it descends to a $1$-form $\lambda_0$ on the twistor space $\mathcal{Z}$, while $[zw]$ gives a section $s$ of $\mathcal{L}\to \mathcal{Z}$, being a function on $\hat{N}$ homogeneous of degree $2$ with respect to the $K^{1,0}$-action.\\
    
    We can therefore locally write
    \begin{equation}
       \lambda=\lambda_0\cdot s\,.
    \end{equation}
    To describe $\lambda_0$ in local coordinates, it is enough to identify $\hat{N}$ with a submanifold of $\widetilde{N}$ transverse to $V_3^L$, and pick local coordinates for the transverse submanifold. In particular, if we pick coordinates of the form $(z,w,x^a)$ then we have
    
  \begin{equation}\label{localcontform}
    \begin{split}
       \lambda_0&=f\Big(\frac{dt}{t}+t^{-1}\theta_+|_{\hat{N}} -2i\theta_3|_{\hat{N}} +t\theta_-|_{\hat{N}}\Big)\\
       &=f\frac{dt}{t}+t^{-1}\theta_+^P|_{\hat{N}} -2i\theta_3^P|_{\hat{N}} +t\theta_-^P|_{\hat{N}}\,,
    \end{split}
    \end{equation}
 where the restrictions $\theta_{\alpha}|_{\hat{N}}$ and $\theta_{\alpha}^P|_{\hat{N}}=f\theta_{\alpha}|_{\hat{N}}$ only depend on the $x^a$-coordinates.\\
 
 We finish this section with the following useful lemma, which will be used to prove that Darboux coordinates for $\lambda$ to be discussed in the next section, are actually holomorphic coordinates on the twistor space.
 
 \begin{lemma} \label{lemmaholsection} The section $s$ of $\mathcal{L}\to \mathcal{Z}$ is holomorphic. In particular, since $\lambda$ is a holomorphic section of $T^*\mathcal{Z}\otimes \mathcal{L} \to \mathcal{Z}$, the $1$-form $\lambda_0$ on $\mathcal{Z}$ must also be holomorphic. 
 
 \end{lemma}
 
 \begin{proof}
 We use the notation from \cite[Section 2]{Conification}. In particular, we consider the real codimension $1$-distribution $\mathcal{D}\oplus E$ of $T\widetilde{N}$ and the endomorphisms $\widetilde{I}_{\alpha}$ of $\mathcal{D}\oplus E\to \widetilde{N}$ used to describe the complex structures $\hat{I}_{\alpha}$ on $\hat{N}$. More specifically:
 
 \begin{equation}
     \mathcal{D}=\text{span}\{e_0,e_{\alpha},\widetilde{V},\widetilde{I_{\alpha}V}\}, \quad E=(\text{span}\{\widetilde{V},\widetilde{I_{\alpha}V}\})^{\perp}\subset T^hP\,.
 \end{equation}
 where $\perp$ inside $T^hP$ is taken with respect to the metric of $P$ given by
 
 \begin{equation}
     g_P=\frac{2}{f_3}\eta^2 + \pi^*g_N\,,
 \end{equation}
 and $\widetilde{W}$ denotes the horizonal lift of a vector field with respect to $\eta$.\\ 
 
 Showing that $s$ is holomorphic is equivalent to showing that $[zw]$ is an $\hat{I}_3$-holomorphic function on $\hat{N}$, which in turn reduces to checking that 
 \begin{equation}
     \widetilde{I}_3^*d(zw)=id(zw)|_{\mathcal{D}\oplus E}\,.
 \end{equation}
 We consider the frame $(e_0,e_{\alpha},\widetilde{W}_0,\widetilde{W}_{\alpha}',E_i)$ as in \cite{Conification}, where $E_i$ is a frame for $E$ and $(e_0,e_{\alpha},\widetilde{W}_0,\widetilde{W}_{\alpha}')$ a global frame for $\mathcal{D}$. In $\cite{Conification}$ it is shown that $\widetilde{I}_{\alpha}$ splits as
 \begin{equation}
     \widetilde{I}_{\alpha}=\widetilde{I}_{\alpha}|_{\text{span}\{e_0,e_{\alpha}\}}\oplus \widetilde{I}_{\alpha}|_{\text{span}\{\widetilde{W}_0,\widetilde{W}_{\alpha}'\}}\oplus \widetilde{I}_{\alpha}|_{E}
 \end{equation}
 In particular, since $d(zw)$ vanishes when evaluated on $(\widetilde{W}_0,\widetilde{W}_{\alpha}',E_i)$, it is enough to check that 
 
 \begin{equation}
     (\widetilde{I}_3^*d(zw))(e_a)=id(zw)(e_a), \quad a=0,1,2,3\,.
 \end{equation}
 This follows from the fact that

\begin{equation}
    \widetilde{I}_3(e_0)=e_3, \quad \widetilde{I}_3(e_3)=-e_0,\quad \widetilde{I}_3(e_1)=e_2, \quad \widetilde{I}_3(e_2)=-e_1\,.
\end{equation}
together with the identities 
\begin{equation}
    d(zw)(e_0)=2zw, \quad d(zw)(e_3)=2izw, \quad d(zw)(e_1)=|z|^2-|w|^2, \quad d(zw)(e_2)=i(|z|^2-|w|^2).
\end{equation}
It then follows that $[zw]$ is a $\hat{I}_3$-holomorphic function on $\hat{N}$, and hence that $s$ is a holomorphic section of $\mathcal{L}\to \mathcal{Z}$. \end{proof}

\subsubsection{Darboux coordinates for the case of tree-level q-map spaces}\label{darbouxcoordssection}

We now wish to focus on the tree level q-map case, and describe $\lambda_0$ in the natural variables appearing in the context of c-map metrics. We will also discuss certain Darboux coordinates for $\lambda$, first found in \cite{NPV}. Our previous arguments (i.e. Proposition \ref{holcontstructure} and Lemma \ref{lemmaholsection}) will allow us to show that they are actually holomorphic coordinates on $(\mathcal{Z},\mathcal{I})$. \\

In the following, we identify $\hat{N}$ with the following submanifold of $\widetilde{N}$, transverse to $V_3^L$:
\begin{equation}
    \hat{N}:=\{(Z^i,\zeta^i,\widetilde{\zeta}_i,\sigma,q)\in \widetilde{N} \;\;| \;\; \text{Arg}(Z^0)=0\}\cong \overline{N}\times \mathbb{H}^{\times}\subset \widetilde{N}\,,
\end{equation}
which then has coordinates $(\rho, z^a,\zeta^i,\widetilde{\zeta}_i,\sigma,z,w)$, where $f=\rho=2\pi r^2$ and $z^a:=Z^a/Z^0$.\\

In this case, by Section \ref{c-mapcase} and (\ref{thetadef}) one finds that the $1$-forms $\theta_{\alpha}$ on $P$ satisfy:
\begin{equation}
        \theta_+=-\frac{1}{f}\langle Z, d\zeta \rangle, \quad
        \theta_-=-\frac{1}{f}\langle \overline{Z}, d\zeta \rangle, \quad 
        \theta_3=\frac{d\sigma}{f}-\frac{1}{4\pi f}\langle \zeta, d\zeta \rangle +\frac{1}{2}\widetilde{\eta}\,.
\end{equation}
There formulas correspond to the conventions at the beginning of Section \ref{c-mapcase}, where $\mathfrak{F}$ produces a CASK manifold of signature $(2,2n)$. In order to relate back to the conventions of Section \ref{review} and \ref{sdualitysection}, corresponding to the QK metric (\ref{QKtree}), we perform  the rescalings (\ref{rescalings}) from Section \ref{c-mapcase}
\begin{equation}\label{res}
    \rho \to \frac{4}{\pi}\rho, \quad \sigma \to \frac{\sigma}{4\pi}\,, \quad \zeta^i \to -\zeta^i
\end{equation}
and write everything in terms of the prepotential $-\mathfrak{F}$ producing a CASK manifold of signature $(2n,2)$ (in particular, we redefine $Z_i=\partial_{Z^i}\mathfrak{F}$ and $\tau_{ij}=\partial_{Z^i}\partial_{Z^j}\mathfrak{F}$ from Section \ref{c-mapcase} by $Z_i:= \partial_{Z^i}(-\mathfrak{F})$ and $\tau_{ij}:=\partial_{Z^i}\partial_{Z^j}(-\mathfrak{F})$), obtaining

\begin{equation}
        \theta_+=-\frac{\pi}{4f}\langle Z\wedge d\zeta \rangle, \quad
        \theta_-=-\frac{\pi}{4f}\langle \overline{Z}\wedge d\zeta \rangle, \quad
        \theta_3=\frac{d\sigma}{16f}+\frac{1}{16 f}\langle \zeta, d\zeta \rangle +\frac{1}{2}\widetilde{\eta}\,.
\end{equation}

As we saw in Section \ref{c-mapcase}, the rescalings (\ref{res}) take the QK metric (\ref{HKQK}) obtained via HK/QK correspondence to (\ref{QKtree}). In particular, after the scaling we have the relation $\rho=\frac{\tau_2^2}{2}h(t)=\frac{\tau_2^2}{16}e^{-\mathcal{K}}$ corresponding to the Mirror map. On the other hand, we now have $\frac{4}{\pi}\rho=2\pi r^2$ which implies  $\rho=\frac{\pi^2}{4} |Z^0|^2e^{-\mathcal{K}}$ (here we have used that  $r^2=|Z^0|^2e^{-\mathcal{\widetilde{K}}}=\frac{|Z^0|^2}{2}e^{-\mathcal{K}}$), so we find the  relation 
\begin{equation}
    \frac{\tau_2^2}{16}e^{-\mathcal{K}}=\frac{\pi^2}{4}|Z^0|^2e^{-\mathcal{K}} \quad \implies \quad  2\pi|Z^0|=\tau_2.
\end{equation}

Recall the definition of the vector $Z=(Z^0,\ldots, Z^n,Z_0,\ldots , Z_n)$, $Z_i = -\partial_{Z^i}\mathfrak{F}(Z^0,\ldots, Z^n)$. Hence in terms of the normalized central charges 
$\widetilde{Z}:=Z/Z^0=(z^i,F_i)$, where $z^i:=Z^i/Z^0$ 
(in particular, $z^0=1$) and $F_i:= F_i(z^0,\ldots, z^n) = F_i(Z^0,\ldots, Z^n) =  -\partial_{Z^i}\mathfrak{F}(Z^0,\ldots , Z^n)/Z^0$ are homogeneous of degree zero; and using that $\tilde{\eta}|_{\hat{N}}=-\frac{1}{2}d^c\mathcal{K}$,  we can write

\begin{equation}
    \begin{split}
        \theta_{+}|_{\hat{N}}&=-\frac{\tau_2}{8f}\langle \widetilde{Z},d\zeta \rangle, \quad \theta_{-}|_{\hat{N}}=\overline{\theta_{+}}|_{\hat{N}}\\
        \theta_3|_{\hat{N}}&=\frac{d\sigma}{16f}+\frac{1}{16 f}\langle \zeta, d\zeta \rangle -\frac{1}{4}d^c\mathcal{K}\,,
    \end{split}
\end{equation}
and hence,

\begin{equation}
    \lambda_0=f\Big(\frac{dt}{t} -\frac{\tau_2}{8ft}\langle \widetilde{Z},d\zeta \rangle -t\frac{\tau_2}{8f}\langle \overline{\widetilde{Z}},d\zeta \rangle -2i\Big(\frac{d\sigma}{16f}+\frac{1}{16 f}\langle \zeta, d\zeta \rangle -\frac{1}{4}d^c\mathcal{K}\Big)\Big)\,.
\end{equation}

The claim, following \cite{NPV}, is that one can obtain Darboux coordinates as follows:

\begin{proposition} \label{DCtreelevelcmap} Let $\lambda \in \Omega^1(\mathcal{Z},\mathcal{L})$ be the holomorphic $1$-form valued in $\mathcal{L}$, defining the holomorphic contact structure of $(\mathcal{Z},\mathcal{I})$. Then on the open subset $\mathcal{Z}_0:=\overline{N}\times (\mathbb{C}P^1-\{0,\infty\})\subset \mathcal{Z}\cong \overline{N}\times \mathbb{C}P^1$, we can write 

\begin{equation}
    \lambda=\frac{1}{8i}(d\alpha +\langle \xi, d\xi \rangle)\cdot s=\frac{1}{8i}(d\alpha +\widetilde{\xi}_id\xi^i-\xi^id\widetilde{\xi}_i)\cdot s\,,
\end{equation}
where
\begin{equation}\label{Darbouxcoords}
    \begin{split}
        \xi^{i}&=\zeta^{i}- \frac{i\tau_2}{2}(t^{-1}z^i +t\overline{z}^i)\\
        \widetilde{\xi}_i&=\widetilde{\zeta}_{i}- \frac{i\tau_2}{2}(t^{-1}F_{i}+t\overline{F}_i)\\
        \alpha&= \sigma -\frac{i\tau_2}{2}(t^{-1}\langle \widetilde{Z},\zeta \rangle +t \langle \overline{\widetilde{Z}}, \zeta \rangle)\\
    \end{split}
\end{equation}
and $s$ is the holomorphic section of $\mathcal{L}\to \mathcal{Z}$ given in Lemma \ref{lemmaholsection}.  In particular,   $(\xi^i,\widetilde{\xi}_i,\alpha)$ are holomorphic coordinates on $(\mathcal{Z},\mathcal{I})$.\\

\end{proposition}
\begin{proof}

We start by computing 

\begin{equation}
    \begin{split}
        \widetilde{\xi}_id\xi^i-\xi^id\widetilde{\xi}_i&=\langle \zeta, d\zeta \rangle +\frac{id\tau_2}{2}(t^{-1}\langle\widetilde{Z},\zeta \rangle +t \langle\overline{\widetilde{Z}}, \zeta \rangle)+\frac{i\tau_2}{2}(-t^{-2}\langle\widetilde{Z},\zeta\rangle dt+\langle \overline{\widetilde{Z}},\zeta \rangle dt) + 8 i f \frac{dt}{t}\\
        &\quad +\frac{i\tau_2}{2}(t^{-1}\langle d\widetilde{Z},\zeta \rangle + t \langle d\overline{\widetilde{Z}},\zeta \rangle) -\frac{i\tau_2}{2}(t^{-1}\langle \widetilde{Z},d\zeta \rangle + t \langle \overline{\widetilde{Z}},d\zeta \rangle) - 4 f d^c \mathcal{K}\,,
    \end{split}
\end{equation}
where for the $8if\frac{dt}{t}$ and $-4fd^c\mathcal{K}$ terms we have used that the CASK relation $F_i=\tau_{ij}z^j$ implies

\begin{equation}
    \begin{split}
    \Big(-\frac{\tau_2^2}{2}F_i\overline{z}^i+\frac{\tau_2^2}{2}\overline{F}_iz^i\Big)\frac{dt}{t}&=    \Big(-\frac{\tau_2^2}{2}\tau_{ij}z^j\overline{z}^i+\frac{\tau_2^2}{2}\overline{\tau}_{ij}\overline{z}^jz^i\Big)\frac{dt}{t}\\
    &=\Big(\frac{i}{2}\tau_2^2N_{ij}z^j\overline{z}^i\Big)\frac{dt}{t}=\Big(\frac{i}{2}\tau_2^2e^{-\mathcal{K}}\Big)\frac{dt}{t}\\
    &=8if\frac{dt}{t}\,.
    \end{split}
\end{equation}
and by using the relation $dF_i=\tau_{ij}dz^i$ 
\begin{equation}
    \begin{split}
    -\frac{\tau_2^2}{4}F_id\overline{z}^i-\frac{\tau_2^2}{4}\overline{F}_idz^i +\frac{\tau_2^2}{4}d\overline{F}_iz^i+\frac{\tau_2^2}{4}dF_i\overline{z}^i&=-\frac{\tau_2^2}{4}\Big(iN_{ij}\overline{z}^jdz^i -iN_{ij}z^id\overline{z}^j\Big)\\
    &=-\frac{\tau_2^2}{4}N_{ij}z^i\overline{z}^jd^c\mathcal{K}\\
    &=-4fd^c\mathcal{K}\,.
    \end{split}
\end{equation}
On the other hand, we find that

\begin{equation}
    \begin{split}
    d\alpha=&d\sigma -\frac{id\tau_2}{2}(t^{-1}\langle \widetilde{Z},\zeta \rangle +t \langle \overline{\widetilde{Z}},\zeta \rangle) + \frac{i\tau_2}{2}(t^{-2}\langle \widetilde{Z},\zeta \rangle - \langle \overline{\widetilde{Z}},\zeta \rangle)dt\\
    &-\frac{i\tau_2}{2}(t^{-1}\langle d\widetilde{Z},\zeta \rangle +t \langle d\overline{\widetilde{Z}},\zeta \rangle)-\frac{i\tau_2}{2}(t^{-1}\langle \widetilde{Z},d\zeta \rangle +t \langle \overline{\widetilde{Z}},d \zeta \rangle),
    \end{split}
\end{equation}
so we conclude that

\begin{equation}
    \begin{split}
         d\alpha +\widetilde{\xi}_id\xi^i-\xi^id\widetilde{\xi}_i&=8if\frac{dt}{t}+d\sigma +\langle \zeta, d\zeta \rangle-i\tau_2(t^{-1}\langle \widetilde{Z},d\zeta \rangle + t \langle \overline{\widetilde{Z}},d\zeta \rangle) -4 f d^c \mathcal{K}\\
         &=8 if\Big(\frac{dt}{t} -\frac{\tau_2}{8ft}\langle \widetilde{Z},d\zeta \rangle -t\frac{\tau_2}{8f}\langle \overline{\widetilde{Z}},d\zeta \rangle -2i\Big(\frac{d\sigma}{16f}+\frac{1}{16 f}\langle \zeta, d\zeta \rangle -\frac{1}{4}d^c\mathcal{K}\Big)\Big)\,.
    \end{split}
\end{equation}

Hence,
\begin{equation}
    \lambda=\frac{1}{8i}(d\alpha +\widetilde{\xi}_id\xi^i-\xi^id\widetilde{\xi}_i)\cdot s\,.
\end{equation}
By Proposition \ref{holcontstructure} and Lemma \ref{lemmaholsection} we know that $\lambda$ is a holomorphic section of $T^*\mathcal{Z}\otimes \mathcal{L}\to \mathcal{Z}$ and $s$ is a holomorphic section of $\mathcal{L}$. Hence,

\begin{equation}
    \lambda_0=\frac{1}{8i}(d\alpha +\widetilde{\xi}_id\xi^i-\xi^id\widetilde{\xi}_i)
\end{equation}
is a holomorphic $1$-form on the twistor space. \\

Now consider the $(0,1)$ vector

\begin{equation}
   \partial_{\widetilde{\xi}_i}+i\mathcal{I}\partial_{\widetilde{\xi}_i}
\end{equation}
and the $(2,0)$ form

\begin{equation}
    d(d\alpha +\widetilde{\xi}_id\xi^i-\xi^id\widetilde{\xi}_i)=2d\widetilde{\xi}_i\wedge d\xi^i\,.
\end{equation}
We then have 
\begin{equation}
    0=d\widetilde{\xi}_i\wedge d\xi^i(\partial_{\widetilde{\xi}_i}+i\mathcal{I}\partial_{\widetilde{\xi}_i})=d\xi^i+i\mathcal{I}^*d\xi^i \implies \mathcal{I}^*d\xi^i=id\xi^i\,.
\end{equation}
so that $\xi^i$ is holomorphic. Similarly, we find that $\widetilde{\xi}_i$ is holomorphic. Finally, since $\xi^i$ and $\widetilde{\xi}_i$ are holomorphic, and $\lambda_0$ is $(1,0)$, we conclude that $\alpha$ is also holomorphic. 
\end{proof}
\subsubsection{Darboux coordinates for the case of 1-loop corrected q-map spaces}\label{darbouxcoordssection1loop}

Once we have figured the Darboux coordinates for the tree-level q-map case, we don't have to do much work to obtain the case with 1-loop corrections. Indeed, in terms of the HK data $(N,g_N,\omega_{\alpha},f,V)$ and $(P,\eta)$ of the HK/QK correspondence from Section \ref{c-mapcase}, the 1-loop parameter $c\in \mathbb{R}$ appears only in $f$ via

\begin{equation}
    f=2\pi r^2 -c\,.
\end{equation}
We can therefore write the local expression for the contact form $\lambda$ in (\ref{localcontform}) as

\begin{equation}
    \lambda=\lambda^{\text{cl}}-c\frac{dt}{t}\cdot s\,,
\end{equation}
where $\lambda^{\text{cl}}$ is the contact form for the tree-level case from before. We therefore obtain the following corollary from Proposition \ref{DCtreelevelcmap}.

\begin{corollary}
 Let $\lambda \in \Omega^1(\mathcal{Z},\mathcal{L})$ be the holomorphic $1$-form valued in $\mathcal{L}$, defining the holomorphic contact structure of $(\mathcal{Z},\mathcal{I})$. Then on the open subset $\mathcal{Z}_0:=\overline{N}\times (\mathbb{C}-\mathbb{R}_{\leq 0})\subset \mathcal{Z}\cong \overline{N}\times \mathbb{C}P^1$, we can write 

\begin{equation}
    \lambda=\frac{1}{8i}(d\alpha +\widetilde{\xi}_id\xi^i-\xi^id\widetilde{\xi}_i)\cdot s\,,
\end{equation}
where
\begin{equation}\label{Darbouxcoords1loop}
    \begin{split}
        \xi^{i}&=\zeta^{i}- \frac{i\tau_2}{2}(t^{-1}z^i +t\overline{z}^i)\\
        \widetilde{\xi}_i&=\widetilde{\zeta}_{i}- \frac{i\tau_2}{2}(t^{-1}F_{i}+t\overline{F}_i)\\
        \alpha&= \sigma -\frac{i\tau_2}{2}(t^{-1}\langle \widetilde{Z},\zeta \rangle +t \langle \overline{\widetilde{Z}}, \zeta \rangle)-8ic\log(t)\,,
    \end{split}
\end{equation}
$\log(t)$ uses the principal branch, and $s$ is the holomorphic section of $\mathcal{L}\to \mathcal{Z}$ given in Lemma \ref{lemmaholsection}.  Furthermore,  $(\xi^i,\widetilde{\xi}_i,\alpha)$ are holomorphic coordinates on $(\mathcal{Z},\mathcal{I})$.\\
\end{corollary}
\begin{remark}\leavevmode
\begin{itemize}
    \item We can of course change the branch cut and branch of the $\log(t)$ and obtain new holomorphic Darboux coordinates.  
    \item By performing the coordinate change of fiber coordinate $t\to -it$, we find that the following are Darboux coordinates for $\lambda$:
    \begin{equation}
        \begin{split}
        \xi^{i}&=\zeta^{i}+ \frac{\tau_2}{2}(t^{-1}z^i -t\overline{z}^i)\\
        \widetilde{\xi}_i&=\widetilde{\zeta}_{i}+ \frac{\tau_2}{2}(t^{-1}F_{i}-t\overline{F}_i)\\
        \alpha&= \sigma +\frac{\tau_2}{2}(t^{-1}\langle \widetilde{Z},\zeta \rangle -t \langle \overline{\widetilde{Z}}, \zeta \rangle)-8ic\log(t)\,,
    \end{split}
\end{equation}
recovering the expressions found in the physics literature \cite[Equation 2.13]{Sduality}.
\end{itemize}

\end{remark}
\end{subsection}
\subsection{Lifting the universal isometries of tree-level q-map spaces to the twistor space}

Let $(\overline{N},g_{\overline{N}})$ be a tree-level q-map space with $\text{dim}_{\mathbb{R}}(\overline{N})=4n+4$ ($n \geq 1)$, and $G$ the $(3n+6)$-dimensional universal group of isometries discussed in Section \ref{Siso}. If $(\mathcal{Z},\mathcal{I},\lambda,\tau)$ denotes the twistor space of $(\overline{N},g_{\overline{N}})$, then it is known that the action of $G$ must lift canonically to an action on $(\mathcal{Z},\mathcal{I},\lambda,\tau)$ preserving the twistor space structure \cite{auttwistor} (i.e. it acts holomorphically, preserves the contact structure, and commutes with $\tau$). In the following, we wish to explicitly describe the lifts of the universal isometries that correspond to either an S-duality transformation, or to a an $L$-action transformation. For the explicit description of the lift of the S-duality, we will follow the presciption given in the physics literature \cite{Sduality}.\\

To show that the explicit descriptions from below match the canonical lifts, we will make use of the following lemma:

\begin{lemma}\label{liftunique} If $f$ is an automorphism of $(\mathcal{Z},\mathcal{I},\lambda,\tau)$ covering the identity map of $\overline{N}$ with respect to the canonical projection $\mathcal{Z} \to \overline{N}$, then $f=\mathrm{Id}_{\mathcal{Z}}$. Hence, if we have a two lifts $f_1$ and $f_2$ of an isometry $\overline{f}$ of $(\overline{N},g_{\overline{N}})$ to $(\mathcal{Z},\mathcal{I},\lambda,\tau)$ that preserve the twistor structure, then $f_1=f_2$.

\end{lemma}
\begin{proof}
Fix $p\in \overline{N}$ and recall that the fibers $\mathcal{Z}_p$ of $\mathcal{Z}\to \overline{N}$ are holomorphic submanifolds of $(\mathcal{Z},\mathcal{I})$ with $\mathcal{Z}_p\cong \mathbb{C}P^1$ \cite[Section 4]{Salamon1982}. Using that $f$ is holomorphic and fiber-preserving to conclude that 
it preserves the sphere $\mathcal{Z}_p$ and acts on it as an element  of $\mathrm{PSL}(2,\mathbb{C})$, which is the 
group of holomorphic automorphisms of $\mathbb{C}P^1$. Since it furthermore commutes 
with the antipodal map $\tau|_{\mathcal{Z}_p}$ we conclude that it acts as an element of $\mathrm{PSU}(2)$.
That element is induced by an element of $SO(3)$ acting on $Q_p\cong \mathbb{R}^3 \supset \mathcal{Z}_p\cong S^2$, where $Q\to \overline{N}$ denotes the quaternionic structure of $\overline{N}$. 
Therefore it has a fixed point $A\in \mathcal{Z}_p$. Using the horizontal lift with respect to the contact distribution determined by $\lambda$, we can connect $A$ to any element 
$B\in \mathcal{Z}$ by a horizontal curve (this follows from the fact that the holonomy
group of any QK manifold contains $\mathrm{Sp}(1)=\mathrm{SU}(2)$). 
That curve is the unique horizontal lift $\tilde{c}$ of a curve $c$ in $\overline{N}$ from $p$ to $q = \pi (B)$
with initial condition $A$. Now since $f$ preserves the contact distribution, it follows that it maps $\tilde{c}$ to another horizontal lift
of the same curve $c$ with the same initial condition, since $f(A)=A$. 
So $f$ acts as identity on $\tilde{c}$ and hence $f(B)=B$. 
\end{proof}

From Lemma \ref{liftunique} it follows that the lifts that will be discussed below coincide with the canonical lifts. We start by discussing the lift of the S-duality action, previously found in \cite{APSV,Sduality}. \\

The Darboux coordinates (\ref{Darbouxcoords}) found before match precisely the Darboux coordinates from \cite{APSV,Sduality}, provided we scale the fiber coordinate by
\begin{equation}
    t\to -it\,.
\end{equation}
Indeed after such a scaling, we get

\begin{equation}
    \begin{split}
        \xi^{i}&=\zeta^{i}+ \frac{\tau_2}{2}(t^{-1}z^i -t\overline{z}^i)\\
        \widetilde{\xi}_i&=\widetilde{\zeta}_{i}+ \frac{\tau_2}{2}(t^{-1}F_{i}-t\overline{F}_i)\\
        \alpha&= \sigma +\frac{\tau_2}{2}(t^{-1}\langle \widetilde{Z},\zeta \rangle -t \langle \overline{\widetilde{Z}}, \zeta \rangle)\,,
    \end{split}
\end{equation}
matching \cite[Equation 2.13]{Sduality} (without the $1$-loop correction term for $\alpha$, since we are on the tree-level case).\\

We can now apply the lift of the S-duality action on the twistor space found in \cite{Sduality}:
\begin{definition} With respect to the (smooth) splitting $\mathcal{Z}\cong \overline{N}\times \mathbb{C}P^1$, the lift of the S-duality action is such that on the $t$-variable of the $\mathbb{C}P^1$-fiber over $(\tau=\tau_1+i\tau_2,b^a,t^a,c^a,c_0,c_a,\psi)$ we have that  $\begin{pmatrix}
a & b\\
c & d\\
\end{pmatrix}\in \mathrm{SL}(2, \mathbb{R})$ acts as

\begin{equation}\label{sdualitylift}
    \begin{pmatrix}
a & b\\
c & d\\
\end{pmatrix}\cdot t:= \frac{c\tau_2 +(c\tau_1+d+|c\tau +d|)t}{c\tau_1 +d + |c\tau+d|-c\tau_2t}\,.
\end{equation}
This defines a global lift of the  S-duality action (see Remark \ref{cayley} below). We furthermore remark that the $\mathrm{SL}(2,\mathbb{R})$ transformations generated by $X_e$ and by $X_h$ (i.e. producing an $\mathrm{SL}(2,\mathbb{R})$ transformation with $c=0$) act trivially on $t$, and hence leave the twistor fiber invariant. In other words, the complexity of the transformation (\ref{sdualitylift}) is only due to the ``hidden" symmetry generated by $X_f$. 
\end{definition}

\begin{remark}\label{cayley}
To check that the lift (\ref{sdualitylift}) actually defines an action of $\mathrm{SL}(2,\mathbb{R})$ it is convenient to follow the suggested Cayley transform of \cite[Equation 3.5 and 3.6]{Sduality} on the $t$-coordinate given by

\begin{equation}
    z=\frac{t+i}{t-i}, \quad t=-i\frac{1+z}{1-z}.
\end{equation}
Under the new fiber coordinate $z$, we have that (\ref{sdualitylift}) is now

\begin{equation}
        \begin{pmatrix}
a & b\\
c & d\\
\end{pmatrix}\cdot z  =\frac{c\overline{\tau}+d}{|c\tau +d|}z\,.
\end{equation}
It is then straightforward to check that (\ref{sdualitylift}) actually defines a lift of the $\mathrm{SL}(2,\mathbb{R})$-action to $\mathcal{Z}$.
\end{remark}

\begin{proposition}\cite{APSV,Sduality}
The lift of S-duality to the twistor space defined by (\ref{sdualitylift}) defines a holomorphic lift of S-duality to the twistor space. Furthermore, it preserves the holomorphic contact distribution and the real structure. In particular, it must coincide with the canonical lift of S-duality to the twistor space.
\end{proposition}
\begin{proof}
On the holomorphic coordinates $(\xi^i,\widetilde{\xi}_i,\alpha)$ on the twistor space the element $\begin{pmatrix}
a & b\\
c & d\\
\end{pmatrix} \in \mathrm{SL}(2,\mathbb{R})$ acts as follows \cite{APSV,Sduality}:

\begin{equation}\label{twistorSduality}
    \begin{split}
        \xi^0&\to \frac{a\xi^0+b}{c\xi^0+d}, \quad \xi^a \to \frac{\xi^a}{c\xi^0+d}, \quad \widetilde{\xi}_a \to \widetilde{\xi}_a +\frac{c}{2(c\xi^0+d)}\kappa_{abc}\xi^b\xi^c\\
        \begin{pmatrix}
        \widetilde{\xi}_0 \\
        \widetilde{\alpha}
        \end{pmatrix}&\to \begin{pmatrix}
        d & -c\\
        -b & a \\
        \end{pmatrix}\begin{pmatrix}
        \widetilde{\xi}_0 \\
        \widetilde{\alpha}
        \end{pmatrix} +\frac{1}{6}\kappa_{abc}\xi^a\xi^b\xi^c\begin{pmatrix}
        c^2/(c\xi^0+d)\\
        -[c^2(a\xi^0+b)+2c]/(c\xi^0+d)^2\\
        \end{pmatrix},
    \end{split}
\end{equation}
where $\widetilde{\alpha}$ is related to the previous coordinate $\alpha$ via $\alpha=-2\widetilde{\alpha} -\xi^i\widetilde{\xi}_i$. In particular, in the $(\xi^i,\widetilde{\xi}_i,\widetilde{\alpha})$ coordinates we have that

\begin{equation}
    \lambda = \frac{1}{8i}(d\alpha +\widetilde{\xi}_id\xi^i-\xi^id\widetilde{\xi}_i)\cdot s=-\frac{1}{4i}(d\widetilde{\alpha}+\xi^id\widetilde{\xi}_i)\cdot s \,.
\end{equation}

To check the transformation rule (\ref{twistorSduality}), it is enough to compute the infinitesimal action of S-duality on $(\xi^i,\widetilde{\xi}_i,\widetilde{\alpha})$ and see that it matches the infinitesimal version of (\ref{twistorSduality}). The above transformation rule shows that the action of $\begin{pmatrix}
a & b\\
c & d\\
\end{pmatrix}$ on $\mathcal{Z}$ is holomorphic away from the divisor $\mathcal{D}\subset \mathcal{Z}$ given by

\begin{equation}
 \mathcal{D}=\{p\in \mathcal{Z} \; |\; s(p)=0\}\cup \{p \in \mathcal{Z} \; | \; c\xi^0+d=0\}\,,   
\end{equation}
where we remark that $\{p\in \mathcal{Z} \; |\; s(p)=0\}=(\{0\}\times \overline{N})\cup (\{\infty\}\times \overline{N})$.\\

Since the action of $\begin{pmatrix}
a & b\\
c & d\\
\end{pmatrix}$ is globally defined and continuous on $\mathcal{Z}$, and the divisor $\mathcal{D}$ is defined by the zero sets of the holomorphic section $s$ and the holomorphic function $c\xi^0+d$, it follows by the Riemann removable singularity theorem that the action is holomorphic on all of $\mathcal{Z}$. Furthermore, under the action of  $\begin{pmatrix}
a & b\\
c & d\\
\end{pmatrix}$ given in (\ref{twistorSduality}) one can check that 

\begin{equation}
    (d\widetilde{\alpha}+\xi^id\widetilde{\xi}_i)\to (d\widetilde{\alpha}+\xi^id\widetilde{\xi}_i)/(c\xi^0 +d)\,.
\end{equation}
so that in particular the contact distribution is preserved away from $\mathcal{D}$. By continuity of $\lambda$ it must then be preserved on all of $\mathcal{Z}$.\\

To check that the lift is compatible with the real structure, it is enough to check that (\ref{sdualitylift}) commutes with the antipodal map $t \to -1/\overline{t}$. Indeed, we have

\begin{equation}
    \begin{pmatrix}
    a & b\\
    c & d\\ 
    \end{pmatrix}\cdot \Big(-\frac{1}{\overline{t}}\Big)=\frac{c\tau_2 \overline{t} -(c\tau_1 +d +|c\tau +d|)}{(c\tau_1 +d + |c\tau+d|)\overline{t} +c\tau_2}=-\overline{\Big[\begin{pmatrix}
    a & b\\
    c & d\\ 
    \end{pmatrix}\cdot t \Big]}^{-1}\,.
\end{equation}

It then follow from Lemma \ref{liftunique} that the lift of S-duality given by (\ref{sdualitylift}) coincides with the canonical lift.
\end{proof}

We now discuss how to holomorphically lift the universal isometries of $(\overline{N},g_{\overline{N}})$ corresponding to the $L$-action, to its twistor space. These turns out to be simpler that the S-duality case from before, and the transformation rules of the holomorphic Darboux coordinates $(\xi^i,\widetilde{\xi}_i,\alpha)$ from Section \ref{darbouxcoordssection} turn out to have a very appealing relation to the transformation rules of the variables $(\zeta^i,\widetilde{\zeta}_i,\sigma)$, as we will see below in Proposition \ref{liftofL}. 

\begin{definition}
We define a lift of the action of $L$ on $\overline{N}$ to $\mathcal{Z}\cong \overline{N}\times \mathbb{C}P^1$ by declaring that it should act trivially on the $\mathbb{C}P^1$ fiber.
\end{definition}

Notice that the previous definition is consistent with the lift of S-duality, since the lift of S-duality is such that the $\mathrm{SL}(2,\mathbb{R})$ transformations contained in $L$ (i.e. those generated by $X_e$ and $X_h$) also leave invariant the twistor fiber.

\begin{proposition}\label{liftofL}
The lift of the $L$-action to $\mathcal{Z}$ is holomorphic, preserves the contact distribution and commutes with the real structure. Furthermore, it acts on the holomorphic Darboux coordinates $(\xi^i,\widetilde{\xi}_i,\alpha)$ with the same transformation rules of $(\zeta^i,\widetilde{\zeta}_i,\sigma)$ under the correspondence $\zeta^i \leftrightarrow \xi^i$, $\widetilde{\zeta}_i \leftrightarrow \widetilde{\xi}_i$, $\sigma \leftrightarrow \alpha$. 
\end{proposition}

\begin{proof}

Let us first focus on $L_2=\mathrm{Iwa}(SU(1,n+2))=\mathbb{R}_{>0}\times \mathbb{R}^{2n+2}\times \mathbb{R}$ and denote an element of $L_2$ by $(r,\eta,\kappa)\in \mathbb{R}_{>0}\times \mathbb{R}^{2n+2}\times \mathbb{R}$, where we have used the short-hand notation $\eta=(\widetilde{\eta}_i,\eta^i)$. We then recall that the action of $(r,\eta,\kappa)$ on a point $(z,\rho,\zeta,\sigma)\in \overline{M}\times \mathbb{R}_{>0}\times \mathbb{R}^{2n+2}\times \mathbb{R}=\overline{N}$ is given by 

\begin{equation}
    (r,\eta,\kappa) \cdot (z,\rho,\zeta,\sigma)= (z,r\rho, \sqrt{r}\zeta + \eta, r\sigma + \sqrt{r}\langle \zeta,\eta \rangle + \kappa)\,.
\end{equation}
In particular, due to the relation $\rho=\tau_2^2h(t)/2$ we find that $\tau_2 \to \sqrt{r}\tau_2$ under the previous transformation.\\

We then find that under the action of $(r,\eta,\kappa)$

\begin{equation}
    \begin{split}
    \xi^i=\zeta^i + \frac{\tau_2}{2}(t^{-1}z^i -t\overline{z}^i)\longrightarrow& \sqrt{r}\zeta^i + \eta^i +\sqrt{r}\frac{\tau_2}{2}(t^{-1}z^i -t\overline{z}^i)=\sqrt{r}{\xi^i}+\eta^i\\
    \widetilde{\xi}_i=\widetilde{\zeta}_i + \frac{\tau_2}{2}(t^{-1}F_i -t\overline{F}_i) \longrightarrow& \sqrt{r}\widetilde{\zeta}_i + \widetilde{\eta}_i + \sqrt{r}\frac{\tau_2}{2}(t^{-1}F_i -t\overline{F}_i)=\sqrt{r}\widetilde{\zeta}_i + \widetilde{\eta}_i\\
    \alpha= \sigma +\frac{\tau_2}{2}(t^{-1}\langle \widetilde{Z},\zeta \rangle -t \langle \overline{\widetilde{Z}}, \zeta \rangle)\longrightarrow & r\Big(\sigma +\frac{\tau_2}{2}(t^{-1}\langle \widetilde{Z},\zeta \rangle -t \langle \overline{\widetilde{Z}}, \zeta \rangle) \Big)+ \sqrt{r}\Big(\langle \zeta+\frac{\tau_2}{2}(t^{-1} \widetilde{Z}-t \overline{\widetilde{Z}}),\eta \rangle \Big) + \kappa 
    \end{split}
\end{equation}
or in more abbreviated notation

\begin{equation}\label{twistorL2}
    (r,\eta,\kappa)\cdot (\xi, \alpha)=(\sqrt{r}\xi + \eta, r\alpha + \sqrt{r}\langle \xi, \eta \rangle + \kappa)\,.
\end{equation}

Notice that (\ref{twistorL2}) matches the action of $(r,\eta,\kappa)$ on $(\zeta,\sigma)$ under the replacement $\zeta^i \leftrightarrow \xi^i$, $\widetilde{\zeta}_i \leftrightarrow \widetilde{\xi}_i$, $\sigma \leftrightarrow \alpha$. Furthermore, (\ref{twistorL2}) show that the lift of $L_2$ is holomorphic away from the divisor:

\begin{equation}
    \mathcal{D}=\{p\in \mathcal{Z} \; |\; s(p)=0\}\,.
\end{equation}
Since the lift extends continuously to a global lift to $\mathcal{Z}$, it must then be holomorphic on all of $\mathcal{Z}$.\\

On the other hand, we have that under the action of $(r,\eta,\kappa)$

\begin{equation}
    \begin{split}
    d\alpha+\langle \xi, d\xi \rangle\to &rd\alpha + \sqrt{r}\langle d\xi, \eta \rangle + r\langle \xi, d\xi \rangle + \sqrt{r}\langle \eta, d\xi \rangle\\
    &=r(d\alpha+\langle \xi, d\xi \rangle)\,,
    \end{split}
\end{equation}
so that the contact distribution is preserved on $\mathcal{Z}-\mathcal{D}$. By continuity of $\lambda$, it must then be preserved on all of $\mathcal{Z}$.\\

We now focus on the action of $L_1=\mathbb{R}_{>0}\ltimes \mathbb{R}^{n}$ generated by the trivial lift of the vector fields $D_1$ and $V^a$. Recall that the $\mathbb{R}_{>0}$-factor acts via 

\begin{equation}
    \lambda \cdot (z^a,\rho,\zeta^0,\zeta^a,\widetilde{\zeta}_0,\widetilde{\zeta}_a,\sigma)=(\lambda z^a, \rho, \lambda^{-3/2}\zeta^0, \lambda^{-1/2}\zeta^a, \lambda^{3/2}\widetilde{\zeta}_0,\lambda^{1/2}\widetilde{\zeta}_a,\sigma)
\end{equation}
so the relation $\rho=\tau_2^2h(t)/2$ implies that $\tau_2 \to \lambda^{-3/2}\tau_2$. One then easily checks that 

\begin{equation}\label{d1twistor}
    \lambda \cdot (\xi^0,\xi^a,\widetilde{\xi}_0,\widetilde{\xi}_a,\alpha)=(\lambda^{-3/2}\xi^0,\lambda^{-1/2}\xi^a, \lambda^{3/2}\widetilde{\xi}_0,\lambda^{1/2}\widetilde{\xi}_a, \alpha)\,,
\end{equation}
matching the previous transformation rule for $(\zeta^i,\widetilde{\zeta}_i,\sigma)$. On the other hand, recall that the $\mathbb{R}^{n}$-factor of $L_1$ acts via
\begin{equation}
    v\cdot \begin{pmatrix}
        b^a\\
        t^a\\
        \rho \\
        \zeta^0 \\
        \zeta^a\\
        \widetilde{\zeta}_0\\
        \widetilde{\zeta}_a\\
        \sigma \\
        \end{pmatrix}=\begin{pmatrix}
        b^a+v^a\\
        t^a\\
        \rho \\
        \zeta^0 \\
        \zeta^a +\zeta^0v^a\\
        \widetilde{\zeta}_0+\frac{1}{6}k_{abc}v^av^bv^c\zeta^0 + \frac{1}{2}k_{abc}v^av^b\zeta^c -\widetilde{\zeta}_av^a\\
        \widetilde{\zeta}_a -\frac{1}{2}\zeta^0k_{abc}v^bv^c - k_{abc}v^b\zeta^c\\
        \sigma \\
        \end{pmatrix}\,.
    \end{equation}
By a straightforward, but slightly tedious computation, one can show that the lifted action acts on $(\xi^i,\widetilde{\xi},\alpha)$ via

\begin{equation}\label{vatwistor}
    v\cdot\begin{pmatrix}
    \xi^0\\
    \xi^a\\
    \widetilde{\xi}_0\\
    \widetilde{\xi}_a\\
    \alpha\\
    \end{pmatrix}=\begin{pmatrix}
    \xi^0\\
    \xi^a+\xi^0v^a\\
    \widetilde{\xi}_0+\frac{1}{6}k_{abc}v^av^bv^c\xi^0 + \frac{1}{2}k_{abc}v^av^b\xi^c -\widetilde{\xi}_av^a\\
    \widetilde{\xi}_a -\frac{1}{2}\xi^0k_{abc}v^bv^c-k_{abc}v^b\xi^c\\
    \alpha\\
    \end{pmatrix}
\end{equation}
matching the previous action on $(\zeta^i,\widetilde{\zeta}_i,\sigma)$.\\

Formulas (\ref{d1twistor}) and (\ref{vatwistor}) show that the lift is holomorphic on $\mathcal{Z}-\mathcal{D}$, and by the same argument as before, we then get that they must be holomorphic on all of $\mathcal{Z}$.\\

Furthermore, one can check that the action of $L_1$ leaves the form $d\alpha + \langle \xi, d\xi \rangle$ invariant, so that the contact distribution is preserved on $\mathcal{Z}-\mathcal{D}$. As before, continuity of $\lambda$ then implies that the contact distribution must be preserved on all of $\mathcal{Z}$.\\

Finally, the fact that the lift of $L$ preserves the real structure follows trivially from the fact that it leaves the twistor fibers invariant. We then conclude that lift of the $L$-action coincides with the canonical lift of the $L$-action to the twistor space.
\end{proof}

\end{section}

\begin{section}{Outlook for the cases with quantum corrections}\label{outlook}

In this final section, we go back to the string theory setting, and summarize which quantum corrections of the tree-level q-map metric of $\mathcal{M}_{\text{HM}}^{\text{IIB}}(X)$ are known or expected to preserve the S-duality $\mathrm{SL}(2,\mathbb{Z})$-action by isometries (see for example \cite{ Sduality}). Along the way, we mention how a similar ``S-duality" action by isometries can be conjectured to hold for certain QK metrics, which are formulated in a setting independent of string theory.\\

The type of quantum corrections that $\mathcal{M}_{\text{HM}}^{\text{IIB}}(X)$ receives is divided into the following types (see the review \cite[Section I.2.3]{HMreview1} and the references therein for more details):

\begin{itemize}
    \item $\alpha'$-corrections:  these modify the form of the prepotential (\ref{holprep}) to
    
    \begin{equation}
        \mathfrak{F}=\mathfrak{F}_{\text{cl}} +\mathfrak{F}_{w.s}
    \end{equation}
    where $\mathfrak{F}_{\text{cl}}$ matches (\ref{holprep}) and $\mathfrak{F}_{w.s}$ is the term containing the effects of perturbative $\alpha'$-corrections and world-sheet-instanton corrections. If $\chi(X)$ is the Euler number of the Calabi-Yau three-fold $X$; $\mathrm{Li}_s(x)$ denotes the polylogarithm functions; $n_{\gamma}^{(0)}\in \mathbb{Z}$ denote the genus-$0$ Gopakumar-Vafa invariant associated to $\gamma \in H_2(X,\mathbb{Z})$; and $H_2^+(X,\mathbb{Z})$ denotes the set of non-zero combinations of the form $q_a\gamma^a$ for $q_a \in \mathbb{Z}_{\geq 0}$ and $\{\gamma^a\}$ a given basis of $H_2(X,\mathbb{Z})$, then   $\mathfrak{F}_{w.s}$ is given by (see \cite[Equation I.28]{HMreview1})
    
    \begin{equation}\label{fws}
        \mathfrak{F}_{w.s.}=\chi(X)\frac{\zeta(3)(Z^0)^2}{2(2\pi i)^3}-\frac{(Z^0)^2}{(2\pi i)^3}\sum_{\gamma=q_a\gamma^a \in H_2^+(X,\mathbb{Z})}n_{\gamma}^{(0)}\mathrm{Li}_3(e^{2\pi iq_aZ^a/Z^0})\,,
    \end{equation} where $\zeta(x)$ denotes the Riemann zeta function, and $(Z^0,Z^a)$ with $a=1,2..., n$ denote the special holomorphic coordinates. The first term in (\ref{fws}) corresponds to the perturbative $\alpha'$-correction, while the last term gives the contribution due to world-sheet instantons. We remark that due to the particular form of the correction $\mathfrak{F}_{w.s}$, the metric is no longer in the image of the q-map.\\
    
    These types of corrections break the continuous $\mathrm{SL}(2,\mathbb{R})$ of isometries that was found for the tree-level q-map metric, in particular the ones corresponding to the generators $X_f$ and $X_h$. Furthermore, it can be shown that the discrete group $\mathrm{SL}(2,\mathbb{Z})\subset \mathrm{SL}(2,\mathbb{R})$ does not survive these corrections either \cite[Section 4.4]{BGHL}, so S-duality is broken when one includes $\alpha'$-corrections.\\
    
    The previous statements can be easily abstracted to a string theory-independent setting via the use of mutually local variations of BPS structures (see the beginning of \cite[Section 3]{CT} for a review of this notion). More precisely, consider a CASK domain specified by $(M,\mathfrak{F}_{\text{cl}})$ where $\text{dim}_{\mathbb{C}}(M)=n+1$ and $\mathfrak{F}_{\text{cl}}$ has the form (\ref{holprep}) (i.e. the CASK manifold associated to a PSK manifold in the image of the r-map). One can define a natural (trivial) rank $2n+2$ lattice $\Gamma \to M$ spanned by $\gamma^i:=\text{Re}(\partial_{Z^i})$ and $\gamma_i:=-\text{Im}(\partial_{Z_i})$, where $Z^i$ are the canonical holomorphic coordinates of $M$ and $Z_i:=\partial{\mathfrak{F}_{\text{cl}}}/\partial Z^i$. Furthermore, a canonical central charge describing the CASK geometry is given by $Z_{\gamma^i}:=Z^i$ and $Z_{\gamma_i}:=Z_i$. To this data we can attach numbers $\Omega(\gamma)\in \mathbb{Z}$ with $\gamma \in \Gamma$ such that the tuple $(M,Z,\Gamma,\Omega)$ satisfies the conditions of a  variation of BPS structures. We will assume that if $\Omega(\gamma)\neq 0$ then $\gamma\in \text{span}_{\mathbb{Z}}\{\gamma^i\}$, where $i=0,1,...,n$, so that the variation of BPS structures is mutually local. Furthermore, letting $\Lambda^+:=\text{Span}_{\mathbb{Z}_{\geq 0}}\{\gamma^a\}-\{0\}$, we assume that the BPS indices have the following structure (compare with \cite[equation (4.5)]{HMreview2})
    
    \begin{equation}\label{varBPS}
        \begin{cases}
            \Omega(q_0\gamma^0)=\chi \quad \\
            \Omega(q_0\gamma^0+q_a\gamma^a)=\Omega(q_a\gamma^a) \quad \text{for $q_a\gamma^a \in \Lambda^+$ or $-q_a\gamma^a \in \Lambda^+$}\\
            \Omega(\gamma)=0 \quad \text{else}.\\
        \end{cases}
    \end{equation}
    This particular structure is important for having the S-duality $\mathrm{SL}(2,\mathbb{Z})$ isometries, when also including mutually local D-instanton corrections (see the next point).\\
    
    One can then define the modified prepotential $ \mathfrak{F}=\mathfrak{F}_{\text{cl}} +\mathfrak{F}_{w.s}$ as follows:
    
    \begin{equation}\label{fws}
        \mathfrak{F}_{w.s.}=-\chi\frac{(Z^0)^2}{(2\pi i)^3}-\frac{(Z^0)^2}{(2\pi i)^3}\sum_{\gamma=q_a\gamma^a \in \Lambda^+}\Omega(\gamma)\mathrm{Li}_3(e^{2\pi iq_aZ^a/Z^0})\,.
    \end{equation}
    
    If $\mathfrak{F}=\mathfrak{F}_{\text{cl}} +\mathfrak{F}_{w.s}$ defines a CASK domain on $M_0\subset M$, then we can consider the QK metric obtained via the tree-level c-map as the metric containing the analog of the $\alpha'$-corrections from before. These kind of metrics are then not expected to have the $\mathrm{SL}(2,\mathbb{R})$ (or $\mathrm{SL}(2,\mathbb{Z})$) of isometries, acting via the S-duality action that was found for the corresponding tree-level q-map metric.

    \item $g_s$-corrections: perturbatively in the string coupling $g_s$, the metric only gets  $1$-loop corrections which corresponds to taking $c=\chi(X)/192\pi$ for the $c$-map construction. Hence, the $\alpha'$ and $1$-loop $g_s$-corrections give a metric within the class of (1-loop corrected) c-map metrics.  On the other hand, the non-perturbative $g_s$-corrections are divided in the so-called D-instanton corrections (which in Type IIB string theory are themselves divided into D(-1), D1, D3, and D5 corrections) and NS5-instanton corrections. These corrections take the metric outside the class of c-map metrics.\\
    
    While the inclusion of the full non-perturbative quantum corrections is not well understood, the inclusion of the D-instanton corrections was described in the physics literature \cite{APSV}, by using the twistor space formulation of QK metrics. However, a mathematical treatment dealing with the issues of domains of definition of the metric and its signature has not been given yet. On the other hand, if one considers only the case with D(-1) and D1 instanton corrections, one lands in the case of the so-called mutually-local instanton corrections, which is better mathematically understood. A mathematical treatment of such mutually-local instanton-corrected QK metrics was given in \cite{CT} (see also \cite{HMmetric} for a treatment from the physics literature using twistor methods).\\
    
    While the $\alpha'$ and $1$-loop $g_s$-corrections break the S-duality isometries, it has been shown in the physics literature that these are restored if one adds: either the  D(-1) and D1 instanton corrections \cite{QMS, Sduality}; or D(-1), D1  and D3 corrections  \cite{Sduality2,D31,D32,D33}. Furthermore, S-duality is also expected to act by isometries when all the non-perturbative corrections are included. In a follow up work, the authors intend to do a mathematical treatment of S-duality for the case of mutually local D-instanton corrections, inspired by the work of \cite{QMS, Sduality}. More precisely, the starting point would be a CASK geometry associated to a mutually local variation of BPS structures of the form given in (\ref{varBPS}), with holomorphic prepotential  $\mathfrak{F}=\mathfrak{F}_{\text{cl}}+\mathfrak{F}_{w.s}$. By applying the construction of \cite{CT}, one then obtains an ``instanton corrected" QK metric. This instanton corrected QK metric would be the analog of the metric of $\mathcal{M}_{\text{HM}}^{\text{IIB}}(X)$ with $\alpha'$-corrections, $1$-loop, and D(-1), D1 instanton corrections, and the expectation is that such a metric carries an S-duality $\mathrm{SL}(2,\mathbb{Z})$-action by isometries.\\

\end{itemize}

\end{section}
\clearpage
\appendix
\begin{section}{Computation of the infinitesimal S-duality action in type IIA variables}\label{appA}

In this appendix we do the required computation of Proposition \ref{infS}. Namely we wish to show that 

\begin{equation}
    \begin{split}
        X_e&=\partial_{\zeta^0}+\widetilde{\zeta_0}\partial_{\sigma}\\
     X_f&=(2\rho h(t)^{-1}-(\zeta^0)^2)\partial_{\zeta^0} + (2\rho h(t)^{-1}b^a-\zeta^0\zeta^a)\partial_{\zeta^a}+ \frac{k_{abc}}{2}(\zeta^b\zeta^c-2b^bb^c\rho h(t)^{-1})\partial_{\widetilde{\zeta}_a}\\
        &+\Big(\frac{1}{2}(\sigma + \zeta^i\widetilde{\zeta}_i)+\frac{k_{abc}}{3}\rho h(t)^{-1}b^ab^bb^c \Big)\partial_{\widetilde{\zeta}_0}+\zeta^0t^a\partial_{t^a} +(\zeta^0b^a-\zeta^a)\partial_{b^a} -\zeta^0\rho\partial_{\rho}\\
        &+\Big[2\rho h(t)^{-1}\Big(-\widetilde{\zeta}_0 -b^a\widetilde{\zeta}_a - k_{abc}\Big(\frac{b^ab^b\zeta^c}{2} - \frac{b^ab^bb^c\zeta^0}{6}\Big)\Big) -\frac{\zeta^0}{2}(\sigma-\widetilde{\zeta}_i\zeta^i) +\frac{k_{abc}}{6}\zeta^a\zeta^b\zeta^c\Big]\partial_{\sigma}\\
        X_h&=2\zeta^0\partial_{\zeta^0} +\zeta^a\partial_{\zeta^a}-\widetilde{\zeta}_0\partial_{\widetilde{\zeta}_0} + \sigma \partial_{\sigma}+\rho \partial_{\rho}-t^a\partial_{t^a}-b^a\partial_{b^a}\\
    \end{split}
\end{equation}
using the relation

\begin{equation}
    d\mathcal{M}(Y_e|_{\mathcal{M}^{-1}(p)})=X_e|_p, \quad d\mathcal{M}(Y_f|_{\mathcal{M}^{-1}(p)})=X_f|_p, \quad d\mathcal{M}(Y_h|_{\mathcal{M}^{-1}(p)})=X_h|_p\,,
\end{equation}
together with the formula for the infinitesimal mirror map:
\begin{equation}\label{infMMappendix}
        \begin{split}
        d\mathcal{M}(\partial_{\tau_2})&=\frac{\tau_2e^{-\mathcal{K}}}{8}\partial_{\rho}, \quad\quad          d\mathcal{M}(\partial_{t^a})=-\rho\partial_{t^a}\mathcal{K}\partial_{\rho}+\partial_{t^a}\\
        d\mathcal{M}(\partial_{\tau_1})&=\partial_{\zeta^0}+b^a\partial_{\zeta^a}-\frac{k_{abc}}{2}b^bb^c\partial_{\widetilde{\zeta}_a} +\frac{k_{abc}}{6}b^ab^bb^c\partial_{\widetilde{\zeta}_0} + \Big(-c_0 -c_ab^a+\frac{k_{abc}}{6}b^ab^bc^c\Big)\partial_{\sigma}\\
        d\mathcal{M}(\partial_{b^a})&=\partial_{b^a}+\tau_1\partial_{\zeta^a}+k_{abc}\Big(\frac{1}{2}c^c-\tau_1b^c\Big)\partial_{\widetilde{\zeta}_b}+k_{abc}\Big(-\frac{1}{3}b^bc^c+\frac{1}{2}b^bb^c\tau_1\Big)\partial_{\widetilde{\zeta}_0}\\
        &\;\;\;\;+\Big(-c_a\tau_1 + k_{abc}\Big(\frac{\tau_1b^bc^c}{3}-\frac{c^bc^c}{6}\Big)\Big)\partial_{\sigma}\\
        d\mathcal{M}(\partial_{c^a})&=-\partial_{\zeta^a}+\frac{k_{abc}}{2}b^c\partial_{\widetilde{\zeta}_b} -\frac{k_{abc}}{6}b^bb^c\partial_{\widetilde{\zeta}_0} +\Big(c_a -\frac{k_{abc}}{3}b^bc^c + \frac{k_{abc}}{6}b^bb^c\tau_1\Big)\partial_{\sigma}\\
        d\mathcal{M}(\partial_{c_a})&=\partial_{\widetilde{\zeta}_a}+c^a\partial_{\sigma}, \quad\quad 
        d\mathcal{M}(\partial_{c_0})=\partial_{\widetilde{\zeta}_0} -\tau_1\partial_{\sigma}, \quad\quad 
        d\mathcal{M}(\partial_{\psi})=-2\partial_{\sigma}\,,
    \end{split}
\end{equation}
and the formula for $Y_e$, $Y_f$ and $Y_h$ obtained in Lemma \ref{infsdualityIIB}:

\begin{equation}\label{infsdualityiibexpappendix}
    \begin{split}
        Y_e&=\partial_{\tau_1}+b^a\partial_{c^a}-c_0\partial_{\psi}\\
        Y_f&=(\tau_2^2-\tau_1^2)\partial_{\tau_1}-2\tau_1\tau_2\partial_{\tau_2}+\tau_1t^a\partial_{t^a}+c^a\partial_{b^a}-\psi \partial_{c_0}\\
        Y_h&=2\tau_1\partial_{\tau_1}+2\tau_2\partial_{\tau_2}-t^a\partial_{t^a}-b^a\partial_{b^a}+c^a\partial_{c^a}-c_0\partial_{c_0}+\psi\partial_{\psi}
    \end{split}
\end{equation}\,.

We start computing $X_e$:

\begin{equation}
    \begin{split}
    X_e&=\partial_{\zeta^0}+b^a\partial_{\zeta^a}-\frac{k_{abc}}{2}b^bb^c\partial_{\widetilde{\zeta}_a} +\frac{k_{abc}}{6}b^ab^bb^c\partial_{\widetilde{\zeta}_0} + \Big(-c_0 -c_ab^a+\frac{k_{abc}}{6}b^ab^bc^c\Big)\partial_{\sigma}\\
    &\;\;+b^a\Big(-\partial_{\zeta^a}+\frac{k_{abc}}{2}b^c\partial_{\widetilde{\zeta}_b} -\frac{k_{abc}}{6}b^bb^c\partial_{\widetilde{\zeta}_0} +\Big(c_a -\frac{k_{abc}}{3}b^bc^c +\frac{k_{abc}}{6}b^bb^c\tau_1\Big)\partial_{\sigma}\Big)+2c_0\partial_{\sigma}\\
    &=\partial_{\zeta^0}+\Big(c_0 -\frac{k_{abc}}{6}b^ab^bc^c+\frac{k_{abc}}{6}b^ab^bb^c\tau_1\Big)\partial_{\sigma}\\
    &=\partial_{\zeta^0}+\widetilde{\zeta_0}\partial_{\sigma}
    \end{split}
\end{equation}
where in the last equality we have used the formula for the mirror map (\ref{MM}). \\

The next easiest to compute is $X_h$, where we obtain:
\begin{equation}
\begin{split}
    X_h=&2\tau_1\Big(\partial_{\zeta^0}+b^a\partial_{\zeta^a}-\frac{k_{abc}}{2}b^bb^c\partial_{\widetilde{\zeta}_a} +\frac{k_{abc}}{6}b^ab^bb^c\partial_{\widetilde{\zeta}_0} + \Big(-c_0 -c_ab^a+\frac{k_{abc}}{6}b^ab^bc^c\Big)\partial_{\sigma}\Big)\\
    &+4\rho \partial_{\rho}-t^a\Big(-\rho\partial_{t^a}\mathcal{K}\partial_{\rho}+\partial_{t^a}\Big)-c_0\Big(\partial_{\widetilde{\zeta}_0} -\tau_1\partial_{\sigma}\Big)-2\psi\partial_{\sigma}\\
    &+c^a\Big(-\partial_{\zeta^a}+\frac{k_{abc}}{2}b^c\partial_{\widetilde{\zeta}_b} -\frac{k_{abc}}{6}b^bb^c\partial_{\widetilde{\zeta}_0} +\Big(c_a -\frac{k_{abc}}{3}b^bc^c + \frac{k_{abc}}{6}b^bb^c\tau_1\Big)\partial_{\sigma}\Big)\\
    &-b^a\Big[\partial_{b^a}+\tau_1\partial_{\zeta^a}+k_{abc}\Big(\frac{1}{2}c^c-\tau_1b^c\Big)\partial_{\widetilde{\zeta}_b}+k_{abc}\Big(-\frac{1}{3}b^bc^c+\frac{1}{2}b^bb^c\tau_1\Big)\partial_{\widetilde{\zeta}_0}\\
    &\;\;\;\;+\Big(-c_a\tau_1 + k_{abc}\Big(\frac{\tau_1b^bc^c}{3}-\frac{c^bc^c}{6}\Big)\Big)\partial_{\sigma}\Big]\\
    =&2\zeta^0\partial_{\zeta^0} +\zeta^a\partial_{\zeta^a}-\widetilde{\zeta}_0\partial_{\widetilde{\zeta}_0} + \sigma \partial_{\sigma}+\rho \partial_{\rho}-t^a\partial_{t^a}-b^a\partial_{b^a}
\end{split}
\end{equation}
where as before, in the last equality we have grouped the components together and used the mirror map \ref{MM}. In the first equality, we have furthermore used that $\rho=\tau_2^2h(t)/2=\tau_2^2e^{-\mathcal{K}}/16$ for the $\partial_{\rho}$-component.\\

Finally, the expression for $X_f$ gives:
\begin{equation}
    \begin{split}
        X_f=&(\tau_2^2-\tau_1^2)\Big(\partial_{\zeta^0}+b^a\partial_{\zeta^a}-\frac{k_{abc}}{2}b^bb^c\partial_{\widetilde{\zeta}_a} +\frac{k_{abc}}{6}b^ab^bb^c\partial_{\widetilde{\zeta}_0} + \Big(-c_0 -c_ab^a+\frac{k_{abc}}{6}b^ab^bc^c\Big)\partial_{\sigma}\Big)\\
        &-2\tau_1\tau_2\frac{\tau_2e^{-\mathcal{K}}}{8}\partial_{\rho}+\tau_1t^a\Big(-\rho\partial_{t^a}\mathcal{K}\partial_{\rho}+\partial_{t^a}\Big)-\psi\Big(\partial_{\widetilde{\zeta}_0} -\tau_1\partial_{\sigma}\Big)\\
        &+c^a\Big[\partial_{b^a}+\tau_1\partial_{\zeta^a}+k_{abc}\Big(\frac{1}{2}c^c-\tau_1b^c\Big)\partial_{\widetilde{\zeta}_b}+k_{abc}\Big(-\frac{1}{3}b^bc^c+\frac{1}{2}b^bb^c\tau_1\Big)\partial_{\widetilde{\zeta}_0}\\
        &\;\;\;\;+\Big(-c_a\tau_1 + k_{abc}\Big(\frac{\tau_1b^bc^c}{3}-\frac{c^bc^c}{6}\Big)\Big)\partial_{\sigma}\Big]\\
        =&(2\rho h(t)^{-1}-(\zeta^0)^2)\partial_{\zeta^0} + (2\rho h(t)^{-1}b^a-\zeta^0\zeta^a)\partial_{\zeta^a}+ \frac{k_{abc}}{2}(\zeta^b\zeta^c-2b^bb^c\rho h(t)^{-1})\partial_{\widetilde{\zeta}_a}\\
        &+\Big(\frac{1}{2}(\sigma + \zeta^i\widetilde{\zeta}_i)+\frac{k_{abc}}{3}\rho h(t)^{-1}b^ab^bb^c \Big)\partial_{\widetilde{\zeta}_0}+\zeta^0t^a\partial_{t^a} +(\zeta^0b^a-\zeta^a)\partial_{b^a} -\zeta^0\rho\partial_{\rho}\\
        &+\Big[2\rho h(t)^{-1}\Big(-\widetilde{\zeta}_0 -b^a\widetilde{\zeta}_a - k_{abc}\Big(\frac{b^ab^b\zeta^c}{2} - \frac{b^ab^bb^c\zeta^0}{6}\Big)\Big) -\frac{\zeta^0}{2}(\sigma-\widetilde{\zeta}_i\zeta^i) +\frac{k_{abc}}{6}\zeta^a\zeta^b\zeta^c\Big]\partial_{\sigma}\\
    \end{split}
\end{equation}
where in the last equality we have used again the Mirror map (\ref{MM}) and organized the terms by components.

\end{section}

\bibliography{References}
\bibliographystyle{alpha}
\end{document}